\documentclass{article}
\usepackage[english]{babel}
\usepackage{graphicx}
\usepackage{slashed}
\usepackage{stmaryrd}
\usepackage{amsmath,amsfonts,amssymb}

\usepackage{ntheorem}
\usepackage{dsfont}
\usepackage{verbatim}
\usepackage{float}
\usepackage{accents}
\usepackage{mathrsfs}
\usepackage{amsmath}
\usepackage{amssymb}
\usepackage{hyperref}
\usepackage{tikz}
\usepackage{pgfplots}
\pgfplotsset{compat=1.14}
\usetikzlibrary{patterns}
\usetikzlibrary{positioning,arrows,arrows.meta}
\usepackage{geometry}
\usepackage{enumitem}

\theoremseparator{.} 
\newtheorem{Th}{Theorem}[section]
\newtheorem{Def}[Th]{Definition}
\newtheorem{Rq}[Th]{Remark}
\newtheorem{Pro}[Th]{Proposition}
\newtheorem{Cor}[Th]{Corollary}
\newtheorem{Lem}[Th]{Lemma}
\theoremstyle{empty}
\newtheorem{refproof}{Proof}

\newcommand{\R}{\mathbb{R}}
\newcommand{\T}{\mathbf{T}}
\newcommand{\M}{\mathcal{M}}

\newcommand{\dr}{\mathrm{d}}

\newcommand{\N}{\mathcal{N}}
\newcommand{\n}{n^{\mathstrut}}
\newcommand{\uu}{\underline{u}}
\newcommand{\m}{\mu^{\mathstrut}}
\newcommand{\Pm}{\mathcal{P}}
\newcommand{\Rm}{\mathcal{R}}

\newcommand{\C}{\mathcal{P}}

\newenvironment{proof}{\noindent\textit{Proof.~}}{\hfill$\square$\bigbreak} 

\title{Decay estimates for the massless Vlasov equation \\ on Schwarzschild spacetimes}

\author{L\'eo Bigorgne\footnote{Department of Pure Mathematics and Mathematical Statistics, University of Cambridge, Cambridge.
{\em E-mail address:} {\tt lb847@cam.ac.uk}}}
\date{}
\begin{document}

\maketitle
    
\begin{abstract}
We consider solutions to the massless Vlasov equation on the domain of outer communications of the Schwarschild black hole. By adapting the $r^p$-weighted energy method of Dafermos and Rodnianski, used extensively in order to study wave equations, we prove superpolynomial decay for a non-degenerate energy flux of the Vlasov field $f$ through a well-chosen foliation. An essential step of this methodology consists in proving a non-degenerate integrated local energy decay. For this, we take in particular advantage of the red-shift effect near the event horizon. The trapping at the photon sphere requires however to lose an $\epsilon$ of integrability in the velocity variable. Pointwise decay estimates on the velocity average of $f$ are then obtained by functional inequalities, adapted to the study of Vlasov fields, which allow us to deal with the lack of a conservation law for the radial derivative.
\end{abstract}

    \tableofcontents

\section{Introduction and preliminaries}

The main goal of this article is to prove pointwise decay estimates for the velocity average of the solutions to the massless Vlasov equation in the exterior of a Schwarzschild spacetime $(\mathcal{M},g)$ of mass $M>0$. This region is covered by a coordinate system $(t,r,\theta, \varphi)$, where $(t,r) \in \R \times ]2M,+\infty[$ and $(\theta , \varphi)$ are spherical coordinates on the unit sphere $\mathbb{S}^2$, and the metric can be written
$$ g \hspace{1mm} = \hspace{1mm} - \left( 1-\frac{2M}{r} \right) \dr t^2+ \left( 1-\frac{2M}{r} \right)^{-1} \dr r^2+r^2 \dr \sigma^{\mathstrut}_{\mathbb{S}^2}, \qquad \dr \sigma^{\mathstrut}_{\mathbb{S}^2}= \dr \theta^2+\sin^2( \theta) \dr \varphi^2.$$
For the purpose of this paper, it will be more convenient to work with Regge-Wheeler coordinates $(t,r^*,\theta , \varphi )$, where
\begin{equation}\label{defrstar}
 r^* \hspace{1mm} := \hspace{1mm} r+2M \log (r-2M)-3M-2M \log M, \qquad r^* \in \R,
 \end{equation}
is the Regge-Wheeler tortoise coordinate and vanishes at the photon sphere $r=3M$, which contains trapped null geodesics. In this coordinate system, the metric takes the form
\begin{equation}\label{defgS}
g \hspace{1mm} = \hspace{1mm} - \left( 1-\frac{2M}{r} \right) \dr t^2+ \left( 1-\frac{2M}{r} \right) \dr r^{*2}+r^2 \dr \sigma^{\mathstrut}_{\mathbb{S}^2}.
\end{equation}
The Vlasov equation on a Schwarzschild background models the evolution of particles of mass $m \geq 0$ which do not self-interact. They are represented by the particle density $f$, which is a nonnegative function defined on a subset $\mathcal{P}$ of the cotangent bundle $T^{\star} \mathcal{M}$, sometimes abusively referred as the co-mass shell. In the exterior region $r>2M$ and in the coordinate system $(t,r^*,\theta , \varphi ,v_{r^*},v_{\theta}, v_{\varphi})$ of the co-mass shell induced by the Regge-Wheeler coordinates, the Vlasov equation for massless particles $m=0$ reads
\begin{equation}\label{Vlasovintro}
 -\frac{v_t}{1-\frac{2M}{r}} \partial_tf + \frac{v_{r^*}}{1-\frac{2M}{r}} \partial_{r^*}f+\frac{v_{\theta}}{r^2} \partial_{\theta}f + \frac{v_{\varphi}}{r^2 \sin^2 ( \theta )} \partial_{\varphi}f + \frac{r-3M}{r^4} |\slashed{v}|^2 \partial_{v_{r^*} }f+\frac{\cot (\theta) }{r^2 \sin^2 (\theta)}  v_{\varphi}^2 \partial_{v_{\theta} }f  \,  = \, 0 ,
 \end{equation}
where
\begin{equation*}
v_t \, = \, - \bigg| \,|v_{r^*}|^2 +\left( 1-\frac{2M}{r} \right)\! \frac{|\slashed{v}|^2}{r^2} \, \bigg|^{\frac{1}{2}}, \qquad \quad \slashed{v} := \sqrt{|v_{\theta}|^2+\frac{ |v_{\varphi}|^2}{\sin^2 (\theta )}},
\end{equation*}
and reflects that the particles are moving along null geodesics. In this article, we study solutions to massless Vlasov equation on the exterior region of Schwarzschild spacetime arising from sufficiently regular data prescribed on, say, the Cauchy hypersurface $\{t^* := t+2M\log (r-2M) =0 \}$. In particular, we will prove decay estimates for the energy flux of $f$ through a well-chosen foliation and for its velocity average. For instance, we will prove that any sufficiently regular solution $f$ to \eqref{Vlasovintro} satisfies, for any $R^* \in \R$ and $p\geq 0$,
$$ \forall \, (t,r^*,\omega ) \in \R \times \R \times \mathbb{S}^2, \quad r^* \geq R^*, \quad t^* \geq 0, \qquad  \int_{\C } |f| |v_t|^2 \dr \m_{\C} \Big\vert_{(t,r^*,\omega)} \, \leq \, \frac{C_{f,R^*,p}}{r^2(1+|t-r^*|)^p},$$
where the constant $C_{f,R^*,p}$ only depends on $M$, $R^*$, $p$ and a certain energy norm of $f(t^*=0)$. Before presenting our main results (see Theorem \ref{theorem}), we introduce the notations used all along this paper.

\subsection{The exterior of Schwarzschild spacetime}\label{subsec1}

The maximally extended Schwarzschild spacetime $(\M,g)$ is a time-oriented Lorentzian manifold solution to the vacuum Einstein equations $R_{\mu \nu}=0$.

\begin{figure}[H] 
\begin{center}
\begin{tikzpicture}[scale=0.4]
\node (I)    at ( 4,0)   {$\mathscr{D}$};
\node (I')   at (-4,0)   {$\mathscr{D}'$};
\node (II)  at (0, 2.5) {$\mathscr{B}$};
\node (II')   at (0,-2.5) {$\mathscr{B}'$};

\path  
  (I') +(90:4)  coordinate  (IItop)
       +(-90:4) coordinate (IIbot)
       +(0:4)   coordinate                  (IIright)
       +(180:4) coordinate (IIleft)
       ;
\draw (IIleft) -- (IItop) --(IIright) --(IIbot) --(IIleft) -- cycle;

\path 
   (I) +(90:4)  coordinate[label=90:$i^+$]  (Itop)
       +(-90:4) coordinate[label=-90:$i^-$]  (Ibot)
       +(180:4) coordinate (Ileft)
       +(0:4)   coordinate[label=0:$i^0$]  (Iright)
       ;
\draw (Ileft) -- 
          node[midway, above=1pt ]    {$\cal{H}^+$}          
      (Itop) --
          node[midway, above right] {$\cal{I}^+$}
      (Iright) -- 
          node[midway, below right] {$\cal{I}^-$}
      (Ibot) --          
          node[midway, below=1pt ]    {$\cal{H}^{\!-}$}    
      (Ileft) -- cycle;

\draw[decorate,decoration=zigzag] (IItop) -- (Itop)
      node[midway, above, inner sep=2mm] {$r=0$};

\draw[decorate,decoration=zigzag] (IIbot) -- (Ibot)
      node[midway, below, inner sep=2mm] {$r=0$};

\end{tikzpicture}
\end{center}
\caption{The Penrose diagram $\M/SO_3(\R)$ of the Schwarzschild spacetime.}\label{Penrose1}
\end{figure}
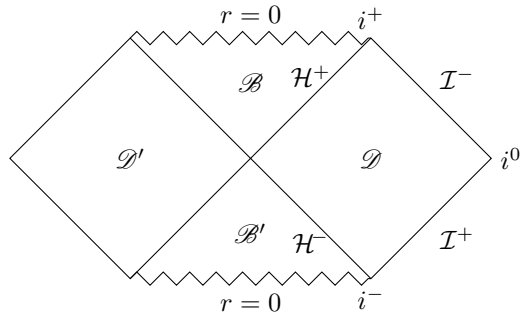

The region which we are interested in and represented by the coordinate system $(t,r>2M,\theta , \varphi)$, or alternatively by $(t,r^*,\theta , \varphi)$, is $\mathscr{D}$. The black hole $\mathscr{B}$, covered by the coordinate system $(t,0<r<2M,\theta , \varphi)$, is separated from the domain of outer communications by the event horizon $\mathcal{H}^+$, which is the null hypersurface $r=2M$. The region $\mathscr{D}'$ is a copy of $\mathscr{D}$ and $\mathscr{B}'$ is a white hole. Note that the time orientation can be chosen so that $\partial_t$ is future oriented for $r >2M$.

In this article, we will only work on the exterior region $\mathscr{D}$. It will be convenient to use the retarded and advanced Eddington-Finkelstein coordinates 
$$u := t-r^*, \qquad \underline{u} =t+r^*,$$
which turn out to be null since the metric takes the form
\begin{equation}\label{metricuuu}
g \hspace{1mm} = \hspace{1mm} -\left( 1-\frac{2M}{r} \right) \dr u \dr \underline{u} +r^2 \dr \sigma^{\mathstrut}_{ \mathbb{S}^2}.
\end{equation}
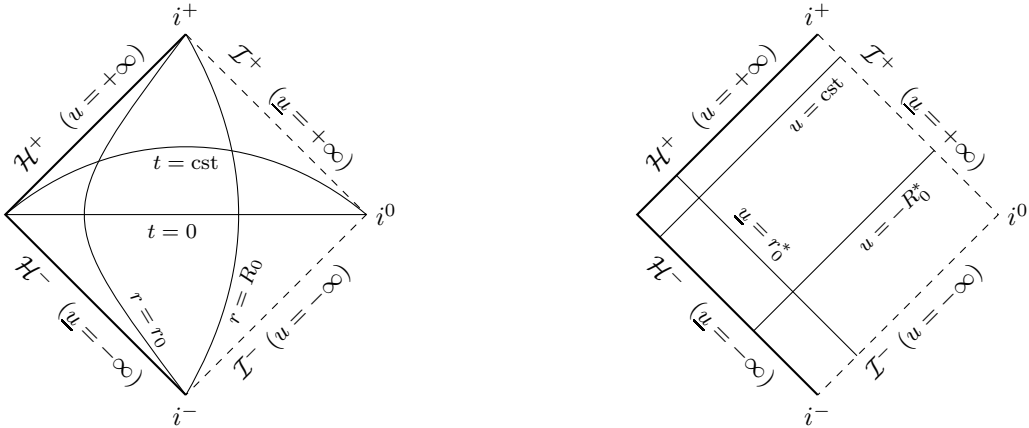
\begin{figure}[H] 
\begin{center}
\begin{tikzpicture}[scale=0.6]
\node (II)   at (-4,0)   {};
\path  
  (II) +(90:4)  coordinate[label=90:$i^+$]  (IItop)
       +(-90:4) coordinate[label=-90:$i^-$] (IIbot)
       +(180:4) coordinate (IIleft)
       +(0:4)   coordinate[label=0:$i^0$]  (IIright)
       ;
\draw[thick] (IIbot) --
          node[midway, below, sloped] {$\cal{H}^-$ \, \small{($\underline{u}=-\infty$)}}    
      (IIleft) --
        node[midway, above, sloped] {$\cal{H}^+$ \, \small{($u=+\infty$)}}
      (IItop) ;
\draw[dashed]  (IItop)--    
          node[midway, above, sloped] {$\cal{I}^+$ \, \small{($\underline{u}=+\infty$)}}
      (IIright) -- 
          node[midway, below, sloped] {$\cal{I}^-$  \small{($u=- \infty$)}}
      (IIbot) ;
\draw (-4,4) to [bend left] node[pos=0.8,below right=-3pt,sloped]{\footnotesize{$r=R_0$}}  (-4,-4);
\draw (-4,4) .. controls (-7,0)  .. node[near end, above right=0.2pt,sloped]{\footnotesize{ $r=r_0$}} (-4,-4) ;
\draw (-8,0) to [bend left=40] node[below]{\footnotesize{$t=\mathrm{cst}$}} (0,0);
\draw (-8,0) to node[below]{\footnotesize{$t=0$} \; \;} (0,0);

\node (I)   at (10,0)   {};
\path  
  (I) +(90:4)  coordinate[label=90:$i^+$]  (Itop)
       +(-90:4) coordinate[label=-90:$i^-$] (Ibot)
       +(180:4) coordinate (Ileft)
       +(0:4)   coordinate[label=0:$i^0$]  (Iright)
       ;
\draw[thick] (Ibot) --
          node[midway, below, sloped] {$\cal{H}^-$ \, \small{($\underline{u}=-\infty$)}}    
      (Ileft) --
        node[midway, above, sloped] {$\cal{H}^+$ \, \small{($u=+\infty$)}}
      (Itop) ;
\draw[dashed]  (Itop)--    
          node[midway, above, sloped] {$\cal{I}^+$ \, \small{($\underline{u}=+\infty$)}}
      (Iright) -- 
          node[midway, below, sloped] {$\cal{I}^-$  \small{($u=- \infty$)}}
      (Ibot) ;
\draw (12.58,1.42) to node[pos=0.3,below,sloped]{\footnotesize{$u=-R_0^*$}} (8.58,-2.58) ;
\draw (10.5,3.5) to node[pos=0.2,below,sloped]{\footnotesize{$u=\mathrm{cst}$}} (6.5,-0.5) ;
\draw (6.87,0.87) to node[ pos=0.4, above,sloped]{\footnotesize{$\uu=r^*_0$}}(10.87,-3.13) ;
\end{tikzpicture}
\caption{The level sets of the coordinates $t$, $r$, $u$ and $\uu$. }\label{fig20}
\end{center}
\end{figure}

\begin{Rq}\label{globcoord}
The Penrose diagram of Figure \ref{Penrose1} can be covered by global coordinates $(U,\underline{U})$. Apart from the usual degeneration of the spherical variables, $(U,\underline{U}, \theta ,\varphi)$ defines a global system of coordinates on $(\M,g)$, where, in the exterior region $r>2M$,
\begin{equation*}
 U = \arctan \left( - \exp \left(-\frac{u}{4M} \right) \right), \qquad \underline{U} = \arctan \left(  \exp \left(\frac{\underline{u}}{4M} \right) \right)
 \end{equation*}
and the future event horizon $\mathcal{H}^+$ corresponds to $U=0$ (or abusively to $u=+\infty$).
\end{Rq}

It is well-known, and it can be easily observed in \eqref{defgS}, that the exterior of Schwarzschild spacetime is static and spherically symmetric. More precisely, $\partial_t$ is a timelike Killing vector field and $(\Omega_1,\Omega_2,\Omega_3)$, where
\begin{equation}\label{defomega}
 \Omega_1 \hspace{1mm} := \hspace{1mm} - \sin \varphi \, \partial_{\theta}-\cot \theta \cos \varphi \, \partial_{\varphi} , \qquad \Omega_2 \hspace{1mm} := \hspace{1mm}  \cos \varphi \, \partial_{\theta}-\cot \theta \sin \varphi \, \partial_{\varphi} , \qquad \Omega_3 \hspace{1mm} := \hspace{1mm} \partial_{\varphi} ,
 \end{equation}
is a basis of Killing vector fields generating the $SO_3(\R)$-symmetry. 
\begin{Rq}\label{exprdtheta}
Note that $\partial_{\theta} = - \sin (\varphi) \Omega_1+\cos (\varphi) \Omega_2$ and $\frac{1}{\sin (\theta)} \partial_{\varphi} = -\frac{\cos (\varphi)}{\cos (\theta)} \Omega_1-\frac{\sin (\varphi)}{\cos (\theta ) } \Omega_2$.
\end{Rq}
Finally, because of the degeneracy of the spherical coordinates $(\theta,\varphi)$, we will sometimes need to use a different set of spherical variables $(\widetilde{\theta},\widetilde{\varphi})$, that we choose for simplicity so that
$$ \sin (\theta ) \cos (\varphi ) \, = \, \cos (\widetilde{\theta} ), \qquad  \sin (\theta ) \sin (\varphi ) \, = \, \sin (\widetilde{\theta} ) \cos (\widetilde{\varphi} ), \qquad \cos (\theta ) \, = \, \sin (\widetilde{\theta} ) \sin (\widetilde{\varphi} ).$$
More precisely, $(\widetilde{\theta} , \widetilde{\varphi})$ is obtained from $(\theta, \varphi)$ by applying the permutation $x\textrm{-axis} \rightarrow y\textrm{-axis}$, $y\textrm{-axis} \rightarrow z\textrm{-axis}$, $z\textrm{-axis} \rightarrow x\textrm{-axis}$ to the axis of the sphere. In particular, we have $\dr \m_{\mathbb{S}^2} =\dr \widetilde{\theta}^2 +\sin^2 (\widetilde{\theta} ) \dr \widetilde{\varphi}^2 $,
\begin{equation}\label{othersphecoord}
 \partial_{\widetilde{\theta}} \, = \,   - \sin (\widetilde{\varphi}) \Omega_2+\cos (\widetilde{\varphi}) \Omega_3 , \qquad \partial_{\widetilde{\varphi}} \, = \, \Omega_1 \qquad \qquad \text{and} \qquad \qquad \sin (\theta ) \leq \frac{1}{\sqrt{2}} \Leftrightarrow \sin (\widetilde{\theta}) \geq \frac{1}{\sqrt{2}}.
 \end{equation}
In order to completely cover $\mathbb{S}^2$ by these two spherical coordinate systems, we take $(\widetilde{\theta}, \widetilde{\varphi} ) \in ]0, \pi[ \times ]-\pi , \pi [$.

We now construct the foliation used in this article in order to study solutions to the massless Vlasov equation. Consider $\accentset{\circ}{\mathcal{S}}$ a spherically symmetric spacelike Cauchy hypersurface in $(\mathcal{M},g)$ crossing the event horizon to the future of the sphere of bifurcation $\mathcal{H}^+ \cap \mathcal{H}^-$ and terminating at spatial infinity $i^0$ or at future null infinity $\mathcal{I}^+$. Further, we define $\mathcal{S} := \accentset{\circ}{\mathcal{S}} \cap \mathscr{D}$ and we will study solutions to the massless Vlasov equation, arising from initial data given on $\mathcal{S}$, in $J^+(\mathcal{S} ) \cap \mathscr{D}$, where $J^+( \mathcal{S})$ is the causal future of $\mathcal{S}$. For this, we will use the foliation $(\Sigma_{\tau})_{\tau \in \R}$ defined as follows. We fix, for all this paper, a constant $R_0 >3M$ and we consider $t_0 \in \R$ such that $\{(t_0,R_0) \} \times \mathbb{S}^2 \cap \mathcal{S} \neq \emptyset$. We introduce $R_0^*:=r^*(R_0)$, $u_0:=t_0-R_0^*$ and, for any $\tau \in \R$,
$$ \mathcal{N}_{\tau} := \{ (t,r^*,\omega) \in \R \times \R \times \mathbb{S}^2 \, / \, t-r^* = \tau+ u_0 \, , \, r^* \geq R_0^* \} \cap J^+(\mathcal{S}) ,$$
where $\mathcal{N}_{\tau}$ is the piece of the outgoing null hypersurface $u=\tau+u_0$ located in the region $\{ r^* \geq R_0^* \}$ of the causal future of $\mathcal{S}$. Then, we define
\begin{itemize}
\item $\Sigma_0 := \left( \mathcal{S} \cap \{ r^* < R_0^* \} \right) \cup \mathcal{N}_0$.
\item $\Sigma_{\tau} := \varphi_{\tau} \left( \Sigma_0 \right)$ for all $\tau \geq 0$, where $\tau \mapsto \varphi_{\tau}$ is the flow generated by the static Killing vector field $\partial_t$. 
\item $\Sigma_{\tau} = \mathcal{N}_{\tau}$ for all $\tau <0$.
\end{itemize}
\begin{Rq}
For all $\tau \geq 0$, we have $\Sigma_{\tau}=\left(\varphi_{\tau} \left( \mathcal{S} \right) \cap \{ r^* < R_0^* \} \right) \cup \mathcal{N}_{\tau}$. 
\end{Rq}
For $- \infty \leq a < b \leq +\infty$, we also introduce
$$ \Rm_a^b \, := \, \bigsqcup_{a \leq \tau \leq b} \Sigma_{\tau}, \qquad \text{so that } \qquad J^+(\mathcal{S}) \cap \mathscr{D} \, = \, \Rm_{-\infty}^{+\infty} \, = \, \bigsqcup_{\tau \in \R} \Sigma_{\tau}.$$
These subsets are represented in the following Penrose diagram of the exterior of Schwarzschild spacetime.
\begin{figure}[H]
\begin{center}
\begin{tikzpicture}[scale=1]
\fill[color=gray!15] (-7.13,0.87) .. controls  (-6.13,-0.13) and   (-2.84,0).. (0,0)--(-4,4);

\node (II)   at (-4,0)   {};
\path  
  (II) +(90:4)  coordinate[label=90:$i^+$]  (IItop)
       +(180:4) coordinate (IIleft)
       +(0:4)   coordinate[label=0:$i^0$]  (IIright)
       ;
\draw[thick] (IIleft) -- 
        node[midway, above=1pt, sloped] {$\cal{H}^+$ \, \small{($u=+\infty$)}}
      (IItop);
 \draw[dashed]     (IItop)--    node[midway, above=1pt, sloped] {$\cal{I}^+$ \, \small{($\underline{u}=+\infty$)}}     (IIright);

\fill[color=gray!35] (-6.37,1.63) .. controls (-5.87,1.13)  and (-4,1.7)..  (-3,1.4)--(-2.2,2.2) -- (-2.88,2.88)-- (-3.29,2.47) .. controls (-4,3.1) and (-5.29,2.11)  ..  (-5.59,2.41)  ;
\draw[dashed] (0,0)--(-1,-1); 
\draw[thick] (-8,0)--(-7,-1); 
\draw[very thin,blue]  (-7.13,0.87) .. controls  (-6.13,-0.13) and (-2.84,0).. node[below]{\small{$\mathcal{S}$}}   (0,0)  ;
\draw (-6.37,1.63) .. controls (-5.87,1.13)  and (-4,1.7).. node[below]{\small{$\Sigma_{\tau}$}}  (-3,1.4)--(-2.2,2.2) node[pos=0.8,below=3pt]{\footnotesize{$\N_{\tau}$}};
\draw (-5.59,2.41) .. controls (-5.29,2.11)   and (-4,3.1) .. node[above]{\small{$\Sigma_{s}$}} (-3.29,2.47)--(-2.88,2.88) ;
\draw (-1.41,1.41) -- (-2.82,0) node[pos=0.5, right=1pt]{\footnotesize{$\N_{0}$}} ;
\draw (-0.8,0.8) -- (-1.6,0) node[pos=0.6, right=1pt]{\footnotesize{$\N_{\tau'}$}} ;

\draw (-4,2) node{\small{$\Rm_{\tau}^s$}};
\end{tikzpicture}
\caption{The foliation $(\Sigma_{\tau})_{\tau \in \R}$. In this diagram, $\tau' < 0 < \tau < s$. }\label{fig2}
\end{center}
\end{figure}
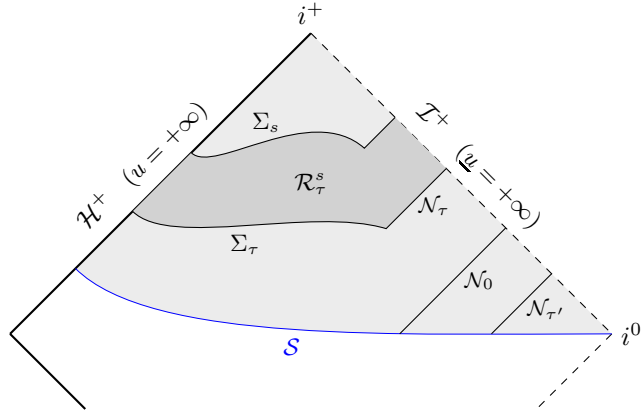

We now collect various formulas and properties for future reference.
\begin{alignat}{2}\label{defderiv}
\partial_{r^*} \, & = \, \left(1 - \frac{2M}{r} \right) \partial_r, \qquad && \partial_{\uu} \,  = \, \frac{1}{2} \left( \partial_t + \partial_{r^*} \right), \qquad \partial_u \, = \, \frac{1}{2} \left( \partial_t - \partial_{r^*} \right), \\
\dr r^* \, & = \, \frac{\dr r}{1 - \frac{2M}{r}} , \qquad  && \dr \uu \,  = \, \dr t + \dr r^* , \qquad \; \; \dr u \, = \, \dr t - \dr r^*. \label{defdiff}
\end{alignat}
We denote by $\dr \mu_{\Rm_a^b}$ the invariant volume form induced by $g$ on $\Rm_a^b$ and by $\dr \m_{\mathbb{S}^2} = \sin (\theta ) \dr \theta \wedge \dr \varphi$ the standard volume element of the unit sphere $\mathbb{S}^2$. Let $\n_{\Sigma_{\tau}}$ be the future oriented normal vector along $\Sigma_{\tau}$ such that $\n_{\Sigma_{\tau}} \vert_{\N_{\tau}} = \partial_{\uu}$ and $\n_{\Sigma_{\tau}}$ is unitary for $r < R_0$. Then, we denote by $\dr \m_{\Sigma_{\tau}}$ the induced volume form on the hypersurface\footnote{This means that $\dr \mu_{\Rm_{-\infty}^{+\infty}}=-g\big(\n_{\Sigma_{\tau}}, \cdot \big) \wedge \dr \m_{\Sigma_{\tau}}$.} $\Sigma_{\tau}$. The following results are proved in Appendix \ref{ape}.
\begin{Lem}\label{vol}
For any $- \infty \leq a < b \leq +\infty$,
$$ \dr \m_{\Rm_a^b} \, = \,  \left( 1 - \frac{2M}{r} \right)r^2  \dr t \wedge \dr r^* \wedge \dr \m_{\mathbb{S}^2} \, = \, r^2 \dr t \wedge \dr r \wedge \dr \m_{\mathbb{S}^2} .$$
There exists a strictly positive smooth function $\gamma : [2M,R_0] \rightarrow \R_+^*$ such that, for all $\tau \in \R$,
$$  \dr \m_{\Sigma_{\tau}}\Big\vert^{\mathstrut}_{\{2M < r < R_0 \}} \, = \,  \gamma (r) r^2 \dr r \wedge \dr \m_{\mathbb{S}^2}, \qquad \qquad \qquad \dr \m_{\Sigma_{\tau}}\Big\vert^{\mathstrut}_{\{ r \geq R_0 \}} \, = \, \dr \m_{\N_{\tau}} \, := \, r^2 \dr \uu \wedge \dr \m_{\mathbb{S}^2}.$$
There exist a strictly positive smooth function $\beta :[2M,R_0] \rightarrow \R_+^*$ such that
$$ \n_{\Sigma_{\tau}} \Big\vert^{\mathstrut}_{\{2M < r < R_0 \}} \, = \, \frac{1}{\beta(r)} \partial_{\uu} + \frac{\beta(r)}{1-\frac{2M}{r}} \partial_u, \qquad \qquad \qquad \n_{\Sigma_{\tau}} \Big\vert^{\mathstrut}_{\N_{\tau}} \, = \, \n_{\N_{\tau}} \, = \, \partial_{\uu} .$$
There exists a smooth function $\underline{U} : [2M, R_0 ] \rightarrow \R$ such that, for any $\tau \geq 0$, $\Sigma_{\tau} \cap \{ r < R_0 \}$ can be parameterized, in the Regge-Wheeler system of coordinates $(t,r^*,\theta, \varphi)$, by
$$ (r,\omega ) \in \, ]2M,R_0[ \times \mathbb{S}^2 \mapsto ( \underline{U}(r)-r^*+\tau , r^*, \omega),$$
where $r^*=r^*(r)$ is defined in \eqref{defrstar}. This implies in particular
$$ \forall \,  r^* < R_0^*,  \quad  |\tau - \underline{u} | \leq \| \underline{U} \|^{\mathstrut}_{L^{\infty}}, \qquad \forall \, r^* \geq R_0^*, \quad |\tau - u | \leq |u_0|$$
and that $\tau \sim t$ on the domains of bounded $r^*$. Moreover, there exists $\gamma_0 :[2M,+\infty[ \rightarrow \R_+^*$ such that
$$ \dr \m_{ \Rm_{-\infty}^{+\infty}}  \, = \, \gamma_0(r) \dr \tau \wedge \dr \m_{\Sigma_{\tau}}, \qquad \qquad \exists \, C \geq 1, \; \forall \, r \geq 2M, \qquad \frac{1}{C} \leq \gamma_0(r) \leq C.$$
\end{Lem}

\begin{Rq}
A possible and explicit choice for $\mathcal{S}$ is the hypersurface $\{ t^* =0 \}$, where $t^* := t+2M \log (r-2M)$. In that case, for all $\tau \geq 0$, we have $\Sigma_{\tau} = \left( \{ t^*=\tau \} \cap \{ r^* < R_0^* \} \right) \cup \mathcal{N}_{\tau}$, 
$$\dr \m_{\Sigma_{\tau}} \Big\vert^{\mathstrut}_{\{ r < R_0 \}} \, = \, \gamma(r) r^2 \dr r \wedge \dr \m_{\mathbb{S}^2}, \qquad \n_{\Sigma_{\tau}}\Big\vert^{\mathstrut}_{\{ r < R_0 \}} \, = \, \frac{1}{ \gamma(r)}\partial_{\uu} + \frac{\gamma(r)}{ \left(1-\frac{2M}{r}\right)} \partial_u, \qquad \gamma(r) \, := \, \left| 1+\frac{2M}{r} \right|^{\frac{1}{2}}.$$
\end{Rq}
We will sometimes need to consider pieces of the hypersurfaces of constant $\uu$. More precisely, let
$$\underline{\N}_{\underline{w}} \, := \, \{ (t,r^*,\omega) \in \R \times \R \times \mathbb{S}^2 \, / \, t+r^* = \underline{w} \} \cap J^+ (\mathcal{S}), \qquad \underline{w} \in \R,$$
and $\n_{\underline{\N}_{\underline{w}}} = \partial_u$ be a future oriented normal vector along $\underline{\N}_{\underline{w}}$. Then, the induced volume form is given by $\dr \m_{\underline{\N}_{\underline{w}}} = -r^2 \dr u \wedge \dr \m_{\mathbb{S}^2}$. We will also denote by $\n_{\mathcal{S}}$ the future oriented normal unit vector to $\mathcal{S}$ and by $\dr \m_{\mathcal{S}}$ the induced volume form on $\mathcal{S}$. 

Finally, in order to lighten the notations, we will throughout this article use the notation $A \lesssim B$ when there exists $C >0$ such that $A \leq C \cdot B$, with $C$ a constant depending only on $M$, the foliation $(\Sigma_{\tau})_{\tau \in \R}$ and the constants $2M <r_0 < r_1 < 3M$ introduced in Section \ref{sec2}. If $C$ depends also on a parameter $p$, we will write $A \lesssim^{\mathstrut}_p B$.
\subsection{Vlasov fields in the cotangent formulation}

Our presentation follows the one of \cite[Subsection $2.1$]{FJS3} in the special case of the exterior of Schwarzschild spacetime. For an introduction to the (co)tangent bundle formulation of Vlasov equations, we refer to \cite{SarbachSchw,Sarbach}. In the cotangent formulation, massless Vlasov fields are defined on the bundle of future light cones
\begin{equation}
 \mathcal{P} \hspace{1mm} := \hspace{1mm}  \displaystyle{\bigcup_{x \in \mathcal{M}}} \mathcal{P}_x, \qquad \mathcal{P}_x \hspace{1mm} := \hspace{1mm} \left\{(x,v) \, / \, v \in T^{\star}_{x}\mathcal{M} , \; \text{$g^{-1}_x(v,v)=0$ and $v$ future oriented}  \right\}.
 \end{equation}
 \begin{Rq}
Note that a $1$-form $v \in T^{\star}_x\M$ is said to be future oriented if and only if the vector field $g^{-1}_x(v, \cdot)$ is future oriented. In particular, if $v \in \mathcal{P}_x$, then $v \neq 0$.
\end{Rq}
In the exterior region $r>2M$, we can decompose any $1$-form $v$ as $v=v_t \dr t+v_{r^*} \dr r^*+v_{\theta} \dr \theta + v_{\varphi} \dr \varphi$, so that $(t,r^*,\theta , \varphi , v_t,v_{r^*},v_{\theta}, v_{\varphi})$ is a coordinate system on $T^{\star} \mathcal{M}$ which is induced by the Regge-Wheeler coordinates. If $g^{-1}(v,v)=0$ and $v$ is future oriented, we have 
\begin{equation}\label{defv0}
v_t \, = \, - \bigg| \,|v_{r^*}|^2 +\left( 1-\frac{2M}{r} \right)\! \frac{|\slashed{v}|^2}{r^2} \, \bigg|^{\frac{1}{2}}, \qquad \quad \slashed{v} := \sqrt{|v_{\theta}|^2+\frac{ |v_{\varphi}|^2}{\sin^2 (\theta )}},
\end{equation}
and $(t,r^*,\theta , \varphi , v_{r^*},v_{\theta}, v_{\varphi})$ are smooth coordinates on $\C$. In this coordinate system, denoted for convenience by $(x^0,x^1,x^2,x^3,v^1,v^2,v^3)$, the massless Vlasov equation takes the form\footnote{This formula holds in a more general setting (see \cite{FJS3}).}
$$\T (f) \hspace{1mm} := \hspace{1mm} v_{\alpha} g^{\alpha \beta} \partial_{x^{\beta}}f-\frac{1}{2} \partial_{x^i} g^{\alpha \beta} v_{\alpha} v_{\beta} \partial_{v_i} f \hspace{1mm} = \hspace{1mm} 0,$$
where $\T$ is the Liouville vector field. Using \eqref{defgS}, $\partial_{r^*}=\left(1-\frac{2M}{r} \right) \partial_r$ and \eqref{defv0}, so that $\frac{1}{2} \partial_{r^*} g^{\alpha \beta} v_{\alpha} v_{\beta} =\frac{r-3M}{r^4}|\slashed{v}|^2$, we obtain
\begin{equation}\label{defT}
\T \, = \, -\frac{v_t}{1-\frac{2M}{r}} \partial_t + \frac{v_{r^*}}{1-\frac{2M}{r}} \partial_{r^*}+\frac{v_{\theta}}{r^2} \partial_{\theta} + \frac{v_{\varphi}}{r^2 \sin^2 ( \theta )} \partial_{\varphi} + \frac{r-3M}{r^4} |\slashed{v}|^2 \partial_{v_{r^*} }+\frac{\cot (\theta) }{r^2 \sin^2 (\theta)}  v_{\varphi}^2 \partial_{v_{\theta} }.
\end{equation} 
Note that $v$ can also be decomposed on the basis $(\dr u, \dr \underline{u} , \dr \theta, \dr \varphi )$, so that
\begin{equation}\label{vu}
\hspace{-3mm} v = v_u \dr u +v_{\underline{u}} \dr \underline{u} +v_{\theta} \dr \theta +v_{\varphi} \dr v_{\varphi}, \qquad v_u = \frac{v_t-v_{r^*}}{2}, \quad v_{\underline{u}} = \frac{v_t+v_{r^*}}{2}, \quad 4|v_{\uu}||v_u|=\left( \! 1-\frac{2M}{r} \!\right) \! \frac{|\slashed{v}|^2}{r^2}.
\end{equation}
The isometries of the Schwarzschild spacetime generates symmetries for the Liouville vector field $\T$ as well as quantities preserved along its flow. For any (conformal) Killing vector field $X$, its complete lift $\widehat{X} \in T T^{\star} \M$ is tangent to $\C$ and commute with $\T$. Moreover, $v(X)$ is solution to the massless Vlasov equation. These results, which hold in a more general setting (see \cite[Proposition $1$]{SarbachSchw} and \cite[Appendix $C$]{FJS}), implies that
\begin{align}
\widehat{\Omega}_1 \hspace{1mm} & := \hspace{1mm} - \sin \varphi \, \partial_{\theta}-\cot \theta \cos \varphi \, \partial_{\varphi} - \cos \varphi \frac{v_{\varphi}}{\sin^2 \theta} \partial_{v_{\theta}}+\left( \cos \varphi \, v_{\theta}-\sin \varphi \cot \theta \, v_{\varphi} \right)\partial_{v_{\varphi}}  , \label{def1} \\
 \widehat{\Omega}_2 \hspace{1mm} & := \hspace{1mm}  \cos \varphi \, \partial_{\theta}-\cot \theta \sin \varphi \, \partial_{\varphi} - \sin \varphi \frac{v_{\varphi}}{\sin^2 \theta} \partial_{v_{\theta}}+\left( \sin \varphi \, v_{\theta}+\cos \varphi \cot \theta \, v_{\varphi} \right)\partial_{v_{\varphi}} , \label{def2} \\
\widehat{\Omega}_3 \hspace{1mm} & : = \hspace{1mm} \partial_{\varphi}, \qquad  \qquad  \widehat{\partial}_t \hspace{1mm}  := \hspace{1mm} \partial_t \label{def3} 
\end{align}
commute with $\T$ and that $v_t$ as well as $\slashed{v} = \left| v(\Omega_1)^1+v(\Omega_2)^2+v(\Omega_3)^2 \right|^{\frac{1}{2}}$ are preserved by the flow of $\T$. We summerize these properties, which can also be obtained by straightforward computations.
\begin{Lem}\label{Comuprop}
We have
$$ [\T , \partial_t] =0, \qquad [\T, \widehat{\Omega}_1]=0, \qquad [\T, \widehat{\Omega}_2]=0, \qquad [\T, \widehat{\Omega}_3]=0$$
as well as $\T(v_t)=0$ and, almost everywhere, $\T(\slashed{v})=0$.
\end{Lem}
\begin{Rq}
Note also that the vector field $S_v:=v_{r^*}\partial_{v^*}+v_{\theta} \partial_{v_{\theta}}+ v_{\varphi} \partial_{v_{\varphi}}$ satisfies $[\T, S_v ]=-\T$. We will not take advantage of it in this article but let us mention that it was crucially used in \cite[Lemma $6.1$]{rVP}.
\end{Rq}
If $I \in \llbracket 1,3 \rrbracket^{n}$ is a multi-index of length $|I|=n \geq 1$, we define $\widehat{\Omega}^I := \widehat{\Omega}_{I_1}\dots \widehat{\Omega}_{I_n}$. For convenience, we define the multi-index of length $|I|=0$ such that for any function $f$, $\widehat{\Omega}^I f=f$.

For any $- \infty \leq a < b \leq +\infty$, we define the following subset of $\C$,
\begin{align*}
 \widehat{\Rm}_a^b \,  := \, \{ (x,v) \in \C \, / \, x \in \Rm_a^b \}.
\end{align*}
Let $x$ be a point in the exterior of Schwarzschild spacetime. The inverve metric $g^{-1}$ induces the volume element $\dr \mu_{T^{\star}_x \M} = \sqrt{|\det g^{-1} |} \dr v_t \wedge \dr v_{r^*} \wedge \dr v_{\theta} \wedge \dr v_{\varphi}$ on $T^{\star}_x \M$. Since $\mathcal{P}_x$ is included in a level set of $q :v \mapsto \frac{1}{2}g^{-1}(v,v)$, the differential form $\dr q=-\frac{v_t}{1-\frac{2M}{r}} \dr v_t+\frac{v_{r^*}}{1-\frac{2M}{r}} \dr v_{r^*}+\frac{v_{\theta}}{r^2} \dr v_{\theta}+\frac{v_{\varphi}}{r^2 \sin^2 \theta} \dr v_{\varphi} $ is normal to $\C_x$. Hence,
$$ \dr \m_{\C_x}\, := \, \frac{\sqrt{|\det g^{-1} |}}{-\frac{v_t}{1-\frac{2M}{r}} } \dr v_{r^*} \wedge \dr v_{\theta} \wedge \dr v_{\varphi} \hspace{1mm} = \hspace{1mm} \frac{\dr v_{r^*} \wedge \dr v_{\theta} \wedge \dr v_{\varphi}}{r^2 |\sin \theta | |v_t|}  $$
is the unique volume form on $\C_x$ satisfying $\dr \mu_{T^{\star}_x \M} = \dr q \wedge \dr \mu_{\C_x}$. For any sufficiently regular function $f : \widehat{\Rm}_{-\infty}^{+\infty} \rightarrow \R$, we denote by $\int_{\C} f \dr \m_{\C}$ the function $x \mapsto \int_{\C_x} f \dr \m_{\C_x}$. Then, its energy momentum tensor $\mathbb{T}$ and particle current density $N$ are defined by
$$\mathbb{T}[f]_{\mu \nu} \, := \,  \int_{\C} f v_{\mu } v_{\nu} \dr \m_{\C} , \qquad N [f]_{\mu} \, := \, \int_{\C} f v_{\mu }  \dr \m_{\C} $$ 
and there holds
\begin{equation}\label{divN}
\nabla^{\mu} \mathbb{T}[f]_{\mu \nu} \, = \, \int_{\C} \T(f) v_{\nu} \dr \m_{\C}, \qquad \nabla^{\mu} N[f]_{\mu} \, = \, \int_{\C} \T(f) \dr \m_{\C},
\end{equation}
so that $N[f]$ is divergence free if $f$ is solution to the massless Vlasov equation. If $X=X^{\nu} \partial_{\nu}$ is a vector field, we denote by $v(X):=v_{\nu} X^{\nu}$ its image under $v_\mu \mathrm{d}x^\mu$. If $X$ is the red-shift vector field $\mathrm{N}$ introduced below or $\n_{\Sigma_{\tau}}$, we will rather write $v_{\mathrm{N}}:=v(\mathrm{N})$ and $  v \cdot \n_{\Sigma_{\tau}}:=v ( \n_{\Sigma_{\tau}} )$. Then, we have
$$ \nabla^{\mu} \left( \mathbb{T}[f]_{\mu \nu} X^{\nu} \right) \, = \, \nabla^{\mu} N \left[ f \, v (X) \right]_{\mu} \, = \,   \int_{\C} \big( \T(f) v(X)+ f \T \left( v(X) \right) \big) \dr \m_{\C}.$$
Thus, using the vector field $X$ as a multiplier corresponds to multiply the particle density $f$ by the weight $v(X)=v_{\nu} X^{\nu}$. In particular, if $X$ is the static Killing field $\partial_t$ and if $\T(f)=0$, one can derive conservation laws for $|v(\partial_t)|^a f=|v_t|^a f$, $a \in \R$. Next, we recall the dominant energy condition. If $f \geq 0$ and $X_1, \, X_2$ are two future directed causal vectors, then
\begin{equation*}
 \mathbb{T}[f](X_1,X_2) \, = \, \int_{\C} f \, v(X_1) v(X_2) \dr \m_{\C} \, \geq \, 0 .
\end{equation*}
This follows from the fact that $v$ is lightlike and future directed, so that $v(X_1) \leq 0$, with a strict inequality if $X_1$ is timelike. Finally, it will be convenient to use
\begin{equation}\label{defrho}
 \rho \big[ f \big] \, := \, \int_{\C} f \, |v \cdot \n_{\Sigma_{\tau}}| \, \dr \m_{\C},
\end{equation}
so that, if $X$ is a timelike future oriented vector field, $\rho [f|v(X)| ] = \mathbb{T}[f](X, \n_{\Sigma_{\tau}})$ is an energy density. In this paper, we will mainly be interested in\footnote{In the context of wave equations, this quantity is often denoted by $J^{\mathrm{N}}[f] \cdot \n_{\Sigma_{\tau}}$, where $J^{\mathrm{N}}[f]_{\mu} := \mathbb{T}[f]_{\mu \nu} \mathrm{N}^{\nu}$.} $\rho [f|v_{\mathrm{N}}^{\mathstrut}|]$.

\subsection{Statement of the main result}

In order to present the main result of this paper, we will use the notations $\lceil p \rceil := \min \{ n \in \mathbb{N} \, / \, n \geq p \}$, $u_- := \max(0,-u)$ and $\langle x \rangle := (1+|x|^2)^{\frac{1}{2}}$. Moreover, we will refer to an energy norm $\mathcal{E}^{p}_s[f]$ which is defined in Section \ref{sec7} (see \eqref{mathcalE}). It is the sum of weighted $L^{q_1}$ and $L^{q_2}$ norms of $\partial_t^n \widehat{\Omega}^I (f)_{\vert \mathcal{S}}$, with $n+|I| \leq 3$ and $3 < q_1 <q_2$.

\begin{Th}\label{theorem}
Let $p \in \R_+$, $s>1$ and $f : \widehat{\Rm}_{-\infty}^{+\infty} \rightarrow \R$ be a sufficiently regular solution to the massless Vlasov equation $\T(f)=0$. The following estimates holds.
\begin{itemize}
\item Integrated local energy decay. For all $\tau \geq 0$,
\begin{align*}
 \int_{0}^{\tau}& \! \int_{\Sigma_{\tau'}} \! \mathds{1}^{\mathstrut}_{r \leq R_0}  \int_{\C} |f| |v^{\mathstrut}_{\mathrm{N}}| |v \cdot \n_{\Sigma_{\tau'}} | \dr \m_{\C} \dr \m_{\Sigma_{\tau'}} \dr \tau' \\ & \qquad \quad  \lesssim^{\mathstrut}_s  \int_{\mathcal{S}}  \int_{\C}  |f|  |v^{\mathstrut}_{\mathrm{N}}| |v \cdot \n_{\mathcal{S}} | \dr \m_{\C} \dr \m_{\mathcal{S}}  +  \! \left| \int_{\mathcal{S}} \int_{\C} \left\langle r^2 \frac{|v_{\uu}|}{|v_t|} \right\rangle^{s-1} \langle v_t \rangle^{5(s -1)} |f|^s |v^{\mathstrut}_{\mathrm{N}}| |v \cdot \n_{\mathcal{S}} | \dr \m_{\C} \dr \m_{\mathcal{S}} \right|^{\frac{1}{s}}\!  . 
 \end{align*}
\item Boundedness of a $r$-weighted energy norm. For all $\tau \in \R$,
$$  \int_{\Sigma_{\tau}} \int_{\C}\bigg(1+ r^{p} \frac{| v_{\uu}|^{p/2}}{| v_t|^{p/2}}\bigg)|f| |v^{\mathstrut}_{\mathrm{N}}| |v \cdot \n_{\Sigma_{\tau}} | \dr \m_{\C} \dr \m_{\Sigma_{\tau}} \, \lesssim^{\mathstrut}_p \, \int_{\mathcal{S}} \int_{\C} \bigg(1+ r^{p} \frac{|v_{\uu}|^{p/2}}{|v_t|^{p/2}} \bigg)  |f| |v^{\mathstrut}_{\mathrm{N}}| |v \cdot \n_{\mathcal{S}} | \dr \m_{\C} \dr \m_{\mathcal{S}}.$$
\item Decay estimate for the energy flux. Assume further that $s^2 \leq 1+\langle 5p \rangle^{-2}$. Then, for all $\tau \in \R$,
\begin{align*}
 \int_{\Sigma_{\tau}}  \int_{\C} |f| |v^{\mathstrut}_{\mathrm{N}}| |v \cdot \n_{\Sigma_{\tau}} | \dr \m_{\C} \dr \m_{\Sigma_{\tau}} \,  \lesssim^{\mathstrut}_{p,s}  \frac{1}{(1+|\tau|)^p}\! \int_{\mathcal{S}} \int_{\C} \bigg( r^{p} \frac{|v_{\uu}|^{p/2}}{|v_t|^{p/2}}+(1+u_-)^p \bigg)  |f| |v^{\mathstrut}_{\mathrm{N}}| |v \cdot \n_{\mathcal{S}} | \dr \m_{\C} \dr \m_{\mathcal{S}} \quad  & \\
  + \frac{1}{(1+|\tau|)^p} \left| \int_{\mathcal{S}}  \int_{\C} \bigg( r^{p} \frac{|v_{\uu}|^{p/2}}{|v_t|^{p/2}}+(1+u_-)^p \bigg)  |f|^{s^{\lceil p \rceil}}\langle v_t\rangle^{5(s^{\lceil p \rceil}-1)}  |v^{\mathstrut}_{\mathrm{N}}| |v \cdot \n_{\mathcal{S}} | \dr \m_{\C} \dr \m_{\mathcal{S}} \right|^{ s^{- \lceil p \rceil}} \! \hspace{-1mm}. &
 \end{align*}
 \item Pointwise decay estimates for the velocity average of $f$. If $s^2 \leq 1+\langle 20(p+2)\rangle^{-2}$ and $\mathcal{E}^p_s[h] < +\infty$, we have for all $\tau \in \R$ and $(t,r^*,\omega) \in \Sigma_{\tau}$,
 $$ \int_{\C} |f| |v^{\mathstrut}_{\mathrm{N}}|^2 \dr \m_{\C}\Big\vert_{(t,r^*,\omega)} \, \lesssim^{\mathstrut}_{p,s}  \frac{\mathcal{E}^p_s[f]}{r^2 (1+|\tau|)^p}, \qquad \qquad \hspace{-0.2mm} \int_{\C}\frac{|v_{\uu}|^{p/2}}{|v_t|^{p/2}} |f| |v^{\mathstrut}_{\mathrm{N}}|^2 \dr \m_{\C}\Big\vert_{(t,r^*,\omega)} \, \lesssim^{\mathstrut}_{p,s}  \frac{\mathcal{E}^p_s[f]}{r^2 (1+|t+r^*|)^p}.$$ 
\end{itemize}
\end{Th}
\begin{Rq}
According to Lemma \ref{vol}, pointwise decay in $|\tau|$ corresponds to decay in $\min (|t-r^*|,|t+r^*|)$. Note also that using Hölder's inequality, we obtain for any $d \in [0,p]$,
$$ \int_{\C}\frac{|v_{\uu}|^{d/2}}{|v_t|^{d/2}} |f| |v^{\mathstrut}_{\mathrm{N}}|^2 \dr \m_{\C}\Big\vert_{(t,r^*,\omega)} \, \lesssim^{\mathstrut}_{p,s}  \frac{\mathcal{E}_s^p[f]}{r^2(1+|\tau|)^{p-d} (1+|t+r^*|)^{d}}.$$
Moreover, we can obtain through \eqref{vu}, which provides $\frac{|\slashed{v}|^2}{r^2} \leq |v_{\uu}||v^{\mathstrut}_{\mathrm{N}}|$, improved pointwise decay estimates for the angular component $\slashed{v}$ as well.
\end{Rq}
\begin{Rq}
Compare to $\int_{\mathcal{S}}  \int_{\C}  |f|  |v^{\mathstrut}_{\mathrm{N}}| |v \cdot \n_{\mathcal{S}} | \dr \m_{\C} \dr \m_{\mathcal{S}}$, the finiteness of the weighted $L^s$ norm on the right hand side of the integrated local energy decay estimate requires to assume more decay in $v$ on $f$ initially. This loss of integrability can however be arbitrary small. In contrast, no additional decay hypothesis in $r$ is in fact necessary. We refer to Remark \ref{Rqlabelplustard} for more details.

The same statement holds for the decay estimate of the energy flux so, in that sense, the rate of decay $(1+|\tau|)^p$ is optimal. The one of the pointwise decay estimates is not (see Remark \ref{Rqplustardaussi}).
\end{Rq}
\begin{Rq}
In view of the existence of steady states for the massive Vlasov equation on the exterior of a Schwarzschild black hole \cite{SarbachSchw}, no analogous result to Theorem \ref{theorem} can hold for massive particles.
\end{Rq}

\subsection{The red-shift effect}\label{redshiftsubsec}

Although $\partial_t$ is timelike in the region of outer communications $\mathscr{D}$, it becomes null on the event horizon. The consequence is that the energy density $\mathbb{T}[f](\partial_t, \n_{\Sigma_{\tau}})$ degenerates near $\mathcal{H}^+$. Even if we will work throughout this article with the Regge-Wheeler coordinates, it is convenient to use here the coordinate system $(\underline{u} , r , \theta , \varphi) \in \R \times \R_+^* \times ]0,\pi[ \times ]0,2 \pi [$, which covers the region $\mathscr{B} \cup \mathcal{H}^+ \cup \mathscr{D}$ of the maximally extended Schwarzschild spacetime $(\mathcal{M},g)$ (see the Penrose diagram of Figure \ref{Penrose1}). The metric takes the form
$$ g \, = \, -\left( 1 - \frac{2M}{r} \right) \dr \underline{u}^2+2 \dr \underline{u} \dr r+r^2 \dr \sigma^{\mathstrut}_{\mathbb{S}^2}$$
and is indeed regular on $\mathcal{H}^+$. In order to avoid any confusion, we denote by $\partial_{\underline{u}}'$ and $\partial_{r}'$ the differentiation with respect to $\underline{u}$ and $r$ in the coordinate system $(\underline{u} , r , \theta , \varphi)$. Then, in $\mathscr{D}$,
$$\partial_t \, = \,  \partial_{\underline{u}}'  , \qquad \partial_{r^*} \, = \, \partial_{\underline{u}}'+\left( 1 - \frac{2M}{r} \right) \partial_r', \qquad \partial_{\underline{u}} \, = \, \partial_{\underline{u}}'+\frac{1}{2}\left( 1 - \frac{2M}{r} \right) \partial_r', \qquad \partial_u \, = \, -\frac{1}{2}\left( 1 - \frac{2M}{r} \right) \partial_r' .$$
This implies that $\partial_{\underline{u}}$ (respectively $\dr \uu$) and, in particular, $ \frac{ 1}{1-\frac{2M}{r}} \partial_u$ (respectively $\left( 1-\frac{2M}{r} \right) \dr u$) can be extended as smooth vector fields (respectively $1$-forms) on $\mathscr{B} \cup \mathcal{H}^+ \cup \mathscr{D}$. Now, consider $(x,v) \in \mathcal{P}$ such that
$$ x=(t,r^*,\theta , \varphi), \quad v_{r^*}=-\frac{1}{2}\left( 1-\frac{2M}{r}\right), \quad v_\theta= v_\varphi=0.$$
In particular, we have $v_t=v_{r^*}$, $v_u=0$ and $v_{\uu}=-\frac{1}{2}\left( 1-\frac{2M}{r}\right)$. Consequently,
$$ v_t \, = \, g(\partial_{\uu}, \partial_{\uu}')\, = \, -\frac{1}{2}\left( 1-\frac{2M}{r}\right) \,  \xrightarrow[ r \to 2M ]{} \, 0, \qquad \qquad  v(\partial_r') \,  = \,  g(\partial_{\uu},\partial_r') \,  = \, 1,$$
so that $\mathbb{T}[f](\partial_t, \n_{\Sigma_{\tau}})$ does not uniformly control $\mathbb{T}[f](\partial_r', \n_{\Sigma_{\tau}})$ on $\mathscr{D}$, for a nonnegative function $f$. As it can be directly checked in \eqref{defv0}, the angular components cannot be uniformly controlled by $\mathbb{T}[f](\partial_t, \n_{\Sigma_{\tau}})$ either. The following result, implied by Lemma \ref{Lemannex}, suggests us to consider an other multiplier than $\partial_t$ in order to obtain a strong control of $f$ near the event horizon.
\begin{Lem}\label{controlcompo}
Let $2M < r_1 \leq R_0$ and $V$ be a $\varphi_t$-invariant future directed timelike vector field on the region $\{ 2M \leq r \leq r_1 \}$ of $\mathcal{H}^+ \cup \mathscr{D}$. Then, for any smooth vector fields $X_1, \, X_2$ and for all $\tau \geq 0$, there holds on the compact set $\Sigma_{\tau} \cap \{ 2M \leq r \leq r_1 \}$,
$$ \big| \mathbb{T}[f](X_1,X_2) \big| \, \leq \, C \cdot \big| \mathbb{T}[f](V,\n_{\Sigma_{\tau}}) \big|,$$
where $C$ depends on $\tau, \, V, \, X_1, \, X_2$ and $\mathcal{S}$.
\end{Lem}
The goal is then to find a $\varphi_t$-invariant future directed timelike vector field $\mathrm{N}$ on $\mathcal{H}^+ \cup \mathscr{D}$, which satisfies $\mathrm{N}=\partial_t$ for $r \geq r_1>2M$, and such that we can control $\nabla^{\mu} \mathbb{T}[f]_{\mu \nu} \mathrm{N}^{\nu}$. As $\partial_{\uu}+\frac{1}{1-\frac{2M}{r}} \partial_u$ can be extended as a timelike vector field on $\mathscr{B} \cup \mathcal{H}^+ \cup \mathscr{D}$, which contains $\{ 2M \leq r \leq r_1 \}$, we have\footnote{In other words, near the event horizon, the velocity current $v$ should then rather be decomposed as $ v = v_{\uu}' \dr \uu+ v_r' \dr r+v_{\theta} \dr \theta +v_{\varphi} \dr \varphi$, where, in particular $\frac{2v_u}{1-\frac{2M}{r}}=  v_r'$.}
$$ \mathbb{T}[f](\mathrm{N},\n_{\Sigma_{\tau}}) \, = \, \int_{\C} f \, |v^{\mathstrut}_{\mathrm{N}}| |v \cdot \n_{\Sigma_{\tau}}| \dr \m_{\C}, \qquad \text{where} \quad v^{\mathstrut}_{\mathrm{N}} := v(N),  \quad |v^{\mathstrut}_{\mathrm{N}}|  \sim |v_{\uu}|+\frac{|v_u|}{1-\frac{2M}{r}} +\frac{|\slashed{v|}}{r} .$$
However, the conservation law given by $\nabla^{\mu} \mathbb{T}[f]_{\mu t}= \nabla_{\mu} N[ v_t f]=0$ merely gives us a uniform bound on $\int_{\Sigma_{\tau}} \mathbb{T}[f](\partial_t , \n_{\Sigma_{\tau}}) \dr \m_{\Sigma_{\tau}}$ (see Proposition \ref{divergence0}). In order to bound the energy norm of $f|v^{\mathstrut}_{\mathrm{N}}|$, or even $f|v^{\mathstrut}_{\mathrm{N}}|^a$, we take advantage of the redshift effect, which is in particular captured by the fact that such a timelike vector field $\mathrm{N}$ can be chosen so that
$$ \exists \, b>0, \quad \exists \, r_0 >2M, \quad \forall \, 2M < r \leq r_0, \qquad -\T \big( |v^{\mathstrut}_{\mathrm{N}}| \big) |v_t| \geq b |v^{\mathstrut}_{\mathrm{N}}|^2.$$
Hence, when we apply the divergence theorem to $\mathbb{T}[f]_{\mu \nu} \mathrm{N}^{\nu}$, the worst error terms have a good sign. It is well-known that a similar issue occurs for the wave equation on black hole spacetimes and it was understood by Dafermos-Rodnianski in \cite{redshift} that the red-shift property of the horizon is the key for proving a non-degenerate energy estimate in the Schwarzschild spacetime. In fact, the existence of such a strictly timelike vector field $\mathrm{N}$ is related to the strict positivity of the surface gravity $\kappa=\frac{1}{4M}$ on the Killing horizon $r=2M$. The surface gravity is defined such that the Killing vector field $\partial_{\uu}'$, which is null on $\mathcal{H}^+$ and equal to $\partial_t$ in $\mathscr{D}$, satisfies $\nabla_{\partial_{\uu}'} \partial_{\uu}'= \kappa \partial_{\uu}'$ on $\mathcal{H}^+$. We refer to \cite[Section $3$]{Volker} for the detailed construction of $\mathrm{N}$ and to \cite[Section $7$]{LecBH} for a more general result covering all classical sub-extremal black holes. The behaviour of the solutions to the wave equation on extremal black holes, for which the surface gravity vanishes, is quite different (see \cite{aretakis1,aretakis2}).

\subsection{Vector field methods}

In order to obtain decay estimates for solutions to wave equations, Klainerman developed in \cite{Kl85} what is now referred as the vector field method. The methodology has been recently adapted for Vlasov equations by Fajman-Joudioux-Smulevici \cite{FJS}. These approaches turned out to be applicable for non-linear equations and led in particular to the proof of the stability of Minkowski spacetime as a solution to Einstein equations \cite{CKM,LR10} or to Einstein-Vlasov system \cite{Taylor,FJS3,Lindblad,BFJS}. Instead of using, say, representation formula, these strategies rely on
\begin{enumerate}
\item a set of vector fields, used as commutators, which reflect the symmetries of the equation. In our case, we will use the complete lift of the Killing vector fields of the Schwarzschild spacetime, i.e. $\partial_t$, $\widehat{\Omega}_1$, $\widehat{\Omega}_2$ and $\widehat{\Omega}_3$, since they commute with the Vlasov operator $\T$.
\item Energy inequalities. Well-chosen vector fields are used as multipliers in order to prove boundedness for weighted $L^p$ norms of the solution studied and its derivatives. For Vlasov equations, as mentionned previously, using the vector field $X$ as a multiplier consists in multiplying $f$ by the weight $v(X)$ and then to apply the divergence theorem to $N[v(X) f]_{\mu}$. More generally, we will consider quantities of the form $|v(X)|^a |\partial_t^n \widehat{\Omega}^I f|$, with $a \in \R$. In particular, if $\T(f)=0$ and $X$ is conformal Killing, this gives a conservation law.
\item Weighted Sobolev embeddings in order to derive pointwise decay estimates for the solution. For the classical vector field method, it is crucial to use commutators or multipliers with weights in $t$ in order to obtain time decay.
\end{enumerate}
The main difficulty for adapting this strategy to the study of wave equations on black holes is related to the lack of symmetries of these Lorentzian manifolds compared to Minkowski spacetime. Nonetheless, pointwise decay estimates has been proved in the case of Schwarschild, and even slowly rotating Kerr spacetimes, by considering multipliers or commutators analogous to the Morawetz and the scaling vector fields of the Minkowski space \cite{redshift,BlueSter,Luk1,Luk2} (see also \cite{AB} for an extension of the vector field method relying on hidden symmetries). In order to control the error terms arising in the energy estimates from the failure of these vector fields, which carry weights in $t$, to be conformal Killing, a new feature is required compared to the classical method of Klainerman. More precisely, in all these works, an integrated energy decay estimate is proved and crucially used. In our case, for the massless Vlasov equation, it would be an inequality of the form
\begin{equation}\label{eq:ugent}
 \int_{\tau=-\infty}^{+\infty} \int_{\Sigma_{\tau}} \int_{\C} A(r) f |v^{\mathstrut}_{\mathrm{N}}| |v \cdot \n_{\Sigma_{\tau}} | \dr \m_{\Sigma_{\tau}} \dr \tau \, \leq \, D_0,
\end{equation}
where $D_0$ is a constant depending on the values of $f$ on the initial hypersurface $\mathcal{S}$ and $A$ is a bounded nonnegative function which can possibly vanish. These type of estimates are obtained by using multipliers $X$ generating a bulk integral with a nonnegative integrand and controllable boundary terms once the divergence theorem is applied to $T[ f]_{\mu \nu} X^{\nu}$. In order to prove a non-degenerate local integrated energy decay on Schwarzschild spacetimes, i.e. an inequality such as \eqref{eq:ugent} where $A$ only vanishes at spatial infinity, the main difficulty once the redshift effect is well understood is to overcome the problems related to the trapping at the photon sphere $r=3M$, where trapped null geodesics are orbiting. In the case of the wave equation, this can only be achieved by losing regularity \cite{Sbierski}. In fact, a loss of an $\epsilon$ of an angular derivative is sufficient \cite{BlueSoffer} and we prove in this paper a similar result for the massless Vlasov equation. More precisely, we obtain a non-degenerate local integrated energy estimate with a loss of an $\epsilon$ of integrability in $\slashed{v}$ and we prove that this loss is necessary (see Propositions \ref{ILED} and \ref{noILED}).

In \cite{rpDR}, Dafermos-Rodnianski presented a new approach for the study of wave equations which turns out to be more robust than the classical vector field method. One of the goal of this paper is to adapt it to massless Vlasov equations. It relies on
\begin{enumerate}
\item the boundedness of an energy norm. In this article, it will be $\int_{\Sigma_{\tau}} \! \int_{\C} f | v^{\mathstrut}_{\mathrm{N}} | |v \cdot \n_{\Sigma_{\tau}} | \dr \m_{\C} \dr \m_{\Sigma_{\tau}}$.
\item A non-degenerate integrated local energy decay. In our case, this would be an inequality such as \eqref{eq:ugent} and where $A(r)>0$.
\item A hierarchy of $r^p$-weighted energy estimates. For the wave equation, this is achieved by using $r^p \partial_{\uu}$, for $0 \leq p \leq 2$, as a multiplier. In this article, in order to derive superpolynomial decay for the massless Vlasov equation, we will consider the weights $r^p |v_{\uu}|^q$, with $0 \leq 2p \leq q$.
\end{enumerate}
Together, these three points are sufficient to derive decay in $\tau$ for, in our case, $\int_{\Sigma_{\tau}} \int_{\C} f | v^{\mathstrut}_{\mathrm{N}} | |v \cdot \n_{\Sigma_{\tau}}|  \dr \m_{\C} \dr \m_{\Sigma_{\tau}}$. The pointwise decay estimates are then obtained by applying the method to, say, $\partial_t^n \widehat{\Omega}^I f$ and then by using Sobolev embeddings. However, for Vlasov fields, there is an additional difficulty compared to the case of the wave equation since we do not have a conservation law for the radial derivative $\partial_{r}f$. Instead, we use the equation $\T(f)=0$ in order to estimate $\partial_{r} \int_{\C} \frac{|v_{r^*}|^2}{|v_t|^2} f |v_t||v \cdot \n_{\Sigma_{\tau}}| \dr \m_{\C}$ in $L^1(\Sigma_{\tau})$ and we control the non-degenerate energy density $\mathbb{T}[f](\mathrm{N},\mathrm{N})$ by higher moments of the solution since, for instance,
$$ \int_{\C}  f |v^{\mathstrut}_{\mathrm{N}}|^2 \dr \m_{\C} \, \lesssim \, \frac{1}{r^{\frac{3}{2}}} \left|\int_{\C} \frac{|v_{r^*}|^2}{|v_t|^2} \frac{|v_{\uu}|^2}{|v_t|^2} |f|^4 \big(1+|v_t|+|\slashed{v}| \big)^4 |v_t||v^{\mathstrut}_{\mathrm{N}}|^{10} \dr \m_{\C} \right|^{\frac{1}{4}}.$$
An important feature of this approach, which explains part of its robustness, is that no vector field with weights growing in time are required. Indeed, the decay in $t$ is retrieved by the energy decay.
\begin{Rq}
We emphasize that it is crucial, in order to apply the method of \cite{rpDR}, to choose a foliation adapted to the behaviour of massless particles such as the spacelike-null hypersurfaces $\Sigma_{\tau}$ considered in this article (see Figure \ref{fig2}). In fact, the flux through $\Sigma_{\tau}$ measures the energy that has not been radiated to null infinity or inside the black hole up to time $\tau \geq 0$. In contrast, the flux through the hypersurface $\{t = \tau \}$ is preserved. Since such an hypersurface intersects the event horizon at the bifurcation sphere $\mathcal{H}^+ \cap \mathcal{H}^-$ and terminates at spatial infinity (see Figure \ref{fig20}), its energy flux will measure all the energy radiated to null infinity and inside the black hole.
\end{Rq}

This method has been extensively used in order to study solutions to wave equations in various context.\\ It can be applied to a large class of spacetimes \cite{Moschidisrp}, including in particular the subextremal Kerr black holes $|a|<M$ (see \cite{Kerrdecay}). With a refinement of the approach, Angelopoulos-Aretakis-Gajic provided a characterization of all the solutions to the wave equation, on a large class of stationary spherically symmetric asymptotically flat spacetimes, which satisfy Price's law as a lower bound \cite{Dejanadv,Dejanann}. The $r^p$-weighted energy method is also substantially applied in the dynamical satibility of Kerr family of black holes as solutions to the Einstein equations. This includes the study of the Teukolsky equation \cite{PasqualottoTeuko,DHRTeuko,GiorgiTeuko2}, stability results for linearized gravity around a Schwarzschild or even a Kerr solution \cite{DHRlinstab,ThomasJohnson,ABBM} as well as the non-linear stability of the Schwarzschild family with respect to polarized axially symmetric perturbations \cite{KlSz}. Finally, let us also mention that the method has also been succesfully applied in order to study various non-linear problems on spacetimes close to the Minkowski space \cite{ShiwuYangmedia0,ShiwuYangwaveLarge,ShiwuYangMKG2,ShiwuYangMKG1,ShiwuYangsemi1,ShiwuYangYuMKG,JosephKeir} and on a fixed black hole background \cite{Pasqualottononlin}.

Very few are known about the asymptotic properties of solutions to the massless Vlasov equation on black hole spacetimes. To our knowledge, there is only one result, due to Andersson-Blue-Joudioux \cite{ABJ}, for slowly rotating Kerr spacetimes. They proved the boundedness of a weighted energy norm as well as an integrated energy decay estimate degenerating at the horizon and on a neighborhood of $r=3M$ but they do not obtain any pointwise decay estimate on the velocity average of the Vlasov field. The main goal of this paper is then to adapt the $r^p$-weighted energy method to massless Vlasov equation and to obtain pointwise decay estimates on $\int_{\C} f |v^{\mathstrut}_{\mathrm{N}}|^2 \dr \m_{\C}$. Since the techniques used in this paper for the study of massless Vlasov fields are compatible with the ones developed by Dafermos-Rodnianski for solutions to wave equations, we expect that they will be useful for studying non-linear equations such as the Einstein-Vlasov system. The problem of late time asymptotics for the massless Vlasov equation constitutes another possible application. How asymptotics along $\mathcal{S}$ (or along $\mathcal{I}^-$ and $\mathcal{H}^-$) translates into asymptotics along $\mathcal{I}^+$?

\subsection{Structure of the paper}

The starting point for the proof of Theorem \ref{theorem} consists in splitting the region $\Rm_{-\infty}^{+\infty}$ into two domains separated by $\N_0$, $\Rm^{+\infty}_0$, where $\tau \geq 0$, and $\Rm_{-\infty}^{0}$, corresponding to $\tau \leq 0$ and where the analysis is much easier. 

In the region $\tau \geq 0$, we consider the massless Vlasov equation \eqref{Vlasovintro} with initial data given on the hypersurface $\Sigma_0$. In Section \ref{sec2}, by taking advantage of the red-shift effect, we prove a non-degenerate energy estimate which allows us in particular to control the energy density $\rho [ |f| |v^{\mathstrut}_{\mathrm{N}}|]$ in $L^1 ( \Sigma_{\tau})$. Next, we derive in Section \ref{sec3} various integrated energy decay estimates, including the one stated in Theorem \ref{theorem}. In addition, we prove in Proposition \ref{noILED} that no integrated local energy decay statement can hold without a loss of integrability. In Section \ref{sec4}, we establish first a hierarchy of $r^p |v_{\uu}|^q$-weighted energy inequalities, leading to the boundedness of the $r$-weighted norm of Theorem \ref{theorem}. With the help of the results of Sections \ref{sec2}-\ref{sec3}, we then prove energy decay estimates for massless Vlasov fields. Section \ref{sec5} is devoted to $L^{\infty}-L^s$ estimates, with $1 \leq s < +\infty$, for solutions to $\T(f)=0$. In particular, this will allow us to derive pointwise decay estimates for $\rho [ |f| |v^{\mathstrut}_{\mathrm{N}}|]$ in the region $\tau \geq 0$.

The full treatment of the second region, located in $\{ r^* \geq R_0^*+t \}$ and corresponding to the exterior of a light cone, can be carried out independantly of the other domain and is done in Section \ref{sec6}. None of the difficulties related to the event horizon or the trapping appear here, which makes the energy decay estimates much easier to prove. Moreover, if the initial data $f\vert_{\mathcal{S}}$ is supported in $\{ r^* < R_0^* \}$, $f$ vanishes on $\Rm_{-\infty}^0$ and the analysis is reduced to the first domain. 

Finally, we prove Theorem \ref{theorem} in Section \ref{sec7}. We apply first the results obtained in Section \ref{sec6} in order to control the solution in the region $\tau \leq 0$, containing in particular the hypersurface $\Sigma_0$. This permits us to study the solution in the domain $\tau \geq 0$ by appealling to the results proved in Sections \ref{sec2}-\ref{sec5}.

\section{Non-degenerate energy inequality}\label{sec2}

As explained in Subsection \ref{redshiftsubsec}, $|v_t|$ does not uniformly control $|v(\partial_r')|$ and $\frac{|\slashed{v}|}{r}$ near the event horizon, and then on $\mathscr{D}$. Consequently, we cannot obtain a strong control in $L^1(\Sigma_{\tau})$ of a solution to $\T(f)=0$ from a direct application of the conservation law of Proposition \ref{divergence0}, proved below. The purpose of this section is then prove an energy inequality allowing us to propagate norms carrying the weight $v^{\mathstrut}_{\mathrm{N}}$, which satisfies
\begin{equation}\label{defvtbar}
|v^{\mathstrut}_{\mathrm{N}}| \, \lesssim \, |v_{\underline{u}}|+ \frac{|v_u|}{1-\frac{2M}{r}}+\frac{\slashed{v}}{r} \, \lesssim \, |v^{\mathstrut}_{\mathrm{N}}|,
\end{equation}
instead of $v_t$.
\begin{Pro}\label{energyredshift}
There exist $2M < r_0 < r_1< 3M$ and a $\varphi_t$-invariant future oriented timelike vector field $\mathrm{N}$ on $\mathcal{H}^+ \cup \mathscr{D}$ such that
\begin{enumerate}
\item $\exists \, b >0, \; \forall \, 2M < r \leq r_0, \quad -\T( |v^{\mathstrut}_{\mathrm{N}}| ) \geq b \, |v^{\mathstrut}_{\mathrm{N}}|^2$,
\item $\exists \, B >0, \; \forall \,r \geq r_0, \quad | \T(v^{\mathstrut}_{\mathrm{N}}  ) | \leq B \, |v_t|^2$,
\item $\forall \, r \geq r_1 , \quad \mathrm{N} = \partial_t$.
\end{enumerate}
For any sufficiently regular function $f : \widehat{\Rm}_0^{+\infty} \rightarrow \R$ and any $a \in \R_+$, we have, for all $0 \leq \tau_1 \leq \tau_2$,
$$ \int_{\Sigma_{\tau_2}} \rho \left[ |f| | v^{\mathstrut}_{\mathrm{N}} |^a \right] \dr \mu^{\mathstrut}_{\Sigma_{\tau_2}} \, \lesssim^{\mathstrut}_a \,  \int_{\Sigma_{\tau_1}} \rho \left[ |f| | v^{\mathstrut}_{\mathrm{N}} |^a \right] \dr \mu^{\mathstrut}_{\Sigma_{\tau_1}}   +\int_{\Rm^{\tau_2}_{\tau_1}} \int_{\C} \left| \T ( f) \right| | v^{\mathstrut}_{\mathrm{N}} |^a \, \dr \mu^{\mathstrut}_{\C} \dr \m_{\Rm^{\tau_2}_{\tau_1}}.$$
\end{Pro}
\begin{Rq}
There are several vector fields verifying the properties of Proposition \ref{energyredshift}. An explicit possible choice of such a vector field $\mathrm{N}$ is given in \eqref{redshiftdef}.
\end{Rq}
\begin{Rq}
The properties satisfied by $\mathrm{N}$ together with Lemma \ref{Lemannex} imply \eqref{defvtbar}. If $\mathrm{N}$ is defined by \eqref{redshiftdef}, then \eqref{defvtbar} can be obtained directly.
\end{Rq}

\subsection{Degenerate conservation law}

We derive here energy inequalities for $\rho[h]$ which, if $h \geq 0$ is a solution to the massless Vlasov equation, are related to the conservation of the number of particles. Combined with $\T(v_t)=0$, it implies estimates for the degenerate energy density $\rho [h |v_t|]$. 
\begin{Pro}\label{divergence0}
Let $h : \widehat{\Rm}_0^{+\infty} \rightarrow \R_+$ be a sufficiently regular nonnegative function. For all $0 \leq \tau_1 \leq \tau_2$,
$$  \int_{\Sigma_{\tau_2}} \rho \big[h \big]  \dr \mu^{\mathstrut}_{\Sigma_{\tau_2}} \, \leq \,  \int_{\Sigma_{\tau_1}} \rho \big[ h \big]  \dr \mu^{\mathstrut}_{\Sigma_{\tau_1}}  +\int_{\Rm^{\tau_2}_{\tau_1}} \int_{\C} \T ( h)  \, \dr \mu^{\mathstrut}_{\C} \dr \mu_{\Rm_{\tau_1}^{\tau_2}} .$$
\end{Pro}
\begin{proof}
Let $w > \tau_2+u_0$, $\underline{w} > w+2R_0^*$ and consider the domain $\Rm_{\tau_1}^{\tau_2} \cap \{ u \geq w \} \cap \{ \underline{u} \geq \underline{w} \}$. As suggested by Figures \ref{fig20}-\ref{fig2}, its boundary is composed by
\begin{itemize}
\item a piece of the hypersurfaces $\Sigma_{\tau_1}$ and $\Sigma_{\tau_2}$. Recall that $v \cdot \n_{\Sigma_{\tau}} \leq 0$.
\item A piece of the hypersurface $C_w:=\{u=w\}$ which admits $\n_{C_w}=\partial_{\uu}$ as normal vector and $ \dr \m_{C_w}=r^2 \dr \uu \wedge \dr \sigma^{\mathstrut}_{\mathbb{S}^2}$ as induced volume form. In particular, $v \cdot \n_{C_w} \leq 0$.
\item A piece of the hypersurface $\underline{\N}_{\underline{w}}$. Recall that $v \cdot \n_{\underline{\N}_{\underline{w}}}=v_u \leq 0$ and $\dr \m_{\underline{\N}_{\underline{w}}}=-r^2 \dr u \wedge \dr \sigma^{\mathstrut}_{\mathbb{S}^2}$.
\end{itemize} 
By applying the divergence theorem to $N[h]_{\mu}$ in the domain considered, we have, in view of \eqref{divN} and \eqref{defrho}, 
\begin{align}
\nonumber \int_{C_w} \! \mathds{1}_{\tau_1 \leq \tau \leq \tau_2} \! \int_{\C} h \, |v \cdot n^{\mathstrut}_{C_w}| \dr \m_{\C} \dr \mu^{\mathstrut}_{C_w}+ \int_{\Sigma_{\tau_2}} \! \mathds{1}_{u \leq w} \,  \mathds{1}_{\underline{u} \leq \underline{w}} \, \rho \left[h \right] \dr \mu^{\mathstrut}_{\Sigma_{\tau_2}} +\int_{\underline{\N}_{\underline{w}}} \! \mathds{1}_{\tau_1 \leq \tau \leq \tau_2} \! \int_{\C} h \, & |v \cdot n^{\mathstrut}_{\underline{\N}_{\underline{w}}}| \dr \m_{\C} \dr \mu^{\mathstrut}_{\underline{\N}_{\underline{w}}} \\ - \int_{\Sigma_{\tau_1}} \mathds{1}_{u \leq w} \,  \mathds{1}_{\underline{u} \leq \underline{w}} \, \rho \left[ h \right] \dr \mu^{\mathstrut}_{\Sigma_{\tau_1}} \, = \, \int_{\Rm^{\tau_2}_{\tau_1}} \mathds{1}_{u \leq w} \,  \mathds{1}_{\underline{u} \leq \underline{w}} \, \int_{\C} \T ( h)  \, \dr \mu^{\mathstrut}_{\C} \dr \mu_{\Rm_{\tau_1}^{\tau_2}} . & \label{eq:divh}
\end{align}
Since $h \geq 0$, we get
\begin{equation*}
\int_{\Sigma_{\tau_2}} \mathds{1}_{u \leq w} \,  \mathds{1}_{\underline{u} \leq \underline{w}} \, \rho \left[ h \right] \dr \mu^{\mathstrut}_{\Sigma_{\tau_2}} \leq   \int_{\Sigma_{\tau_1}} \mathds{1}_{u \leq w} \,  \mathds{1}_{\underline{u} \leq \underline{w}} \, \rho \left[ h \right] \dr \mu^{\mathstrut}_{\Sigma_{\tau_1}} + \int_{\Rm^{\tau_2}_{\tau_1}} \mathds{1}_{u \leq w} \,  \mathds{1}_{\underline{u} \leq \underline{w}} \, \int_{\C} \T ( h)  \, \dr \mu^{\mathstrut}_{\C} \dr \mu_{\Rm_{\tau_1}^{\tau_2}} 
\end{equation*}
and the result follows from the dominated convergence theorem.
\end{proof} 
\begin{Rq}\label{divergence0Rq}
Note that \eqref{eq:divh} and the dominated convergence theorem also gives us that, for all $0 \leq \tau_1 \leq \tau_2$,
$$ \sup_{\underline{w} \in \R} \int_{\underline{\N}_{\underline{w}}} \! \mathds{1}_{\tau_1 \leq \tau \leq \tau_2} \! \int_{\C} h \,  |v_u| \dr \m_{\C} \dr \mu^{\mathstrut}_{\underline{\N}_{\underline{w}}}  \, \leq \,  \int_{\Sigma_{\tau_1}} \rho \big[ h \big]  \dr \mu^{\mathstrut}_{\Sigma_{\tau_1}}  +\int_{\Rm^{\tau_2}_{\tau_1}} \int_{\C} \T ( h)  \, \dr \mu^{\mathstrut}_{\C} \dr \mu_{\Rm_{\tau_1}^{\tau_2}} .$$
This estimate will be useful in order to derive pointwise decay estimates since $|v \cdot \n_{\Sigma_{\tau}}|$ merely controls the component $v_{\uu}$ for $r \geq R_0$.
\end{Rq}
 \begin{Rq}
Although we will not use these properties in this article, it is interesting to remark that we also obtain from \eqref{eq:divh} that the following two limits exist,
\begin{align*}
\int_{\mathcal{H}^+} \mathds{1}_{\tau_1 \leq \tau \leq \tau_2} \int_{\C} h \, | v \cdot n^{\mathstrut}_{\mathcal{H}^+}| \dr \m_{\C} \dr \mu^{\mathstrut}_{\mathcal{H}^+} \, & := \, \lim_{u \to + \infty} \int_{\N_{u}'} \mathds{1}_{\tau_1 \leq \tau \leq \tau_2} \int_{\C} h \, |v_{\uu}| \dr \m_{\C} r^2  \dr \mu^{\mathstrut}_{\mathbb{S}^2} \dr \uu , \\
\int_{\mathcal{I}^+} \mathds{1}_{\tau_1 \leq u-u_0 \leq \tau_2}  \int_{\C} r^2 h \, |v_u| \dr \m_{\C}  \dr \m_{\mathcal{I}^+}  \, & := \, \lim_{\underline{u} \to + \infty} \int_{\underline{\N}_{\uu}} \mathds{1}_{\tau_1 \leq u-u_0 \leq \tau_2}  \int_{\C} r^2 h \, |v_u|  \dr \m_{\C}  \dr \mu^{\mathstrut}_{\mathbb{S}^2} \dr u  .
\end{align*}
\begin{itemize}
\item $\int_{\mathcal{H}^+} \mathds{1}_{\tau_1 \leq \tau \leq \tau_2} \int_{\C} h \, | v \cdot n^{\mathstrut}_{\mathcal{H}^+}| \dr \m_{\C} \dr \mu^{\mathstrut}_{\mathcal{H}^+}$ is the number of particles penetrating the black hole between the advanced times $\tau_1$ and $\tau_2$. If $h$ is well defined on $\mathcal{H}^+$ this quantity can be explicitly computed in the system of coordinate $(\uu, r , \theta , \varphi)$, which is regular on $\mathcal{H}^+$. 
\item $\int_{\mathcal{I}^+} \mathds{1}_{\tau_1 \leq u-u_0 \leq \tau_2}  \int_{\C} r^2 h \, |v_u| \dr \m_{\C}  \dr \m_{\mathcal{I}^+}$ corresponds to the number of particles reaching null infinity between retarded times $\tau_1$ and $\tau_2$. The set $\mathcal{I}^+$ is not part of the Schwarzschild spacetime but can be viewed as the level set $\uu = +\infty$ equipped with the volume form $\dr \m_{\mathcal{I}^+} = \dr u \wedge \dr \m_{\mathbb{S}^2}$ and having $\partial_u$ as a normal vector.
\end{itemize}
 Moreover, if $\T(h)=0$, the following conservation law holds 
$$\int_{\mathcal{H}^+}\! \mathds{1}_{\tau_1 \leq \tau \leq \tau_2}\! \int_{\C}\! h \, | v \cdot n^{\mathstrut}_{\mathcal{H}^+}| \dr \m_{\C} \dr \mu^{\mathstrut}_{\mathcal{H}^+}+ \! \int_{\Sigma_{\tau_2}}\! \rho \big[h \big]  \dr \mu^{\mathstrut}_{\Sigma_{\tau_2}}+\int_{\mathcal{I}^+} \! \mathds{1}_{\tau_1 \leq u-u_0 \leq \tau_2} \! \int_{\C} r^2 h \, |v_u| \dr \m_{\C}  \dr \m_{\mathcal{I}^+} = \!\int_{\Sigma_{\tau_1}}\! \rho \big[h \big]  \dr \mu^{\mathstrut}_{\Sigma_{\tau_1}}\!.$$
 \end{Rq}

\subsection{The red-shift effect}

In order to capture the red-shift effect at the event horizon and improve the control of the solutions in that region, we will take advantage of a combination of the nonnegative weights
$$ - \frac{2v_u}{1-\frac{2M}{r}} \, = \, -\frac{v_t-v_{r^*}}{1-\frac{2M}{r}} \qquad \text{and} \qquad |v_t|.$$
More precisely, consider a constant $2M<r_0<3M$ which will be fixed below, $r_0 < r_1 < 3M$ and a smooth cutoff function $\chi \in \mathcal{C}^{\infty}_c (\R,[0,1])$ such that $\chi_{ \vert ]-\infty , r_0]}=1$ and $\chi_{ \vert[r_1,+\infty[ }=0$. Then, define
\begin{equation}\label{redshiftdef}
\mathrm{N} \, := \, \frac{4 \chi(r)}{1-\frac{2M}{r}} \partial_u +16 \chi (r)\left( 1- \frac{2M}{r} \right) \! \partial_t+ \partial_t, \qquad v^{\mathstrut}_{\mathrm{N}}:=v(\mathrm{N})=  \frac{4 \chi(r)v_u}{1-\frac{2M}{r}}  +16 \chi (r)\left( 1- \frac{2M}{r} \right) \! v_t+ v_t ,
\end{equation}
which is indeed a timelike vector field on $\mathscr{D}$, equal to $\partial_t$ for $r \geq r_1$. Furthermore, using the notations of Subsection \ref{redshiftsubsec}, we have $\mathrm{N}= -\partial_r'+8\left( 1- \frac{2M}{r} \right) \! \partial_{\uu}'+\partial_{\uu}'$ for $r \leq r_0$ so it can be extended on, and beyond, $\mathcal{H}^+$ as a smooth timelike vector field. By construction, $\mathrm{N}$ and its extension are $\varphi_t$-invariant and future oriented.
\begin{Lem}\label{comparo}
The estimate \eqref{defvtbar} holds. Moreover, for all $2M<r < R_0$, we have $|v^{\mathstrut}_{\mathrm{N}}| \lesssim | v \cdot \n_{\Sigma_{\tau}} | \lesssim |v^{\mathstrut}_{\mathrm{N}}|$.
\end{Lem}
\begin{proof}
According to Lemma \ref{vol} and \eqref{redshiftdef}, on the region $\{2M<r<R_0\}$, $v \cdot \n_{\Sigma_{\tau}}$ and $v^{\mathstrut}_{\mathrm{N}}$ are both of the form $\xi(r) v_{\uu}+\zeta(r)\frac{v_u}{1-\frac{2M}{r}}$, where $\xi$ and $\zeta$ are stricly positive functions defined on the compact $[2M,R_0]$. This implies that for $2M< r<R_0$
$$ |v_{\uu}|+\frac{|v_u|}{1-\frac{2M}{r}} \, \lesssim |v^{\mathstrut}_{\mathrm{N}}| \, \lesssim |v_{\uu}|+\frac{|v_u|}{1-\frac{2M}{r}}, \qquad  |v_{\uu}|+\frac{|v_u|}{1-\frac{2M}{r}} \, \lesssim |v \cdot \n_{\Sigma_{\tau}}| \, \lesssim |v_{\uu}|+\frac{|v_u|}{1-\frac{2M}{r}},$$
so that $|v^{\mathstrut}_{\mathrm{N}}|$ and $|v \cdot \n_{\Sigma_{\tau}}|$ are equivalent. Together with \eqref{vu}, it also implies \eqref{defvtbar} for $r < R_0$. Finally, note that \eqref{defvtbar} is straightforward in the region $r \geq R_0$ since each side of the inequality is comparable to $|v_t|$.
\end{proof}
In order to estimate $\T(|v^{\mathstrut}_{\mathrm{N}}|)$, we are led to compute the following quantities.
\begin{Lem}\label{Ty}
There holds on $\C$,
$$\T \left( \frac{2|v_u|}{1-\frac{2M}{r}} \right) \, = \,  -\frac{4M}{r^2}\frac{|v_u|^2}{(1-\frac{2M}{r})^2}+\frac{4}{r}\frac{|v_{\underline{u}}|| v_u|}{(1-\frac{2M}{r})} ,  \qquad \quad \T \left( \left( \! 1- \frac{2M}{r} \! \right)\! |v_t| \right) \, = \, \frac{2M}{r^2} |v_u|^2-\frac{2M}{r^2}|v_{\underline{u}}|^2 .$$
\end{Lem}
\begin{proof}
Recall from \eqref{defT} the expression of $\T$. As $2|v_u|=-v_t+v_{r^*}$, $\T(v_t)=0$ and $\partial_{r^*}=\left(1-\frac{2M}{r} \right) \partial_r$, we have
\begin{align*}
 \T \left( \! \frac{2|v_u|}{1-\frac{2M}{r}} \! \right) \, & = \, 2\T \left(\! \frac{1}{1-\frac{2M}{r}} \! \right)\! |v_u|+\frac{1}{1-\frac{2M}{r}} \T (v_{r^*}\!) \, = \, \frac{2v_{r^*}}{1-\frac{2M}{r}} \partial_{r^*} \left( \! \frac{1}{1-\frac{2M}{r}} \! \right)\! |v_u| +\frac{r-3M}{r^4 \left(1-\frac{2M}{r} \right)} |\slashed{v}|^2 \partial_{v_r^*} (v_{r^*}\!) \\
  & = \, -\frac{4M}{r^2(1-\frac{2M}{r})^2}v_{r^*}|v_u| + \frac{r-3M}{r^4(1-\frac{2M}{r})} |\slashed{v}|^2  .
  \end{align*}
Now, use that $v_{r^*}=v_{\underline{u}}-v_u$, $v_u \leq 0$, $ v_{\uu} \leq 0$ and \eqref{vu} in order to obtain
\begin{align*}
 \T \left( \frac{2|v_u|}{1-\frac{2M}{r}} \right) \, & = \, -\frac{4M|v_u|^2}{r^2(1-\frac{2M}{r})^2}+\frac{4M |v_{\underline{u}}| |v_u|}{r^2(1-\frac{2M}{r})^2}  + 4\frac{r-3M}{r^2(1-\frac{2M}{r})^2} |v_{\underline{u}}| |v_u|  \, = \, -\frac{4M |v_u|^2}{r^2(1-\frac{2M}{r})^2}+\frac{4 |v_{\underline{u}}|| v_u|}{r(1-\frac{2M}{r})} .
 \end{align*}
For the second identity, using $\T (v_t)=0$, $v_{t} =v_{\underline{u}}+v_u$ and $v_{r^*} =v_{\underline{u}}-v_u$, we get
$$ \T \left( \left( 1- \frac{2M}{r} \right) v_t \right) \hspace{1mm} = \hspace{1mm} \frac{v_{r^*}}{1-\frac{2M}{r}} \partial_{r^*} \left( 1-\frac{2M}{r} \right) v_t \hspace{1mm} = \hspace{1mm}  \frac{2M}{r^2}v_{r^*} v_t \hspace{1mm} = \hspace{1mm}  \frac{2M}{r^2}|v_{\underline{u}}|^2 -\frac{2M}{r^2} |v_u|^2$$ 
and it remains no notice that $|v_t|=-v_t$.
\end{proof}
We can now prove the first two properties satisfied by $v^{\mathstrut}_{\mathrm{N}}$.
\begin{Lem}\label{Conditions13}
If $r_0$ is chosen sufficiently close to $2M$, there exist $b>0$ and $B >0$ depending only on $M$ such that the following property holds. On $\C$, we have
$$ \forall \, 2M < r \leq r_0, \qquad  -\T \big( |v^{\mathstrut}_{\mathrm{N}}| \big) \, \geq \, b \, |v^{\mathstrut}_{\mathrm{N}}|^2 \qquad \text{and} \qquad \forall \,  r \geq r_0, \qquad \big| \T \big( |v^{\mathstrut}_{\mathrm{N}}| \big) \big| \,  \leq \, B |v_t|^2 \mathds{1}_{r_0 \leq r \leq r_1}.$$
\end{Lem}
\begin{proof}
Recall from \eqref{redshiftdef} the expression of $v^{\mathstrut}_{\mathrm{N}} \leq 0$. As $\T(|v_t|)=0$ and $\partial_{r^*}= \left( 1 - \frac{2M}{r} \right) \partial_r$, we have
$$ \frac{1}{2} \T \big( |v^{\mathstrut}_{\mathrm{N}}| \big) \, = \, v_{r^*} \partial_{r} \left(\chi (r) \right)\left( \frac{2|v_u|}{1-\frac{2M}{r}} +8\left( 1 - \frac{2M}{r} \right) \! |v_t| \right) +\chi(r) \T \left( \frac{2|v_u|}{1-\frac{2M}{r}} \right) +8\chi (r) \T \left( \left( 1 - \frac{2M}{r} \right) \! |v_t|  \right).$$
On the region $r \geq r_0$, we have $1\!-\!\frac{2M}{r} \geq 1\!-\!\frac{2M}{r_0} >0$, so Lemma \ref{Ty} together with $|v_u|, \, |v_{\uu}|, \, |v_{r^*}| \leq |v_t|$ yield
$$ \big|\T \big( |v^{\mathstrut}_{\mathrm{N}}| \big) \big| \, \lesssim \, \left( \| \chi \|^{\mathstrut}_{L^{\infty}}+\| \chi' \|^{\mathstrut}_{L^{\infty}} \right) |v_t|^2.$$
Moreover, if $r \geq r_1$, we have $v^{\mathstrut}_{\mathrm{N}}=v_t$ so $\T(|v^{\mathstrut}_{\mathrm{N}}|)=0$. Now, recall that $\chi =1$ on the region $2M < r \leq r_0$. Hence using Lemma \ref{Ty} as well as the inequality $4ab \leq a^2+4b^2$, we obtain
\begin{align*}
\forall \, 2M < r \leq r_0, \qquad -\frac{1}{2}\T  \big( |v^{\mathstrut}_{\mathrm{N}}| \big)  \, & \geq \,  \frac{4M |v_u|^2}{r^2(1-\frac{2M}{r})^2}-\frac{|v_u|^2}{r(1-\frac{2M}{r})^2}-\frac{4 |v_{\underline{u}}|^2}{r}-\frac{16M}{r^2}|v_u|^2 +\frac{16M}{r^2}|v_{\underline{u}}|^2 \\
& \geq \hspace{1mm} \frac{4M |v_u|^2}{r^2(1-\frac{2M}{r})^2} \left(1-\frac{r}{4M} -4\left(1-\frac{2M}{r} \right)^2 \right)+\frac{16M}{r^2}|v_{\underline{u}}|^2 \left( 1 - \frac{r}{4M} \right).
\end{align*}
Consequently, there exists $2M <r_0 < 3M$ depending only on $M$ such that for all $2M < r \leq r_0$,
$$-\frac{1}{2}\T \big( | v^{\mathstrut}_{\mathrm{N}}| \big) \,  \geq \, \frac{4M |v_u|^2}{r^2(1-\frac{2M}{r})^2} \cdot \frac{1}{4}+\frac{16M}{r^2}|v_{\underline{u}}|^2 \cdot \frac{1}{4} \, \geq \, \frac{1}{9M} \left( \frac{|v_u|^2}{(1-\frac{2M}{r})^2}+4|v_{\underline{u}}|^2 \right) \! .$$
This implies the first estimate since, for $r \leq r_0$, $|v^{\mathstrut}_{\mathrm{N}}| \lesssim |v_t| +\frac{|v_u|}{1-\frac{2M}{r}}$ and $|v_t| \leq |v_{\uu}|+|v_u| \leq |v_{\uu}|+\frac{|v_u|}{1-\frac{2M}{r}}$.
\end{proof}

\subsection{Proof of Proposition \ref{energyredshift}}

Let $\mathrm{N}$ be the vector field \eqref{redshiftdef}, which satisfies in particular the properties $1.$, $2.$ and $3.$ written in Proposition \ref{energyredshift}. The following lemma is a prerequisite for the proof of the non degenerate energy inequality.
\begin{Lem}\label{Lemforred}
Consider two real numbers $\tau_1 \leq \tau_2$ and three constants $c >0, \, D\geq 0, \, \overline{D} \geq 0$. Let $\phi \in \mathcal{C}^0 ( [\tau_1,\tau_2] , \R)$ satisfying
$$ \forall \, \tau_1 \leq \tau \leq \tau_2, \qquad \phi(\tau_2) +c\int_{\tau}^{\tau_2} \phi (s) \dr s \hspace{1mm} \leq \hspace{1mm} \phi(\tau)+D(\tau_2-\tau) +\overline{D}.$$
Then, 
$$  \phi(\tau_2)  \hspace{1mm} \leq \hspace{1mm} \phi(\tau_1)+\frac{D}{c}+\overline{D}.$$
\end{Lem}
\begin{proof}
Introduce $\psi :\tau \mapsto \phi(\tau)-\frac{D}{c}$ as well as $\Psi : \tau \mapsto \int_{\tau}^{\tau_2} \psi (s) \dr s$. Then, by assumption
\begin{equation}\label{eq:tructruc}
\forall \, \tau \in [\tau_1,\tau_2], \qquad \psi (\tau_2)+c\Psi(\tau) \hspace{1mm} \leq \hspace{1mm} \psi (\tau) + \overline{D}.
\end{equation}
Consequently, for all $\tau_1 \leq \tau \leq \tau_2$,
$$ -\frac{\dr}{\dr \tau} \left( \Psi (\tau) \cdot e^{c \tau} \right) \hspace{1mm} = \hspace{1mm} -\left(- \psi(\tau) +c\Psi(\tau) \right) e^{c \tau} \hspace{1mm} \geq \hspace{1mm} \left( \psi (\tau_2)-\overline{D} \right) e^{c \tau}.$$
Taking the integral between $\tau_1$ and $ \tau_2$ of both side of this inequality, we obtain
$$ \Psi (\tau_1) e^{c\tau_1} \, \geq \, \frac{\psi (\tau_2)-\overline{D}}{c} (e^{c\tau_2}-e^{c\tau_1})$$
and, estimating $c \Psi (\tau_1)$ using \eqref{eq:tructruc}, we get
$$ \psi(\tau_1)-\psi(\tau_2)+\overline{D} \, \geq \, c \Psi (\tau_1) \, \geq \, \psi (\tau_2)(e^{c(\tau_2-\tau_1)}-1)-\overline{D} (e^{c(\tau_2-\tau_1)}-1)  .$$
Using that $\psi = \phi-\frac{D}{c}$, we finally obtain, as $e^{-c(\tau_2-\tau_1)} \leq 1$ and $\frac{D}{c} \geq 0$,
\begin{align*}
 \phi(\tau_2) \hspace{1mm} & \leq \hspace{1mm}  \phi (\tau_1)e^{-c(\tau_2-\tau_1)} +\frac{D}{c} \left(1-e^{-c(\tau_2-\tau_1)} \right)+\overline{D} \, \leq \, \phi (\tau_1) +\frac{D}{c} +\overline{D}.
\end{align*}
\end{proof}
In order to lighten the notations, we introduce the following quantities.
\begin{Def}\label{defbulknorm}
Let $h : \widehat{\Rm}_{-\infty}^{+\infty} \rightarrow \R$ be a sufficiently regular function. Let, for any $a \in \R_+$ and $0 \leq \tau_1 \leq \tau_2 $,
$$ \mathbb{D}^{\mathrm{N}}_{a}[h]_{\tau_1}^{\tau_2} \,= \, \int_{\Rm_{\tau_1}^{\tau_2}} |\T (h)| |v^{\mathstrut}_{\mathrm{N}}|^a \dr \m_{\Rm_{\tau_1}^{\tau_2}}.$$
\end{Def}
Note that $ \mathbb{D}^{\mathrm{N}}_{a}[h]_{\tau_1}^{\tau_2}=0$ if $h$ is solution to the massless Vlasov equation. Finally, we will use several times the following technical result which is a direct consequence of Proposition \ref{divergence0}.
\begin{Lem}\label{Lemmatechenergy}
Let $a \in \R_+$ and $f:\widehat{\Rm}_0^{+\infty} \rightarrow \R$ be a sufficiently regular function. We have, $\forall \, 0 \leq \tau_1 \leq \tau \leq \tau_2$,
$$ \int_{\Sigma_{\tau}}  \rho \Big[|f| |v_t|^a \Big] \dr \mu^{\mathstrut}_{\Sigma_{\tau}} + \int_{\Rm_{\tau}^{\tau_2}}   \int_{\C} |\T (f) | |v^{\mathstrut}_{\mathrm{N}}|^a \dr \mu^{\mathstrut}_{\C}  \dr \mu^{\mathstrut}_{\Rm_{\tau}^{\tau_2}} \, \leq \,  \int_{\Sigma_{\tau_1}}  \rho \Big[|f| |v_{\mathrm{N}}^{\mathstrut}|^a \Big] \dr \mu^{\mathstrut}_{\Sigma_{\tau_1}} + \mathbb{D}^{\mathrm{N}}_{a}[f]_{\tau_1}^{\tau_2}. $$
\end{Lem}
\begin{proof}
Fix $0 \leq \tau_1 \leq \tau \leq \tau_2 $ and remark that $\T(|f||v_t|^a)=\T(|f|)|v_t|^a=\frac{f}{|f|} \T(f)|v_t|^a \leq |\T(f)| |v_t|^a$. Hence, applying Proposition \ref{divergence0} to $h=|f||v_t|^a$ and using $|v_t|^a \leq |v^{\mathstrut}_{\mathrm{N}}|^a$, we get
$$ \int_{\Sigma_{\tau}} \! \rho \Big[ |f| |v_t|^a \Big] \dr \mu^{\mathstrut}_{\Sigma_{\tau}} \leq  \int_{\Sigma_{\tau_1}}  \! \rho \Big[|f| |v^{\mathstrut}_{\mathrm{N}}|^a \Big] \dr \mu^{\mathstrut}_{\Sigma_{\tau_1}} + \int_{\Rm_{\tau_1}^{\tau}}  \! \int_{\C}  |\T (f)| |v^{\mathstrut}_{\mathrm{N}}|^a \dr \mu^{\mathstrut}_{\C}  \dr \m_{\Rm_{\tau_1}^{\tau}} \! = \int_{\Sigma_{\tau_1}}  \!  \rho \Big[|f| |v_{\mathrm{N}}^{\mathstrut}|^a \Big] \dr \mu^{\mathstrut}_{\Sigma_{\tau_1}} + \mathbb{D}^{\mathrm{N}}_{a}[f]_{\tau_1}^{\tau} $$
and the inequality then ensues from $\mathbb{D}^{\mathrm{N}}_{a}[f]_{\tau_1}^{\tau}+\mathbb{D}^{\mathrm{N}}_{a}[f]_{\tau}^{\tau_2} = \mathbb{D}^{\mathrm{N}}_{a}[f]_{\tau_1}^{\tau_2}$ (see Definition \ref{defbulknorm}).
\end{proof}

We are now able to prove the non-degenerate energy inequality stated in Proposition \ref{energyredshift}.
\begin{Pro}\label{energyredshiftbis}
The estimate of Proposition \ref{energyredshift} holds for any $a \in \R_+$ and any sufficiently regular function $f : \widehat{\Rm}_0^{+\infty} \rightarrow \R$. Moreover, we also have for all $0 \leq \tau_1 \leq \tau_2$,
$$ \int_{\tau_1}^{\tau_2}\! \int_{\Sigma_{\tau}} \!  \mathds{1}_{r \leq r_0} \, \rho \left[ |f||v^{\mathstrut}_{\mathrm{N}}|^a \right] \dr \m_{\Sigma_{\tau}} \dr \tau \, \lesssim^{\mathstrut}_a \int_{\tau_1}^{\tau_2} \! \int_{\Sigma_{\tau}} \! \mathds{1}_{r_0 \leq r \leq r_1} \, \rho \Big[ |f||v_t|^a \Big] \dr \m_{\Sigma_{\tau}} \dr \tau+ \int_{\Sigma_{\tau_1}} \rho \left[|f||v^{\mathstrut}_{\mathrm{N}}|^a\right] \dr \mu^{\mathstrut}_{\Sigma_{\tau_1}}+ \mathbb{D}^{\mathrm{N}}_{a}[f]_{\tau_1}^{\tau_2}  .$$
\end{Pro}
\begin{proof}
The main part of the proof consists in bounding sufficiently well the quantity
$$\phi (\tau)  \,  := \, \int_{\Sigma_{\tau}}  \rho \left[|f| |v^{\mathstrut}_{\mathrm{N}}|^a \right] \dr \mu^{\mathstrut}_{\Sigma_{\tau}}  $$
in order to apply Lemma \ref{Lemforred}. Fix $0 \leq \tau_1 \leq \tau_2$ and apply the energy inequality of Proposition \ref{divergence0} to $h=|f| |v^{\mathstrut}_{\mathrm{N}}|^a$. This gives, for any $ \tau_1 \leq \tau \leq \tau_2$,
$$ \int_{\Sigma_{\tau_2}} \rho \left[ |f| |v^{\mathstrut}_{\mathrm{N}}|^a \right] \dr \mu^{\mathstrut}_{\Sigma_{\tau_2}} \, \leq \, \int_{\Sigma_{\tau}} \rho \left[ |f| |v^{\mathstrut}_{\mathrm{N}}|^a \right]  \dr \mu^{\mathstrut}_{\Sigma_{\tau}}  + \int_{\Rm_{\tau}^{\tau_2}} \int_{\C} \T \left( |f| |v^{\mathstrut}_{\mathrm{N}}|^a \right)  \dr \mu^{\mathstrut}_{\C}  \dr \m_{\Rm_{\tau}^{\tau_2}}.$$
According to Lemma \ref{Conditions13} and using, in view of \eqref{redshiftdef}, that $|v_t| \leq |v^{\mathstrut}_{\mathrm{N}}| \lesssim |v_t|$ for $r \geq r_0$, we have
$$ \T ( |f| |v^{\mathstrut}_{\mathrm{N}}|^a ) = \frac{f}{|f|}\T(f )|v^{\mathstrut}_{\mathrm{N}}|^a +|f|\T ( |v^{\mathstrut}_{\mathrm{N}}|^a) \leq |\T(f)| |v^{\mathstrut}_{\mathrm{N}}|^a -ab|f||v^{\mathstrut}_{\mathrm{N}}|^{a+1} \, \mathds{1}_{r \leq r_0} +aB |f|  |v_t|^{a+1} \, \mathds{1}_{r_0 \leq r \leq r_1} .$$
We then deduce from these two inequalities and by applying Lemma \ref{Lemmatechenergy} that, 
$$ \phi (\tau_2) + ab \! \int_{\Rm_{\tau}^{\tau_2}} \! \mathds{1}_{r \leq r_0} \! \int_{\C} \! |f|  |v^{\mathstrut}_{\mathrm{N}}|^{a+1} \dr \mu^{\mathstrut}_{\C}  \dr \m_{\Rm_{\tau}^{\tau_2}}   \leq  \phi(\tau)+ aB \! \int_{\Rm_{\tau}^{\tau_2}} \! \mathds{1}_{r_0 \leq r \leq r_1} \! \int_{\C} \! |f|  |v_t|^{a+1} \dr \mu^{\mathstrut}_{\C}  \dr \m_{\Rm_{\tau}^{\tau_2}} + \mathbb{D}^{\mathrm{N}}_{a}[f]_{\tau_1}^{\tau_2}. $$
Recall now that $\Rm^{\tau_2}_{\tau} $ is foliated by the spacelike hypersurfaces $\Sigma_s $, $s \in [\tau, \tau_2]$. Moreover,
\begin{itemize} 
\item by Lemma \ref{vol}, for all $\tau_1 \leq \tau' \leq \tau_2$, we have $\dr \mu_{\Rm^{\tau_2}_{\tau'}} = \gamma_0(r) \dr \tau \wedge \dr \m_{\Sigma_{\tau}} $ and $C^{-1} \leq \gamma_0 \leq C$ on $]2M,+\infty[$. 
\item According to \eqref{redshiftdef} and Lemma \ref{comparo}, $|v_t| \leq |v^{\mathstrut}_{\mathrm{N}}| \lesssim |v \cdot \n_{\Sigma_{\tau}}| \lesssim |v^{\mathstrut}_{\mathrm{N}}|$ for all $2M < r \leq r_1$. 
\end{itemize}
We then obtain that
\begin{align*}
\int_{\Rm_{\tau}^{\tau_2}} \mathds{1}_{r_0 \leq r \leq r_1}  \int_{\C}  |f|  |v_t|^{a+1} \dr \mu^{\mathstrut}_{\C}  \dr \mu_{\Rm_{\tau}^{\tau_2}} \,&  = \, \int_{s=\tau}^{\tau_2} \int_{\Sigma_s} \mathds{1}_{r_0 \leq r \leq r_1} \, \gamma_0(r)  \int_{\C}   |f|  |v_t|^{a+1} \dr \mu^{\mathstrut}_{\C} \dr \m_{\Sigma_s} \dr s \\ & \lesssim \, \int_{s=\tau}^{\tau_2} \int_{\Sigma_s} \mathds{1}_{r_0 \leq r \leq r_1} \, \rho \Big[ |f| |v_t|^a\Big] \dr \m_{\Sigma_s} \dr s , \\
\int_{\Rm_{\tau}^{\tau_2}} \mathds{1}_{r \leq r_0}  \int_{\C} |f| |v^{\mathstrut}_{\mathrm{N}}|^{a+1} \dr \mu^{\mathstrut}_{\C}  \dr \mu_{\Rm_{\tau}^{\tau_2}} \, & \gtrsim \,  \int_{s=\tau}^{\tau_2} \int_{\Sigma_s} \mathds{1}_{ r \leq r_0} \, \rho \Big[ |f| |v^{\mathstrut}_{\mathrm{N}}|^a\Big] \dr \m_{\Sigma_s} \dr s.
\end{align*}
Consequently, there exist $0 < b' \leq B'$, both depending on $a$, such that
\begin{equation}\label{eq:tregon}
 \phi (\tau_2) + b' \! \int_{s=\tau}^{\tau_2} \int_{\Sigma_s} \! \mathds{1}_{ r \leq r_0} \, \rho \Big[ |f| |v^{\mathstrut}_{\mathrm{N}}|^a \Big] \dr \m_{\Sigma_s} \dr s   \leq  \phi(\tau)+ B' \! \int_{s=\tau}^{\tau_2} \int_{\Sigma_s} \! \mathds{1}_{r_0 \leq r \leq r_1} \, \rho \Big[ |f| |v_t|^a\Big] \dr \m_{\Sigma_s} \dr s + \mathbb{D}^{\mathrm{N}}_{a}[f]_{\tau_1}^{\tau_2}. 
 \end{equation}
This implies the second inequality of the proposition. For the first one, remark that, using $|v^{\mathstrut}_{\mathrm{N}}| \lesssim |v_t|$ for $r \geq r_0$ (see \eqref{defvtbar}) and Lemma \ref{Lemmatechenergy},
$$ \int_{s=\tau}^{\tau_2} \int_{\Sigma_s} \left(  \mathds{1}_{r_0 \leq r \leq r_1} \, \rho \Big[ |f| |v_t|^a\Big] + \mathds{1}_{ r \geq r_0} \, \rho \Big[ |f| |v_{\mathrm{N}}^{\mathstrut}|^a\Big]  \right) \dr \m_{\Sigma_s} \dr s \lesssim  \int_{s=\tau}^{\tau_2} \Big( \int_{\Sigma_{\tau_1}}   \rho \Big[ |f| |v_t|^a\Big] \dr \m_{\Sigma_{\tau_1}}+\mathbb{D}^{\mathrm{N}}_{a}[f]_{\tau_1}^{\tau_2} \Big) \dr s .$$
Hence, adding $b' \int_{s=\tau}^{\tau_2} \int_{\Sigma_s} \mathds{1}_{ r \geq r_0} \, \rho \big[ |f| |v_{\mathrm{N}}^{\mathstrut}|^a\big] \dr \m_{\Sigma_s} \dr s$ to both sides of \eqref{eq:tregon} yields
$$ \phi (\tau_2) + b'\int_{\tau}^{\tau_2} \phi (s) \dr s \,  \leq \, \phi(\tau)+ B_0 (1+\tau_2-\tau) \Big( \int_{\Sigma_{\tau_1}}   \rho \Big[ |f| |v_t|^a\Big] \dr \m_{\Sigma_{\tau_1}}+\mathbb{D}^{\mathrm{N}}_{a}[f]_{\tau_1}^{\tau_2} \Big),$$
for a certain constant $B_0 >0$ which depends on $a$. To conclude the proof, it then only remains to apply Lemma \ref{Lemforred} to $\phi : [\tau_1, \tau_2] \rightarrow \R$, with $c=b'$ and $C=\overline{D}=B_0 \int_{\Sigma_{\tau_1}}   \rho \big[ |f| |v_t|^a\big] \dr \m_{\Sigma_{\tau_1}}+B_0 \mathbb{D}^{\mathrm{N}}_{a}[f]_{\tau_1}^{\tau_2}$.
\end{proof}
\begin{Rq}
Note that a similar non-degenerate energy estimate holds for massive Vlasov fields. In that case, if $m>0$ is the mass of the particles, the distribution function is defined on
\begin{equation*}
 \mathcal{P}_m \hspace{1mm} := \hspace{1mm}  \displaystyle{\bigcup_{x \in \mathcal{M}}} \mathcal{P}_{m,x}, \qquad \mathcal{P}_{m,x} \hspace{1mm} := \hspace{1mm} \left\{ (x,v) \, / \, v \in T^{\star}_{x}\mathcal{M} , \;  \text{$g^{-1}(v,v)=-m^2$ and $v$ future oriented}  \right\},
 \end{equation*} 
so that $ v_t = - \big(\big(1-\frac{2M}{r} \big)m^2+ |v_{r^*}|^2 +\big( 1-\frac{2M}{r}\big) \frac{|\slashed{v}|^2}{r^2} \big)^{\frac{1}{2}}$ in $\mathscr{D}$. One can check that all the computations made in this section can be adapted to this context. The only difference would arise in Lemma \ref{Ty}, since we have for massive particles $\T \big( \frac{2|v_u|}{1-\frac{2M}{r}} \big) \, = \,  -\frac{4M}{r^2}\frac{|v_u|^2}{(1-\frac{2M}{r})^2}+\frac{4}{r}\frac{|v_{\underline{u}} v_u|}{(1-\frac{2M}{r})} - \frac{m^2}{r}$. As the extra term $-\frac{m^2}{r}$ has a good sign, it does not prevent Lemma \ref{Conditions13} to hold in the context of massive particles. In particular, the constants $r_0$, $r_1$, $b$ and $B$ could be chosen to be independant of the mass $m \in \R_+$.
\end{Rq}
\section{Integrated local energy decay}\label{sec3}

Positive bulk integrals for solutions to the free massless Vlasov equation can be generated by using multipliers of the form $\phi (r)\partial_{r^*}$. This corresponds to multiplying the distribution function by the weight $\phi (r) v_{r^*}$. In fact, in order to obtain stronger results, we will also consider weights of the form $\phi(r) \frac{v_{r^*}}{|v_{r^*}|}|v_{r^*}|^{\delta}$.
\begin{Rq}
The benefices of considering $\phi(r) \frac{v_{r^*}}{|v_{r^*}|}|v_{r^*}|^{\delta}$ instead of $\phi (r)v_{r^*}$ in the study of Vlasov fields can be compared with the use of modifications of the current $J^X[\psi]_{\alpha} = T[\psi]_{\alpha \beta} X^{\beta}$, where $X=\phi(r) \partial_{r^*}$, for the study of scalar fields. In particular, the modified currents contain not only first order derivatives of $\psi$, but also lower order term. 
\end{Rq}
Unfortunately, the integrated local energy decay that we will obtain in this way degenerates at the photon sphere $r=3M$. The second part of this section will then consist in dealing with this degeneracy. The results proved in this section will imply in particular the following estimate, which is crucial for performing the $r^p$-weighted energy method. Recall the notation $\langle w \rangle := (1+|w|^2)^{\frac{1}{2}}$ and the norms introduced in Definition \ref{defbulknorm}.
\begin{Pro}[Integrated local energy decay]\label{ILED}
Let $f : \widehat{\Rm}_0^{+\infty} \rightarrow \R$ be a sufficiently regular solution to $\T(f)=0$. For any $a \in \R_+$ and $s>1$, we have for all $0 \leq \tau_1 \leq \tau_2 $,
\begin{align*}
 \int_{\tau_1}^{\tau_2}  \int_{\Sigma_{\tau}} & \mathds{1}^{\mathstrut}_{ r \leq R_0} \,  \rho  \left[ |f| |v^{\mathstrut}_{\mathrm{N}}|^a \right]   \dr \m_{\Sigma_{\tau}} \dr \tau \\
& \quad \lesssim_{a,s}^{\mathstrut}   \int_{\Sigma_{\tau_1}}  \rho \left[ |f| |v^{\mathstrut}_{\mathrm{N}}|^a \right]   \dr \m_{\Sigma_{\tau_1}} + \left| \int_{\Sigma_{\tau_1}}  \rho \left[\left\langle r^{2(s-1)}\frac{|v_{\underline{u}}|}{|v_t|} \right\rangle \langle \slashed{v} \rangle^{4(s-1)}  |f|^s |v_t|^{1+s(a-1)} |v^{\mathstrut}_{\mathrm{N}}|^{s-1}\right] \dr \m_{\Sigma_{\tau_1}} \right|^{\frac{1}{s}} 
\end{align*}
and
\begin{align*}
 \int_{\tau_1}^{\tau_2}  \int_{\Sigma_{\tau}} & \mathds{1}^{\mathstrut}_{ r \leq R_0} \,  \rho  \left[ |f| |v^{\mathstrut}_{\mathrm{N}}|^a \right]   \dr \m_{\Sigma_{\tau}} \dr \tau \\
& \quad \lesssim_{a,s}^{\mathstrut}   \int_{\Sigma_{\tau_1}}  \rho \left[ |f| |v^{\mathstrut}_{\mathrm{N}}|^a \right]   \dr \m_{\Sigma_{\tau_1}} + \left| \int_{\Sigma_{\tau_1}}  \rho \left[\left\langle r^{2(s-1)}\frac{|v_{\underline{u}}|}{|v_t|} \right\rangle \langle v_t \rangle^{4(s-1)}  |f|^s |v_t|^{1+s(a-1)} |v^{\mathstrut}_{\mathrm{N}}|^{s-1}\right] \dr \m_{\Sigma_{\tau_1}} \right|^{\frac{1}{s}}\! .
\end{align*}
In particular, we have $\langle v_t \rangle^{4(s-1)}  |f|^s |v_t|^{1+s(a-1)} |v^{\mathstrut}_{\mathrm{N}}|^{s-1}\lesssim \langle v_t \rangle^{(4+a)(s-1)}  |f|^s |v^{\mathstrut}_{\mathrm{N}}|^a$.
\end{Pro}
\begin{Rq}\label{Rqlabelplustard}
We prove in Proposition \ref{noILED} below that we cannot control the left hand side of these two inequalities by $\int_{\Sigma_{\tau_1}} \! \rho [ |f| |v^{\mathstrut}_{\mathrm{N}}|^a ] \dr \m_{\Sigma_{\tau_1}}$ because of the trapping at the photon sphere. This result has to be compared with the one obtained by Sbierski in the context of wave equations (see \cite[Subsection $3.1.1$]{Sbierski}). Any non-degenerate integrated local energy decay statement for the wave equation on Schwarzschild spacetime has to lose regularity. In fact, Blue-Soffer proved in \cite{BlueSoffer} that it is sufficient to lose an $\epsilon$ of an angular derivative. If we schematically consider that a function $\Psi \in L^1(\R^3)$ behaves as\footnote{This property is of course not satisfied for any $L^1(\R^3)$-function but this is an heuristic discussion.} $\Psi(x)=\mathcal{O}_{|x| \rightarrow +\infty} (|x|^{-3})$, then 
\begin{itemize}
\item $\int_{\Sigma_{\tau_1}} \! \rho [ |f| |v^{\mathstrut}_{\mathrm{N}}|^a ] \dr \m_{\Sigma_{\tau_1}} \! < +\infty$ schematically implies that for $|v_{r^*}|+|\slashed{v}| \to + \infty$, $|f| = \mathcal{O} ( (|v_{r^*}|+|\slashed{v}|)^{-3-a} )$
\item whereas $\int_{\Sigma_{\tau_1}}  \rho \left[\left\langle r^{2(s-1)}\frac{|v_{\underline{u}}|}{|v_t|} \right\rangle \langle \slashed{v} \rangle^{4(s-1)}  |h|^s |v_t|^{1+s(a-1)} |v^{\mathstrut}_{\mathrm{N}}|^{s-1}\right] \dr \m_{\Sigma_{\tau_1}} \! < +\infty$ provides $|f| = \mathcal{O}(r^{-2-\frac{1}{s}})$ and $|f|=\mathcal{O} \big( (|v_{r^*}|+|\slashed{v}|)^{-\frac{3}{s}-a} (1+|\slashed{v}|)^{-4+\frac{4}{s}} \big)$.
\end{itemize} 
In this sense, we have then proved an integrated local energy decay estimate for massless Vlasov fields requiring the loss of an $|\slashed{v}|^{\epsilon}$ of integrability and Proposition \ref{ILED} then has to be compared with \cite[Theorem $8.20$]{BlueSoffer}. 
\end{Rq}

\subsection{Degenerate integrated local energy decay} 

We fix for all this subsection a sufficiently regular function $f : \widehat{\Rm}_0^{+\infty} \rightarrow \R$, which does not necessarily satisfies $\T(f)=0$. The first step of the proof of Proposition \ref{ILED}, and the main goal of this subsection, consists in proving the following result.

\begin{Pro}\label{ILEDdeg4}
For any $\delta >0$, we have, for all $0 \leq \tau_1 \leq \tau_2 $,
$$\int_{\Rm_{\tau_1}^{\tau_2}}\!  \int_{\C}  \frac{1}{r} \! \left(   \frac{|v_{r^*}|^\delta |v_t|^{2-\delta}}{\log^2(3+r)}\! +\! \frac{|r-3M|^{\delta}}{r^{\delta}} \frac{|\slashed{v}|^2}{r^2}  \right) \! |f|  \dr \m_{\C} \dr \m_{\Rm_{\tau_1}^{\tau_2}}  \lesssim^{\mathstrut}_{\delta}  \!  \int_{\Sigma_{\tau_1}} \! \rho \Big[ |f||v_t| \Big] \dr \m_{\Sigma_{\tau_1}}\! +\int_{\Rm_{\tau_1}^{\tau_2}} \! \int_{\C} |\T(f)| |v_t| \dr \m_{\C} \dr \m_{\Rm^{\tau_2}_{\tau_1}}  .$$
\end{Pro}
\begin{Rq}
We can obtain a better control for the region $\{ r \leq r_0 \}$ by combining this inequality with the one of Proposition \ref{energyredshiftbis}.
\end{Rq}
As we will consider functions which are not necessarily nonnegatives, we will not be able to apply Proposition \ref{divergence0}. Instead, we will use the next estimate.

\begin{Lem}\label{divergencebis}
Let $F : \widehat{\Rm}_0^{+\infty} \rightarrow \R$ be a sufficiently regular function such that $|F| \leq D  |f||v_t|$ for a certain constant $D>0$. Then, for all $0 \leq \tau_1 \leq \tau_2 $,
$$\int_{\Rm^{\tau_2}_{\tau_1}} \int_{\C} \T ( F) \, \dr \mu^{\mathstrut}_{\C} \dr \m_{\Rm_{\tau_1}^{\tau_2}}  \, \leq \, 2 D  \int_{\Sigma_{\tau_1}} \rho \Big[ |f| |v_t| \Big]  \dr \mu^{\mathstrut}_{\Sigma_{\tau_1}}   +D\int_{\Rm^{\tau_2}_{\tau_1}} \int_{\C} |\T ( f) ||v_t| \dr \mu^{\mathstrut}_{\C} \dr \m_{\Rm_{\tau_1}^{\tau_2}} .$$
\end{Lem}
\begin{proof}
We follow closely the proof of Proposition \ref{divergence0}, which cannot be applied to $F$. Fix $0 \leq \tau_1 \leq \tau_2$ and consider $w > \tau_2+u_0$, $\underline{w} > w+2R_0^*$. Then, remark that \eqref{eq:divh} can be applied to $h=F$, so that we obtain from the triangle inequality for integrals and $|F| \leq D|f||v_t|$,
\begin{align*}
 \int_{\Rm^{\tau_2}_{\tau_1}} \mathds{1}_{u \leq w} \,  \mathds{1}_{\underline{u} \leq \underline{w}} \, \int_{\C} \T ( F)  \, \dr \mu^{\mathstrut}_{\C} \dr \m_{\Rm_{\tau_1}^{\tau_2}} \leq D\! \int_{\Sigma_{\tau_1}} \mathds{1}_{u \leq w} \,  \mathds{1}_{\underline{u} \leq \underline{w}} \, \rho \big[ |f| |v_t| \big] \dr \mu^{\mathstrut}_{\Sigma_{\tau_1}}\!+D\! \int_{\Sigma_{\tau_2}} \! \mathds{1}_{u \leq w} \,  \mathds{1}_{\underline{u} \leq \underline{w}} \, \rho \big[|f||v_t| \big] \dr \mu^{\mathstrut}_{\Sigma_{\tau_2}} & \\ +D\int_{C_w} \! \mathds{1}_{\tau_1 \leq \tau \leq \tau_2} \! \int_{\C} |f| |v_t| |v \cdot n^{\mathstrut}_{C_w}| \dr \m_{\C} \dr \mu^{\mathstrut}_{C_w} +D \int_{\underline{\N}_{\underline{w}}} \! \mathds{1}_{\tau_1 \leq \tau \leq \tau_2} \! \int_{\C} |f| |v_t|  |v \cdot n^{\mathstrut}_{\underline{\N}_{\underline{w}}}| \dr \m_{\C} \dr \mu^{\mathstrut}_{\underline{\N}_{\underline{w}}}. &
 \end{align*}
Apply \eqref{eq:divh} to $h=|f||v_t|$ and use $|\T(|f||v_t|)|=\big|\frac{f}{|f|}\T(f)|v_t|\big|=|\T(f)||v_t|$ in order to obtain
$$\int_{\Rm^{\tau_2}_{\tau_1}} \mathds{1}_{u \leq w} \,  \mathds{1}_{\underline{u} \leq \underline{w}} \, \int_{\C} \T ( F)  \, \dr \mu^{\mathstrut}_{\C} \dr \m_{\Rm_{\tau_1}^{\tau_2}}  \, \leq \, 2 D \int_{\Sigma_{\tau_1}}  \rho \big[ |f||v_t| \big] \dr \mu^{\mathstrut}_{\Sigma_{\tau_1}}  +D\int_{\Rm^{\tau_2}_{\tau_1}}  \int_{\C}| \T ( f) ||v_t| \dr \mu^{\mathstrut}_{\C} \dr \m_{\Rm_{\tau_1}^{\tau_2}} .$$
It then remains to apply the dominated convergence theorem.
\end{proof}
Furthermore, in order to generate positive bulk integrals, we will use the function $\Phi_{\delta}$ define as follows. 
\begin{Lem}\label{technisobo}
Let $ \delta>0$. The function $\Phi_{\delta} : s \mapsto \frac{s}{|s|} |s|^{\delta}$ is locally in $W^{1,1}(\R)$ and its derivative is almost everywhere equal to $\Phi_{\delta}':s \mapsto \delta |s|^{\delta-1}$.
\end{Lem}
\begin{proof}
Let $\psi \in \mathcal{C}^{\infty}_c (\R)$ and $\epsilon >0$. Then, by integration by parts
$$ \int_{-\infty}^{-\epsilon} -  (-s)^{\delta} \psi'(s) \dr s + \int_{\epsilon}^{+\infty} s^{\delta} \psi'(s) \dr s \, = \, -\epsilon^{\delta} (\psi'(-\epsilon)+ \psi'(\epsilon)) -\int_{|s| \geq \epsilon } \frac{\delta \psi (s) }{|s|^{1-\delta}} \, \dr s.$$
Since $s \mapsto \frac{s}{|s|} |s|^{\delta}$ and $s \mapsto \delta|s|^{\delta-1}$ are locally in $L^1(\R)$, $\epsilon^{\delta} \to_{\epsilon \to 0} 0$ and that $\psi$ and $\psi'$ are bounded and compactly supported, the dominated convergence theorem gives
$$ \int_{s \in \R} \frac{s}{|s|} |s|^{\delta} \psi'(s) \dr s  \, = \, -\int_{s \in \R} \frac{\delta \psi (s) }{|s|^{1-\delta}} \, \dr s .$$
\end{proof}
Let us now state a direct consequence of the previous results before explaining the strategy that we will follow in order to prove a degenerate integrated local energy decay estimate.
\begin{Lem}\label{Bulkphi}
Let $0 < \delta \leq 1$ and $\phi \in \mathcal{C}^1(]2M,+\infty[,\R)$ such that $\left\| \phi \right\|_{L^{\infty}} < + \infty$. Then, for all $0 \leq \tau_1 \leq \tau_2 $,   
\begin{multline*}
\int_{\Rm_{\tau_1}^{\tau_2}} \int_{\C} \left( \phi'(r)|v_{r^*}|^{1+\delta}  +  \delta \phi (r) \frac{(r-3M)}{r^2}\frac{|\slashed{v}|^2}{r^2|v_{r^*}|^{1-\delta}}  \right) |v_t|^{1-\delta} |f| \, \dr \m_{\C} \dr \m_{\Rm_{\tau_1}^{\tau_2}}  \\ \leq \, 2 \| \phi \|_{L^{\infty}}\left(  \int_{\Sigma_{\tau_1}} \rho \Big[ |f| |v_t| \Big] \dr \m_{\Sigma_{\tau_1}} + \int_{\Rm_{\tau_1}^{\tau_2}} \int_{\C} |\T(f)| |v_t| \dr \m_{\C} \dr \m_{\Rm_{\tau_1}^{\tau_2}}  \right)\!.
\end{multline*}
\end{Lem}
\begin{proof}
This follows from Lemma \ref{divergencebis} applied with $F = \phi(r) \Phi_{\delta}(v_{r^*})|v_t|^{1-\delta}  |f|$ and $D= \| \phi \|_{L^{\infty}}$. Indeed, using Lemma \ref{technisobo} and $\T(v_t)=0$, we have
\begin{align*}
\T\left(F \right)  \, & = \, \T ( \phi (r) ) \frac{v_{r^*}}{|v_{r^*}|}|v_{r^*}|^{\delta}|v_t|^{1-\delta}  |f|+\phi (r) \T \left( \Phi_{\delta}(v_{r^*})  \right) |v_t|^{1-\delta} |f| +\phi(r) \frac{v_{r^*}}{|v_{r^*}|}|v_{r^*}|^{\delta}|v_t|^{1-\delta}  \T ( |f|) \\ & \geq \, \phi'(r)|v_{r^*}|^{1+\delta}|v_t|^{1-\delta}  |f|+ \phi(r)\frac{(r-3M)|\slashed{v}|^2}{r^4}\delta|v_{r^*}|^{\delta-1}|v_t|^{1-\delta} |f| - \| \phi\|_{L^{\infty}} |v_t| |\T(f)|.
\end{align*}
\end{proof}
The goal now is to find a function $\phi \in L^{\infty}$ such that
\begin{itemize}
\item the integrand of the term in the left hand side of the inequality of Lemma \ref{Bulkphi} is nonnegative. More precisely, we would like $\phi'$ to be strictly positive, $(r-3M) \phi(r) \geq 0$ and $(r-3M) \phi(r) \geq C_{\eta} >0$ for all $|r-3M| \geq \eta >0$.
\item Moreover, we would like that $\phi' \to_{r \to +\infty} 0$ with a slow decay rate and that $ (r-3M) \phi(r) \sim_{r \to 3M} (r-3M)^{\eta}$ with $\eta$ as small as possible.
\end{itemize}
Note that the choice $\delta=1$ and $\phi(r)=\log^{- 1} (3+3M)-\log^{-1}(3+r)$ checks all the conditions but the last one since $r \mapsto (r-3M) \phi'(r)$ vanishes quadratically at the photon sphere. Although it is not possible to remove completely this quadratic degeneracy, we will prove, as stated in Proposition \ref{ILEDdeg4}, that we can considerably reduce it. For this, the rough idea is to apply the previous lemma for arbitrary small $\delta$ and $\phi(r)\sim_{r \to 3M} \Phi_{\delta} (r-3M)$. Let us start by proving an integrated local energy decay estimate with a strong degeneracy at $r=3M$.
\begin{Pro}\label{ILEDdeg}
We have, for all $ 0 \leq \tau_1 \leq \tau_2 $,
$$\int_{\Rm_{\tau_1}^{\tau_2}} \! \int_{\C} \! \frac{1}{r} \left(   \frac{|v_{r^*}|^2}{ \log^2(
3+r)} \! + \! \frac{(r-3M)^2}{r^2 }  \frac{|\slashed{v}|^2}{r^2} \right) \! |f| \, \dr \m_{\C} \dr \m_{\Rm_{\tau_1}^{\tau_2}}  \lesssim   \int_{\Sigma_{\tau_1}} \rho \Big[ |f| |v_t| \Big] \dr \m_{\Sigma_{\tau_1}} \! + \! \int_{\Rm_{\tau_1}^{\tau_2}} \! \int_{\C} \! |\T(f)| |v_t| \dr \m_{\C} \dr \m_{\Rm^{\tau_2}_{\tau_1}}  .$$
\end{Pro}
\begin{proof}
Apply Lemma \ref{Bulkphi} with $\delta=1$, $\phi(r)=\log^{- 1} (3+3M)-\log^{-1}(3+r)$ and note that
$$\phi'(r) \, = \, \frac{1}{(3+r) \log^2(3+r)} \gtrsim \frac{1}{r \log^2(3+r)} .$$
It then only remains to prove that
\begin{equation}\label{kevatalenn2}
 \forall \, r >2M, \qquad (r-3M) \phi (r)  \hspace{1mm} \gtrsim \hspace{1mm}   \frac{(r-3M)^2}{r} . 
\end{equation}
We treat first the region close to $r=3M$. By Taylor-Lagrange inequality,
$$ \forall \, |r-3M| < M, \hspace{1cm} \left| \phi (r)- \partial_r \phi(3M) (r-3M) \right| \leq \frac{1}{2} \sup_{]2M,4M]} |\partial_r^2 \phi | \cdot (r-3M)^2.$$
As
$$\partial_r \phi(3M) \, = \,   (3+3M)^{-1} \log^{-2}(3+3M) \, > \,0 ,  \qquad \sup_{[2M,4M]} |\partial_r^2 \phi| \, \leq \, (3+2M)^{-2} < + \infty$$
and since $\frac{2M}{r} \leq 1$ for $r \in ]2M,4M]$, there exists a constant $0<\eta<M$ such that
$$ \forall \, |r-3M| < \eta, \hspace{1cm}  (r-3M)\phi (r) \, \geq \, \frac{\partial_r \phi(3M)}{2}(r-3M)^2 \, \gtrsim \,  \frac{(r-3M)^2}{r}.$$
This implies \eqref{kevatalenn2} for $|r-3M| < \eta$. For the remaining region, we use that $(r-3M) \phi$ is nonnegative and then that $\phi$ strictly increases and vanishes at $r=3M$ in order to get
$$ \forall \, |r-3M| \geq \eta, \qquad (r-3M) \phi (r) \, = \, |r-3M||\phi(r)| \geq |r-3M| | \min \left( \phi (3M-\eta), \phi (3M+\eta) \right)| \gtrsim |r-3M|.$$
This leads to \eqref{kevatalenn2} for $|r-3M| \geq \eta$ since $|r-3M| \gtrsim \frac{(r-3M)^2}{r}$ in this region.
\end{proof}
We now improve the estimate near the photon sphere $r = 3M$.
\begin{Pro}\label{ILEDdeg2}
For any $0<\delta_1 \leq 1$ and $0 <\delta_2 \leq 1$, we have, for all $ 0 \leq \tau_1 \leq \tau_2 $,
\begin{multline*}
 \int_{\mathcal{R}_{\tau_1}^{\tau_2}}  \int_{\mathcal{P}^{\mathstrut}}  \mathds{1}_{r \leq R_0} \left(   \frac{|v_{r^*}|^{1+\delta_1}}{|r-3M|^{1-\delta_2}} + |r-3M|^{1+\delta_2} \frac{|\slashed{v}|^2}{r^2|v_{r^*}|^{1-\delta_1}} \right) \! |v_t|^{1-\delta_1} |f| \, \dr \mu_{\mathcal{P}} \mathrm{d} \mu_{\mathcal{R}_{\tau_1}^{\tau_2}} \\ \lesssim^{\mathstrut}_{\delta_1,\delta_2} \,  \int_{\Sigma_{\tau_1}} \rho \Big[ |f| |v_t| \Big] \dr \m_{\Sigma_{\tau_1}}  +  \int_{\Rm_{\tau_1}^{\tau_2}} \int_{\C} |\T(f)| |v_t| \dr \m_{\C} \dr \m_{\Rm^{\tau_2}_{\tau_1}}  .
\end{multline*}  
\end{Pro}
\begin{proof}
Let $\chi \in \mathcal{C}^{\infty}(\R,\R_+)$ be a nonnegative cutoff function such that $\chi (s)=1$ for all $s \leq R_0$ and $\chi(s) =0$ for all $s \geq 2R_0$. Consider $(\delta_1,  \delta_2 ) \in ]0,1]^2$ and $\phi$ defined by $\phi (r) = \chi (r)\Phi_{\delta_2} (r-3M)$. In view of the definition of $\chi$ and using Lemma \ref{technisobo}, we have for all $r>2M$,
\begin{align*}
 &\frac{r-3M}{r^2}\phi(r)   =  \chi (r) \frac{|r-3M|^{1+\delta_2}}{r^2}  \geq  \mathds{1}^{\mathstrut}_{r \leq R_0} |r-3M|^{1+\delta_2}, \\
 &\phi'(r)   =  \chi(r) \Phi_{\delta_2}'(r-3M)+\chi'(r) \Phi_{\delta_2}(r-3M)  \, \geq \, \delta_2 \mathds{1}_{r \leq 2R_0} |r-3M|^{\delta_2-1}\! -\mathds{1}^{\mathstrut}_{R_0 \leq r \leq 2R_0} \| \chi' \|^{\mathstrut}_{L^{\infty}}|2R_0-3M|^{\delta_2} \!  .
 \end{align*}
The result then ensues from Lemma \ref{Bulkphi}, applied to $\phi$ with $\delta=\delta_1$, provided that
$$ \int_{\Rm^{\tau_2}_{\tau_1}}  \int_{\C}  \mathds{1}^{\mathstrut}_{R_0 \leq r \leq 2R_0}|v_{r^*}|^{1+\delta_1} |v_t|^{1-\delta_1}  |f| \dr \m_{\C} \dr \m_{\Rm^{\tau_2}_{\tau_1}}  \lesssim \int_{\Sigma_{\tau_1}}  \rho \big[  |f| |v_t|  \big] \dr \m_{\Sigma_{\tau_1}} \! +  \int_{\Rm_{\tau_1}^{\tau_2}}  \int_{\C}  |\T(f)| |v_t| \dr \m_{\C} \dr \m_{\Rm^{\tau_2}_{\tau_1}} .$$
holds. This last inequality follows from Proposition \ref{ILEDdeg} since, for $3M<R_0 \leq r \leq 2R_0$, we have
$$ |v_{r^*}|^{1+\delta_1}|v_t|^{1-\delta_1} \, \lesssim \, |v_t|^2 \, = \, |v_{r^*}|^2+\left(1-\frac{2M}{r}\right) \frac{|\slashed{v}|^2}{r^2} \, \lesssim \, \frac{|v_{r^*}|^2}{r \log (3+r)}+ \frac{(r-3M)^2}{r^3}\frac{|\slashed{v}|^2}{r^2}.$$
\end{proof}
From this, we can deduce the following inequality which, combined with Proposition \ref{ILEDdeg}, implies the integrated local energy estimate of Proposition \ref{ILEDdeg4}. Recall that the constant $r_0$, introduced in Proposition \ref{energyredshift}, satisfied $2M < r_0 < 3M$ and that $R_0 >3M$.
\begin{Cor}\label{ILEDdeg3}
For any $\delta >0$, there holds, for all $ 0 \leq \tau_1 \leq \tau_2 $,
\begin{align*}
\int_{\Rm_{\tau_1}^{\tau_2}}  \int_{\C} \mathds{1}^{\mathstrut}_{r_0 \leq r \leq R_0} \left( \frac{|v_{r^*}|^\delta |v_t|^{2-\delta}}{r \log^2(3+r)}+ \frac{|r-3M|^{\delta}}{r^{1+\delta}} \frac{|\slashed{v}|^2}{r^2} \right)&|f|  \dr \m_{\C} \dr \m_{\Rm_{\tau_1}^{\tau_2}} \\
&  \lesssim^{\mathstrut}_{\delta}   \int_{\Sigma_{\tau_1}}  \rho \Big[ |f| |v_t| \Big] \dr \m_{\Sigma_{\tau_1}} +\int_{\Rm_{\tau_1}^{\tau_2}}  \int_{\C} |\T(f)| |v_t| \dr \m_{\C} \dr \m_{\Rm^{\tau_2}_{\tau_1}} .
\end{align*}
\end{Cor}
\begin{proof}
We start by dealing with the angular component $\slashed{v}$. Fix $\delta >0$ and consider $p>2$ as well as its conjugate exponent $q=\frac{p}{p-1}<2$. On the region $r_0 \leq r \leq R_0$, we have $\frac{|\slashed{v}|}{r^2} \lesssim |v_t|$ and $r^{-1-\delta} \lesssim 1$, so
$$\frac{|r-3M|^{\delta}}{r^{1+\delta}}\frac{|\slashed{v}|^2}{r^2} \, \lesssim^{\mathstrut}_{\delta} \, |r-3M|^{\frac{1+\delta}{2}} \frac{|\slashed{v}|^{\frac{2}{q}}|v_t|^{\frac{1}{2}}}{r^{\frac{2}{q}}|v_{r^*}|^{\frac{1}{2}}} \cdot \frac{|v_t|^{\frac{2}{p}-\frac{1}{2}}|v_{r^*}|^{\frac{1}{2}}}{|r-3M|^{\frac{1-\delta}{2}}} \, \leq \, \frac{1}{q} |r-3M|^{q\frac{1+\delta}{2}} \frac{|\slashed{v}|^2|v_t|^{\frac{q}{2}}}{r^2|v_{r^*}|^{\frac{q}{2}}}+\frac{1}{p} \frac{|v_t|^{2-\frac{p}{2}}|v_{r^*}|^{\frac{p}{2}}}{|r-3M|^{p\frac{1-\delta}{2}}}. $$
Note now that if $p=2+2\delta$, we have $p\frac{1-\delta}{2}=1-\delta^2 <1$ and $q\frac{1+\delta}{2}=1+\frac{\delta^2}{1+2\delta} >1$. We then obtain the result from this last inequality and by applying Proposition \ref{ILEDdeg2}, first with $\delta_1= 1-\frac{q}{2}$, $\delta_2 = q\frac{1+\delta}{2}-1$ and then with $\delta_1= \frac{p}{2}-1$, $\delta_2 = 1-p\frac{1-\delta}{2}$.

For the radial component $v_{r^*}$, remark that on $\{r_0 \leq r \leq R_0 \}$,
$$ \frac{|v_{r^*}|^\delta |v_t|^{2-\delta}}{r \log^2(3+r)}  \lesssim^{\mathstrut}_{\delta}  \frac{|v_{r^*}|^2}{|r-3M|^{\frac{1}{2}}}+|r-3M|^{\frac{\delta}{4-2\delta}}|v_t|^2 \lesssim \frac{|v_{r^*}|^2}{|r-3M|^{\frac{1}{2}}}+|r-3M|^{\frac{\delta}{4}}\frac{|\slashed{v}|^2}{r^{3+\frac{\delta}{4}}}.$$
It remains to apply Proposition \ref{ILEDdeg2}, for $\delta_1=1$ and $\delta_2=1/2$, as well as the result for the angular component $\slashed{v}$ that we just proved.
\end{proof}

All the estimates proved in this subsection degenerate for the angular component $|\slashed{v}|^2 |f|$ for $r \approx 3M$, reflecting that there exist trapped null geodesics orbiting on the photon sphere $r=3M$. This problem also appears for the wave equation and can be solved by loosing an $\epsilon$ of an angular derivative \cite{BlueSoffer}. Moreover, this loss is necessary \cite[Subsection $3.1.1$]{Sbierski} and we prove in the following proposition a similar result for the massless Vlasov equation.
\begin{Pro}\label{noILED}
For all $n \in \mathbb{N}^*$, there exists a smooth solution $f_n \in \Rm_0^{+\infty} \rightarrow \R$ to the massless Vlasov equation $\T(f_n)=0$ and $T_n \in \R_+$ such that
$$ \int_{\Rm_0^{T_n}} \mathds{1}_{|r^*| \leq 1} \rho \Big[ |f_n| |v_t| \Big] \dr \m_{\Rm_0^{T_n}}  \, \geq \, n \int_{\Sigma_0} \rho \Big[ |f_n| |v_t| \Big] \dr \m_{\Sigma_0}.$$
This implies that a non-degenerate integrated local energy decay statement for massless Vlasov fields has to lose integrability.
\end{Pro}
\begin{proof}
Given a $\mathcal{C}^1$ curve $c : [0,T[ \rightarrow \M$, we denote by $\vec{c}$ its velocity vector and we define the associated covector $ c^\star (s) := g(\vec{c}(s), \cdot)$. Since in this article the Vlasov fields are defined on $\C \subset T^{\star}\M$, we rather work with $c^{\star}$ instead of $\vec{c}$. Note that $c$ is a null geodesic future oriented, i.e. $\ddot{c}^{\alpha}+\Gamma^{\alpha}_{\beta \xi} \dot{c}^{\beta} \dot{c}^{\xi}=0$, $g(\vec{c},\vec{c})=0$ and $\vec{c}$ is future oriented, if and only if $c$ satisfies
\begin{equation}\label{ode}
 c^{\star} \in \C,  \qquad \frac{\dr  c^{\star}_{\alpha}}{ \dr s} + \frac{1}{2} \partial_{x^{\alpha}} g^{ \beta \xi}  c^{\star}_{\beta}  c^{\star}_{\xi} =0 ,\qquad \text{where} \quad c^{\star} = g^{-1} (\vec{c}, \cdot )=c^{\star}_t \dr t+c^{\star}_{r^*} \dr r^*+c^{\star}_{\theta} \dr \theta + c^{\star}_{\varphi} \dr \varphi.
 \end{equation}
For simplicity, we have denoted $(t,r^*, \theta, \varphi)$ by $(x^0,\dots, x^3)$. This means that $(c,c^\star_{r^*},c^{\star}_{\theta},c^{\star}_{\varphi})$ is a characteristic of the massless Vlasov equation.
Since the obstruction for a non-degenerate integrated local energy decay to hold comes from the trapping at the photon sphere $r^*=0$, we consider the trapped null geodesic 
$$\gamma : t \mapsto (t, 0, \frac{\pi}{2}, (27M^2)^{-\frac{1}{2}} t), \qquad \gamma^\star(t) = \frac{1}{3} \dr t+\sqrt{3}M \dr \varphi .$$ 
The idea is to consider solutions approaching the time dependant distribution 
$$ F(t)  \hspace{1mm} := \hspace{1mm} \delta^{\mathstrut}_{r^*=0} \otimes  \delta^{\mathstrut}_{\theta=\frac{\pi}{2}} \otimes  \delta_{\varphi= (27M^2)^{-\frac{1}{2}} t} \otimes  \delta^{\mathstrut}_{v_{r^*}=0} \otimes  \delta^{\mathstrut}_{v_{\theta}=0} \otimes \delta_{v_{\varphi}=\sqrt{3}M} \in \mathcal{D} \left( \R_t \times \R_{r^*} \times \mathbb{S}^2 \times \left( \R^3 \setminus \{0 \} \right) \right).$$
Recall from Lemma \ref{vol} that $\dr \m_{\Rm_0^T} = \gamma_0(r) \dr \tau \wedge \dr \m_{\Sigma_{\tau}}$ with $C^{-1} \leq \gamma_0 \leq C$ on $]2M,+\infty[$. Then, for any function $h : \widehat{\Rm}^{T}_0 \rightarrow \R$ solution to the massless Vlasov equation and supported in $\{-1 \leq r^* \leq 1\}$, we have, according to the energy inequality of Proposition \ref{divergence0},
$$\int_{\Rm_0^{T}} \! \mathds{1}_{|r^*| \leq 1} \, \rho \Big[ |h| |v_t| \Big] \dr \m_{\Rm^{T}_0} \, \geq \, C^{-1}  \! \int_{\tau=0}^{T} \int_{\Sigma_{\tau} }  \mathds{1}_{|r^*| \leq 1} \, \rho \Big[ |h| |v_t| \Big] \dr \m_{\Sigma_{\tau}} \dr \tau \, = \,  \frac{T}{C} \cdot \int_{\Sigma_0}  \rho \Big[ |h| |v_t| \Big] \dr \m_{\Sigma_0}  .$$
Moreover, according to Lemma \ref{vol}, we have $t-t_0 \leq \tau$ for all $|r^*| \leq 1$, where $t_0:=1+|u_0|+\|\underline{U}\|_{L^{\infty}}$. The proposition is then implied by the following assertion. For all $n \in \mathbb{N}^*$, there exists a smooth solution $f_n \in \widehat{\Rm}_0^{+\infty} \rightarrow \R$ to the massless Vlasov equation $\T(f_n)=0$ such $f_n(t,\cdot)$ is supported in $\{-1 \leq r^* \leq 1 \}$ for all $t \in [t_0,n]$. 

Let $\chi \in \mathcal{C}^{\infty} (\R,\R_+)$ be a function satisfying $\chi(0)=1$ and $\chi(s)=0$ if $|s| \geq 1$. Consider, for all $0 < \epsilon \leq 1$, the unique solution $h_{\epsilon}$ to the massless Vlasov equation such that\footnote{Note that there exist constants $C_{\epsilon}$ such that, in the sense of the distributions, $C_{\epsilon}h_{\epsilon}(t_0,\cdot) \to_{\epsilon \to 0} F(t_0)$. }
$$h_{\epsilon}(t_0,r^*,\theta,\varphi,v_{r^*},v_{\theta},v_{\varphi})  \hspace{1mm} = \hspace{1mm} \chi \left( \frac{r^*}{\epsilon} \right) \chi \left( \frac{\theta-\frac{\pi}{2}}{\epsilon} \right) \chi \left( \frac{ \varphi}{\epsilon} \right) \chi \left( \frac{v_{r^*}}{\epsilon} \right) \chi \left( \frac{v_{\theta}}{\epsilon} \right) \chi \left( \frac{ v_{\varphi}-\sqrt{3}M}{\epsilon} \right).$$
Since $h_{\epsilon}$ solves the massless Vlasov equation on $J^+( \{ t=t_0\}) \cap \mathscr{D}$, $h_{\epsilon}$ is conserved along any future oriented null geodesic $c : [t_0,+\infty[ \rightarrow \M$ such that $t(c(t_0))\geq t_0$. More precisely,
\begin{equation}\label{caracmethod}
 \forall \, s \in [t_0,+\infty[, \qquad h_{\epsilon}(c(s),c^{\star}_{r^*}(s),  c^{\star}_{\theta} (s),  c^{\star}_{\varphi} (s) ) \, = \, h_{\epsilon}(c(t_0), c^{\star}_{r^*}(t_0), c^{\star}_{\theta} (t_0), c^{\star}_{\varphi} (t_0) ) .
 \end{equation}
In particular, $h_{\epsilon}$ is equal to $1$ along $\gamma$ and we will prove that the support of $h_{\epsilon}$ stay localized around $\gamma$ during a time $T_{\epsilon}$ satisfying $T_{\epsilon} \to_{\epsilon \to 0} +\infty$. By continous dependence on the initial data of the solutions to the geodesic equations \eqref{ode}, we know that there exists $\delta_n >0$ such that all null geodesics $c$ satisfying $|c(t_0)-\gamma(t_0)|+|c^{\star}(t_0)-\gamma^{\star}(t_0)| \leq \delta_n$ verifies
$$ \forall \, s \in [t_0, n], \qquad |c(s)-\gamma(s)|+| c^{\star}(s)- \gamma^{\star} (s)| \leq 1. \qquad \text{In particular} \quad |r^*(c(s))| \leq 1.$$
Applying this property to the future oriented null geodesics $c$ satisfying $t(c(t_0))=t_0$, we obtain using \eqref{caracmethod} that $f_n(t, \cdot):=h_{\delta_n}(t,\cdot)$ is supported in $\{-1 \leq r^* \leq 1 \}$ for all $t \in [t_0,n]$. Since $J^+( \{ t=t_0\}) \cap \mathscr{D}$ contains $\widehat{\Rm}^{+\infty}_0$, this concludes the proof.
\end{proof}

\subsection{Proof of Proposition \ref{ILED}}

Let $f : \widehat{\Rm}_0^{+\infty} \rightarrow \R$ be a sufficiently regular solution to $\T(f)=0$. Let further $0 \leq \tau_1 \leq \tau_2$ and $a \in \R_+$. Recall from \eqref{defvtbar} that $|v^{\mathstrut}_{\mathrm{N}}| \lesssim |v_t|$ for $r \geq r_0$. Hence, applying Proposition \ref{energyredshiftbis} and using $r_1 \leq R_0$, we get
$$ \int_{\tau_1}^{\tau_2}\! \int_{\Sigma_{\tau}} \!  \mathds{1}_{r \leq R_0} \, \rho \left[ |f||v^{\mathstrut}_{\mathrm{N}}|^a \right] \dr \m_{\Sigma_{\tau}} \dr \tau  \lesssim^{\mathstrut}_a \int_{\tau_1}^{\tau_2} \! \int_{\Sigma_{\tau}} \! \mathds{1}_{r_0 \leq r \leq R_0} \, \rho \Big[ |f||v_t|^a \Big] \dr \m_{\Sigma_{\tau}} \dr \tau+ \int_{\Sigma_{\tau_1}} \! \rho \left[|f||v^{\mathstrut}_{\mathrm{N}}|^a\right] \dr \mu^{\mathstrut}_{\Sigma_{\tau_1}}.$$
It then remains to deal with the first term on the right hand side of the previous inequality. In order to lighten the notations, we introduce $h=f |v_t|^{a-1}$. As, in view of Lemma \ref{vol}, $|v \cdot n^{\mathstrut}_{\Sigma_{\tau}}| \lesssim |v_t|$ on $\{ r_0 \leq r \leq R_0 \}$, we have
$$\forall \, r_0 \leq r \leq R_0, \qquad  \rho \Big[ |h||v_t| \Big] \, \lesssim \, \int_{\C}  \left( |v_{r^*}|^2+\left( \! 1 - \frac{2M}{r} \! \right) \frac{|\slashed{v}|^2}{r^2} \right) \! |h| \,  \dr \m_{\C} \, \lesssim \,   \int_{\C} \! \left( |v_{r^*}|^2+\mathds{1}_{|v_{r^*}| \leq |\slashed{v}|}^{\mathstrut} \, |\slashed{v}|^2 \right) \! |h|\,  \dr \m_{\C} . $$
Hence, since $ \dr \m_{\Rm_{\tau_1}^{\tau_2}} =  \gamma_0(r) \dr \tau \wedge \dr \m_{\Sigma_{\tau}}$, with $\|1/\gamma_0\|^{\mathstrut}_{L^{\infty}} < +\infty$ (see Lemma \ref{vol}) and $r \log(3+r) \leq R_0 \log (3+R_0)$ on $\{ r_0 \leq r \leq R_0 \}$, we obtain
$$\int_{\tau=\tau_1}^{\tau_2}\int_{\Sigma_{\tau}} \mathds{1}^{\mathstrut}_{r_0 \leq r \leq R_0} \, \rho \Big[ |h| |v_t| \Big] \dr \m_{\Sigma_{\tau}} \dr \tau \, \lesssim \, \int_{\Rm_{\tau_1}^{\tau_2}} \int_{\C} \left(\frac{|v_{r^*}|^2}{r \log (3+r)}+\mathds{1}^{\mathstrut}_{r_0 \leq r \leq R_0}\mathds{1}_{|v_{r^*}| \leq |\slashed{v}|}^{\mathstrut} |\slashed{v}|^2 \right)  |h| \, \dr \m_{\C} \dr \m_{\Rm_{\tau_1}^{\tau_2}}.$$  
Since $|v_t| \leq |v^{\mathstrut}_{\mathrm{N}}|$ and $\T(h) = 0$, an application of Proposition \ref{ILEDdeg4} yields
\begin{equation*}
 \int_{\tau=\tau_1}^{\tau_2} \int_{\Sigma_{\tau}} \mathds{1}_{r \leq R_0} \, \rho \Big[|f||v^{\mathstrut}_{\mathrm{N}}|^a \Big] \dr \m_{\Sigma_{\tau}} \dr \tau  \lesssim  \int_{\Sigma_{\tau_1}} \! \rho \left[|f||v^{\mathstrut}_{\mathrm{N}}|^a\right] \dr \mu^{\mathstrut}_{\Sigma_{\tau_1}}  +  \int_{\Rm_{\tau_1}^{\tau_2}} \int_{\C} \mathds{1}^{\mathstrut}_{r_0 \leq r \leq R_0}\mathds{1}_{|v_{r^*}| \leq |\slashed{v}|}^{\mathstrut} |\slashed{v}|^2 |h| \dr \m_{\C} \dr \m_{\Rm_{\tau_1}^{\tau_2}}. 
\end{equation*}
Note now the identity
$$ \mathds{1}^{\mathstrut}_{r_0 \leq r \leq R_0}\mathds{1}_{|v_{r^*}| \leq |\slashed{v}|}^{\mathstrut} \leq \mathds{1}^{\mathstrut}_{r_0 \leq r \leq R_0}\mathds{1}_{d |v_t| \leq |\slashed{v}| \leq D |v_t|}^{\mathstrut}, \qquad d=\Big(1+\frac{1}{27M^2}\Big)^{-1/2} , \quad D= R_0 \left(1-2M/r_0\right)^{\frac{1}{2}} .$$
The remainder of this section consists in proving the following result, which, together with the last two inequalities, implies Proposition \ref{ILED}.
\begin{Pro}\label{Lemaddi2}
Let $h: \widehat{\Rm}_0^{+\infty} \rightarrow \R$ be a sufficiently regular solution to $\T(h)=0$. Then, for any $s>1$,
\begin{align*}
\forall \, 0 \leq \tau_1 \leq \tau_2, \qquad \mathcal{Q}_{\tau_1}^{\tau_2}&:=\int_{\Rm_{\tau_1}^{\tau_2}} \int_{\C} \mathds{1}^{\mathstrut}_{r_0 \leq r \leq R_0} \mathds{1}_{d |v_t| \leq |\slashed{v}| \leq D |v_t|}^{\mathstrut} |\slashed{v}|^2 |h| \dr \m_{\C} \dr \m_{\Rm_{\tau_1}^{\tau_2}} \\
& \lesssim^{\mathstrut}_{s} \,  \left| \int_{\Sigma_{\tau_1}}  \rho \Big[\mathds{1}_{\frac{d}{2}|v_t| \leq |\slashed{v}| \leq 2D |v_t|}(|v_t|+r^{2(s-1)}|v_{\underline{u}}|) \langle \slashed{v} \rangle^{4(s-1)}  |v_{\mathrm{N}}|^{s-1}|h|^s \Big] \dr \m_{\Sigma_{\tau_1}} \right|^{\frac{1}{s}}  .
\end{align*}
\end{Pro}
One of the idea of the proof is to control the spacetime integral of $|v_{r^*}|^{1-\delta} |h||r-3M|^{-1+\delta}$ over a region which is bounded in space and then to apply Proposition \ref{ILEDdeg2}. Unfortunately, we cannot use the weight $\Phi_{\delta_1}(v_{r^*})\Phi_{\delta_2}(r-3M)$ with $\delta_1 <0$ since it would generate infinite integrals. We then consider a weight function which is singular only at the photon sphere, which is of Lebesgue measure $0$. More precisely, let, for $0 < \delta \leq 1/4$, 
$$w^\delta_\epsilon := \left| \epsilon+\Big(|v_t|^2-\frac{1}{27M^2}|\slashed{v}|^2 \Big)^2 \right|^{-\frac{\delta}{2}} \chi \left( \frac{|\slashed{v}|}{|v_t|} \right), \qquad \epsilon \in \R_+,$$
where $\chi \in C^\infty_c(\R)$ verifies $\chi (s)=1$ for all $d \leq s \leq D$ and $\chi(s)=0$ for all $d/2 \leq s \leq 2D$. Then, the following estimate holds.
\begin{Lem}\label{Lemaddi1}
Let $0 < \delta \leq 1/4$ and $h: \widehat{\Rm}_0^{+\infty} \rightarrow \R$ be a sufficiently regular solution to $\T(h)=0$. Then,
$$
\forall \, 0 \leq \tau_1 \leq \tau_2, \qquad \int_{\Rm_{\tau_1}^{\tau_2}} \int_{\C} \mathds{1}^{\mathstrut}_{r_0 \leq r \leq R_0} \mathds{1}_{d |v_t| \leq |\slashed{v}| \leq D |v_t|}^{\mathstrut} |\slashed{v}|^2 |h| \dr \m_{\C} \dr \m_{\Rm_{\tau_1}^{\tau_2}} \, \lesssim^{\mathstrut}_{\delta} \,  \int_{\Sigma_{\tau_1}} \rho \Big[ w_0  |v_t|^{1+2\delta} |h| \Big] \dr \m_{\Sigma_{\tau_1}}   .$$
\end{Lem}

\begin{proof}
As, for $\epsilon >0$, $w_\epsilon$ is smooth and verifies $\T(w_\epsilon)=0$, we can apply Propositions \ref{ILEDdeg4} and \ref{ILEDdeg2} to the function $w_\epsilon h|v_t|^{2\delta}$. Then Beppo-Levi theorem provides
\begin{equation*}
 \int_{\Rm_{\tau_1}^{\tau_2}}  \int_{\C^{\mathstrut}}  \mathds{1}_{r_0 \leq r \leq R_0} \left(   \frac{|v_{r^*}|^{1+\delta}|v_t|^{1-\delta}}{|r-3M|^{1-\delta}} +  |r-3M|^{\delta}|\slashed{v}|^2 \right) \! w_0 |v_t|^{2\delta} |h| \, \dr \m_{\C} \dr \m_{\Rm_{\tau_1}^{\tau_2}}  \lesssim^{\mathstrut}_{\delta} \,  \int_{\Sigma_{\tau_1}} \rho \Big[ w_0  |v_t|^{1+2\delta} |h| \Big] \dr \m_{\Sigma_{\tau_1}}   .
 \end{equation*}
Note now that, in the region $\{r_0 \leq r \leq R_0 \}\cap \{ d |v_t| \leq |\slashed{v}| \leq D |v_t| \}$,
\begin{align*}
 \frac{|\slashed{v}|^2}{w_0} & \leq |v_{r^*}|^{2\delta}|\slashed{v}|^2+\left( \frac{1}{r^2}\left( 1-\frac{2M}{r} \right)-\frac{1}{27M^2}\right)^{\delta}|\slashed{v}|^{2+2\delta}  \\
& =|v_{r^*}|^{2\delta}|\slashed{v}|^2 +\frac{|r+6M|^{\delta}}{|27M^2r^{3}|^{\delta}}|r-3M|^{2\delta}|\slashed{v}|^{2+2\delta} \lesssim^{\mathstrut}_\delta |v_t|^{2\delta} \left( |v_{r^*}|^{2\delta}|\slashed{v}|^{2-2\delta}+|r-3M|^{2\delta}|\slashed{v}|^{2} \right).
\end{align*}
Finally, use Young's inequality for products, with the conjugate exponents $\frac{1+\delta}{2\delta}$ and $\frac{1+\delta}{1-\delta}$, in order to get
$$ \mathds{1}_{r_0 \leq r \leq R_0} \mathds{1}_{d |v_t| \leq |\slashed{v}| \leq D |v_t|}^{\mathstrut} |\slashed{v}|^2 \lesssim^{\mathstrut}_\delta |v_t|^{2\delta}w_0 \left( \frac{|v_{r^*}|^{1+\delta}|\slashed{v}|^{1-\delta}}{|r-3M|^{1-\delta}}+|r-3M|^{2\delta}|\slashed{v}|^{2} \right).$$
\end{proof}
We will now remove the singular weight $w_0$ through a Hölder inequality. However, if the function $h$ is not compactly supported, one step more is required in order to control the $L^1$ norm of $w_0h$ by an $L^s$ norm of $h$ that we can bound uniformly in time. For this, since this $L^s$ norm will carry a small $r$-weight, we would like it to be weighted by the component $v_{\underline{u}}$ as well.

\begin{refproof}[Proof of Proposition \ref{Lemaddi2}.]
Fix $\delta >0$ such that $\frac{s\delta}{s-1}<1$. Let $n \geq 2$ and $\tau_1 =s_0<s_1 <\dots < s_n = \tau_2$ be a partition of $[\tau_1,\tau_2]$ such that $s_{i+1}-s_i \leq 1$ for all $0 \leq i \leq n-1$. Consider further a cutoff function 
$$\widehat{\chi} \in C_c^\infty (\R), \qquad \forall \, y \leq R_0+2, \; \widehat{\chi}(y) =1, \qquad \forall \, y \geq R_0+3, \; \widehat{\chi}(y)=0.$$
We construct iteratively two sequences of solutions to the Vlasov equation, $(h_i)_{0 \leq i \leq n-1}$ and $(k_i)_{0 \leq i \leq n-1}$, as follows.
\begin{itemize}
\item $h_0$ is the solution to $\T(f)=0$ satisfying $h_0\vert_{\Sigma_{s_0}}=(\widehat{\chi}(r)|h| )\vert_{\Sigma_{s_0}}$ and $k_0:=|h|-h_0$.
\item Let $N \leq n-2$ and assume that we have constructed $(k_i)_{0 \leq i \leq N}$. We define $h_{N+1}$ as the unique solution to the massless Vlasov equation such that $h_{N+1}\vert_{\Sigma_{s_{N+1}}}=(\widehat{\chi}(r)k_N)\vert_{\Sigma_{s_{N+1}}}$ and $k_{N+1}:=k_N-h_{N+1}$.
\end{itemize}
Remark then that the following properties hold. For all $0 \leq i \leq n-1$,
\begin{enumerate}
\item $|h|=k_i+\sum_{0 \leq j \leq i} h_j$, where $k_i$ and each $h_j$ are nonnegative functions.
\item Then, on $\widehat{\Rm}_{s_i}^{s_{i+1}} \cap \{ r_0 \leq r \leq R_0\}$, we have $|h|=\sum_{0 \leq j \leq i} h_j$. Indeed, by finite speed of propagation and since $s_{i+1}-s_i \leq 1$, $k_i$ is supported in $\{ r \geq R_0+1 \}$ on the time slab $[s_i,s_{i+1}]$.
\item If $i \geq 1$, $h_{i}$ is supported in $\{ R_0+1 \leq r \leq R_0+3\}$ on $\Sigma_{s_{i}}$. Since $\left(1-\frac{2M}{r}\right)|\slashed{v}|^2=r^2|v_{\underline{u}}v_u|$, we have $|\slashed{v}|^2\lesssim |v_{\underline{u}}||v_t|$ on this domain. As $|v_t| \leq 2d^{-1}|\slashed{v }|$ on the support of $w_0$, we get
$$\int_{\Sigma_{s_{i}}} \rho \Big[ w_0  |v_t|^{1+2\delta} h_{i} \Big] \dr \m_{\Sigma_{s_{i}}} \lesssim \int_{\Sigma_{s_{i}}} \rho \Big[ w_0 |v_{\underline{u}}| |v_t|^{2\delta} h_{i} \Big] \dr \m_{\Sigma_{s_{i}}}.$$
Note now that  $\T(w_0 |v_{\underline{u}}| |v_t|^{2\delta} h_{i} )=\T(|v_{\underline{u}}|)w_0 |v_t|^{2\delta} h_{i}\leq 0$ on $\widehat{\Rm}_{\tau_1}^{s_{i}}$. Indeed, $\T(|v_{\underline{u}} |)=-\frac{r-3M}{r^4}|\slashed{v}|^2$ and $h_{i}$ vanishes for $r \leq 3M \leq R_0+1$. Consequently, the energy estimate of Proposition \ref{divergence0} provides
 $$\int_{\Sigma_{s_{i}}} \rho \Big[ w_0  |v_t|^{1+2\delta} h_{i} \Big] \dr \m_{\Sigma_{s_{i+1}}} \lesssim  \int_{\Sigma_{\tau_1}} \rho \Big[ w_0 |v_{\underline{u}}| |v_t|^{2\delta} h_{i} \Big] \dr \m_{\Sigma_{\tau_1}}, \qquad i \geq 1.$$
\end{enumerate}
We then deduce, using the property $2.$, that
$$ \mathcal{Q}_{\tau_1}^{\tau_2} = \sum_{i=0}^{n-1} \int_{\Rm_{s_i}^{\tau_2}} \int_{\C} \mathds{1}^{\mathstrut}_{r_0 \leq r \leq R_0} \mathds{1}_{d |v_t| \leq |\slashed{v}| \leq D |v_t|}^{\mathstrut} |\slashed{v}|^2 h_i \dr \m_{\C} \dr \m_{\Rm_{s_i}^{\tau_2}}.
$$
Then, combining Lemma \ref{Lemaddi1}, applied to any $h_i$ between times $s_i$ and $\tau_2$, with property $3.$ yields
$$ \mathcal{Q}_{\tau_1}^{\tau_2} \lesssim^{\mathstrut}_\delta \int_{\Sigma_{\tau_1}} \rho \Big[ w_0  |v_t|^{1+2\delta} h_{0} \Big] \dr \m_{\Sigma_{\tau_1}}+\sum_{i=1}^{n-1} \int_{\Sigma_{\tau_1}} \rho \Big[ w_0 |v_{\underline{u}}| |v_t|^{2\delta} h_{i} \Big] \dr \m_{\Sigma_{\tau_1}}.
$$
As $1 \lesssim r^{2-2s}$ on the support of $h_0$ and since $\sum_{0 \leq j \leq n-1} h_i \leq |h|$ by the property $1.$, we finally obtain
\begin{align*}
 \mathcal{Q}_{\tau_1}^{\tau_2} 
 & \lesssim^{\mathstrut}_\delta   \int_{\Sigma_{\tau_1}} \rho \Big[(r^{2-2s}|v_t|+|v_{\underline{u}}|) w_0  |v_t|^{2\delta} |h| \Big] \dr \m_{\Sigma_{\tau_1}} \\
& \leq \left| \int_{\Sigma_{\tau_1}}  \rho \Big[\mathds{1}_{\frac{d}{2}|v_t| \leq |\slashed{v}| \leq 2D |v_t|}(|v_t|+r^{2(s-1)}|v_{\underline{u}}|) \langle \slashed{v} \rangle^{4(s-1)}  |v_{\mathrm{N}}|^{s-1}|h|^s \Big] \dr \m_{\Sigma_{\tau_1}} \right|^{\frac{1}{s}}  \left| \mathcal{K} \right|^{\frac{s-1}{s}}, 
\end{align*}
where
$$ \mathcal{K} :=  \int_{\Sigma_{\tau_1}} \frac{1}{r^2}  \int_{\C} (|v_t|+|v_{\underline{u}}|) |v_t|^{2\frac{s\delta}{s-1}} |w_0|^{\frac{s}{s-1}} \, \frac{|v \cdot n_{\Sigma_\tau}|}{ \langle \slashed{v} \rangle^4 |v_{\mathrm{N}}|}  \dr \m_{\C} \dr \m_{\Sigma_{\tau_1}}.$$
Remark now that
$$|w_0|^{\frac{s}{s-1}} \leq \left|(v_{r^*}-P(r)|\slashed{v}|)(v_{r^*}+P(r)|\slashed{v}|)\right|^{-\frac{s\delta}{s-1}} \mathds{1}_{\frac{d}{2} |v_t| \leq |\slashed{v}| \leq 2D |v_t|}, \qquad P(r):=\left|\frac{(r-3M)^2(r+6M)}{27M^2 r^3} \right|^{\frac{1}{2}}\!.$$
As $s\delta (s-1)^{-1} <1$, $\dr \m_{\C} = r^{-2} \sin^{-1} (\theta) |v_t|^{-1} \dr v_{r^*} \dr v_{\theta} \dr v_{\varphi}$ and $|v \cdot n_{\Sigma_\tau}| \lesssim |v_{\mathrm{N}}|$ by Lemma \ref{comparo}, we get
\begin{align*}
& \int_{\C} (|v_t|+|v_{\underline{u}}|) |v_t|^{2\frac{s\delta}{s-1}} |w_0|^{\frac{s}{s-1}} \, \frac{|v \cdot n_{\Sigma_\tau}|}{\langle \slashed{v} \rangle^4|v_{\mathrm{N}}|}  \dr \m_{\C} \\ & \lesssim  \int_{(v_\theta , v_\varphi ) \in \R^2}  \int_{|v_{r^*}| \leq \frac{2}{d}}  \frac{|\slashed{v}|^{2\frac{s\delta}{s-1}} \dr v_{r^*} \dr v_{\theta} \dr v_{\varphi}}{\left|(v_{r^*}-P(r)|\slashed{v}|)(v_{r^*}+P(r)|\slashed{v}|)\right|^{\frac{s\delta}{s-1}}r^2 \sin (\theta ) \langle \slashed{v} \rangle^4}  \lesssim_{s,\delta}  \int_{(v_\theta , v_\varphi ) \in \R^2} \frac{\dr v_{r^*} \dr v_{\theta} \dr v_{\varphi}}{r^2 \sin (\theta ) \langle \slashed{v} \rangle^{3}} \leq \frac{4}{r^2},
\end{align*}
where, in the last step, we used $|\slashed{v}|^2=|v_\theta|^2+\sin^{-2}(\theta)|v_\varphi|^2$. As $s>1$ and in view of the expression of the volume form of $\Sigma_{\tau_1}$ (see Lemma \ref{vol}), we finally get $\mathcal{K} \lesssim \int_{\Sigma_{\tau_1}} r^{-4} \dr \m_{\Sigma_{\tau_1}} < + \infty.$
 
\end{refproof}

\section{Energy decay estimates}\label{sec4}

We prove in this section decay estimates for the energy $\int_{\Sigma_{\tau}} \rho \big[ |f| |v^{\mathstrut}_{\mathrm{N}}| \big] \dr \m_{\Sigma_{\tau}}$ by adapting the $r^p$-weighted energy method of Dafermos-Rodnianski to the massless Vlasov equation. In the case of wave equations, vector fields of the form $r^p \partial_{\uu}$, $0 \leq p \leq 2$, are used as multipliers\footnote{When applied to the time derivative of the solution, one can in fact extend the method to the range $0 \leq p < 4$ (see \cite[Section $5C$]{Volker}).}. In our case, this corresponds to considering weights of the form $r^p |v_{\uu}|$, $0 \leq p \leq 2$. In fact, assuming enough decay on the initial data, we will prove that stronger results can be obtained by using $r^p |v_{\uu}|^q$, with $0 \leq p \leq 2q$. This reflects that outside the wave zone, one can prove superpolynomial decay for the solutions to the massless Vlasov equation.

Before presenting the main result of this section, we introduce the following notations.
\begin{Def}\label{defsec4}
We denote by $2 \mathbb{Z}$ the set of the even integers and by $ 2 \mathbb{Z}+1$ the set of the odd integers. Moreover, we recall the floor and the ceiling functions, defined for any $x \in \R$ by
$$ \lfloor x \rfloor \, := \, \max \{ n \in \mathbb{Z} \, / \, n \leq x \}, \qquad \qquad  \lceil x \rceil \, := \, \min \{ n \in \mathbb{Z} \, / \, n \geq x \}.$$
\end{Def}
\begin{Pro}[energy decay estimates]\label{energydecaysec4}
Let $(a,p) \in \R_+^2$ and $f: \widehat{\Rm}_0^{+\infty} \rightarrow \R$ be a sufficiently regular solution to the massless Vlasov equation $\T(f)=0$. Then, we have
$$ \forall \, \tau \in \R_+, \qquad \int_{\Sigma_{\tau}} \rho \bigg[ r^p \frac{|v_{\uu}|^{p/2}}{|v_t|^{p/2}} |f| |v^{\mathstrut}_{\mathrm{N}}|^a \bigg] \dr \m_{\Sigma_{\tau}}  \lesssim^{\mathstrut}_{a,p} \int_{\Sigma_{0}} \rho \bigg[ \bigg( 1+r^p \frac{|v_{\uu}|^{p/2}}{|v_t|^{p/2}} \bigg) |f| |v^{\mathstrut}_{\mathrm{N}}|^a \bigg] \dr \m_{\Sigma_{0}}.$$
Consider further $s>1$ satisfying $s^2 \leq 1+\langle 5p \rangle^{-2}$. Then, there holds
\begin{align*}
 \forall \, \tau \in \R_+, \quad \int_{\Sigma_{\tau}} \rho \Big[  |f| |v^{\mathstrut}_{\mathrm{N}}|^a \Big] \dr \m_{\Sigma_{\tau}}  &\lesssim^{\mathstrut}_{a,p}  \frac{1}{(1+\tau)^p}\int_{\Sigma_{0}} \rho \bigg[ \bigg( 1+r^p \frac{|v_{\uu}|^{p/2}}{|v_t|^{p/2}} \bigg) |f| |v^{\mathstrut}_{\mathrm{N}}|^a \bigg] \dr \m_{\Sigma_{0}} \\
 & \; \; \; \; +\frac{1}{(1+\tau)^p}\bigg|\int_{\Sigma_{0}} \rho \bigg[ \bigg( 1+r^p \frac{|v_{\uu}|^{p/2}}{|v_t|^{p/2}} \bigg) |f|^{s^{\lceil p \rceil}} \! \langle v_t \rangle^{(4+a)(s^{\lceil p \rceil}-1)} |v^{\mathstrut}_{\mathrm{N}}|^a \bigg] \dr \m_{\Sigma_{0}}\bigg|^{s^{-\lceil p \rceil}}\! \hspace{-1mm}.
 \end{align*} 
\end{Pro}
\begin{Rq}\label{energydecaysec4Rq}
We obtain by using Hölder's inequality and then Young's inequality for products, both applied for the conjugate exponents $\lambda^{-1}$ and $(1-\lambda)^{-1}$, that, for any $\lambda \in ]0,1[$,
$$ \int_{\Sigma_{\tau}}\! \rho \bigg[ r^{\lambda p} \frac{|v_{\uu}|^{\lambda p/2}}{|v_t|^{\lambda p/2}} |f| |v^{\mathstrut}_{\mathrm{N}}|^a \bigg] \dr \m_{\Sigma_{\tau}} \! \leq  (1+\tau)^{\lambda p} \int_{\Sigma_{\tau}}\! \rho \bigg[  |f| |v^{\mathstrut}_{\mathrm{N}}|^a \bigg] \dr \m_{\Sigma_{\tau}} +\frac{1}{(1+\tau)^{p-\lambda p}} \int_{\Sigma_{\tau}}\! \rho \bigg[ r^p \frac{|v_{\uu}|^{p/2}}{|v_t|^{p/2}} |f| |v^{\mathstrut}_{\mathrm{N}}|^a \bigg] \dr \m_{\Sigma_{\tau}}\!. $$
\end{Rq}

\subsection{Hierarchy of $r^p|v_{\uu}|^q$-weighted energy estimates}

We start with a computation, which will also be useful for the treatment of the region $\tau \leq 0$.
\begin{Lem}\label{Lemhierar}
Let $(p,q) \in \R^2$ satisfying $0 \leq p \leq 2q$. Then, we have almost everywhere
$$ -\T \left( r^p \frac{|v_{\uu}|^q}{|v_t|^q} \right) \, = \,  pr^{p-1}\frac{|v_{\uu}|^{q+1}}{|v_t|^{q}}  +\frac{1}{4}\left(2q-p \right)r^{p-1} \frac{|v_{\uu}|^{q-1}}{|v_t|^{q}} \frac{|\slashed{v}|^2}{r^2}- (3q-p)\frac{M}{2}r^{p-2}  \frac{|v_{\uu}|^{q-1}}{|v_t|^{q}} \frac{|\slashed{v}|^2}{r^2}. $$
Note that in view of \eqref{vu}, $|v_{\uu}|$ does not really carry a negative exponent. When $p=0$, we have in particular
$$ \forall \, r \geq R_0, \qquad -\T \left(  \frac{|v_{\uu}|^q}{|v_t|^q} \right) \, \gtrsim \,  \frac{q}{r} \frac{|v_{\uu}|^{q-1}}{|v_t|^q} \frac{|\slashed{v}|^2}{r^2}. $$
\end{Lem}
\begin{proof}
We will use many times that $v_{\uu}, \, v_u \leq 0$, so that $|v_{\uu}|=-v_{\uu}$ and $|v_u|=-v_u$. Since $\T(v_t)=0$ and $2v_{\uu}=v_t+v_{r^*} $, we get
$$ \T \left( r^p \frac{|v_{\uu}|^q}{|v_t|^q} \right) \, = \, pr^{p-1}\frac{|v_{\uu}|^q}{|v_t|^q}\T \left( r \right)  -\frac{q}{2}r^p \frac{|v_{\uu}|^{q-1}}{|v_t|^q}\T \left( v_{r^*} \right)  .$$
Recall the expression \eqref{defT} of $\T$ and $\partial_{r^*}=\left(1-\frac{2M}{r} \right) \partial_r$. Then, $\T(r)=v_{r^*}=|v_u|-|v_{\uu}|$ and $\T(v_{r^*})= \frac{(r-3M)|\slashed{v}|^2}{r^4}$, so
$$ \T \left( r^p \frac{|v_{\uu}|^q}{|v_t|^q} \right) \, = \, -pr^{p-1}\frac{|v_{\uu}|^{q+1}}{|v_t|^q}+ pr^{p-1}\frac{|v_{\uu}|^{q-1}}{|v_t|^q}|v_{\uu}||v_u|  -\frac{q}{2}r^{p-1} \frac{|v_{\uu}|^{q-1}}{|v_t|^q} \frac{|\slashed{v}|^2}{r^2}+ \frac{3Mq}{2}r^{p-2}  \frac{|v_{\uu}|^{q-1}}{|v_t|^q} \frac{|\slashed{v}|^2}{r^2} .$$
The first identity then ensues from the mass-shell condition $g^{-1}(v,v)=0$, which provides $4|v_{\uu}|| v_u| = \left( 1 - \frac{2M}{r} \right)\frac{|\slashed{v}|^2}{r^2}$. This implies in particular that
$$ -\T \left(  \frac{|v_{\uu}|^q}{|v_t|^q} \right) \, = \,  \frac{q}{2r}\left(1-\frac{3M}{r} \right) \frac{|v_{\uu}|^{q-1}}{|v_t|^q} \frac{|\slashed{v}|^2}{r^2} $$
and the second part of the Lemma follows from $R_0 >3M$.
\end{proof}
We are now ready to prove hierarchized $r^p|v_{\uu}|^q$-weighted estimates.

\begin{Pro}\label{rphierar}
Let $(p,q) \in \R_+$ such that $0 \leq p \leq 2q$ and $h : \widehat{\Rm}_0^{+\infty} \rightarrow \R$ be a sufficiently regular function satisfying $\T(h)=0$. There holds for all $0 \leq \tau_1 \leq \tau_2$,
\begin{align*}
 \int_{\N_{\tau_2}} \int_{\C} r^p \frac{|v_{\uu}|^{q}}{|v_t|^{q}} |h
 |  |v_{\uu}|  \dr \m_{\N_{\tau_2}} +\int_{\tau=\tau_1}^{\tau_2} \int_{\N_{\tau}} \int_{\C} pr^{p-1}\frac{|v_{\uu}|^{q+1}}{|v_t|^{q}}& |h| +\left(2q-p \right)r^{p-1} \frac{|v_{\uu}|^{q-1}}{|v_t|^{q}} \frac{|\slashed{v}|^2}{r^2}|h| \dr \m_{\C} \dr \m_{\N_{\tau}} \dr \tau \\
& \hspace{-2mm} \lesssim^{\mathstrut}_q  \int_{\N_{\tau_1}} \int_{\C} r^p \frac{|v_{\uu}|^{q}}{|v_t|^{q}} |h|  |v_{\uu}|  \dr \m_{\N_{\tau_1}} +\int_{\Sigma_{\tau_1}} \rho \Big[ |h| \Big] \dr \m_{\Sigma_{\tau_1}}.
\end{align*}
\end{Pro}
\begin{Rq}\label{rqcontrol}
As for $r \geq R_0$, $\frac{|\slashed{v}|^2}{r^2} \geq 4\big(1-\frac{2M}{R_0} \big)^{-1} |v_u||v_{\uu}|$, Proposition \ref{rphierar} implies that for any $0 < p < 2q$,
$$
\int_{\tau=\tau_1}^{\tau_2} \int_{\N_{\tau}} \int_{\C} r^{p-1}\frac{|v_{\uu}|^{q-1}}{|v_t|^{q-1}} |h| |v_{\uu}| \dr \m_{\C} \dr \m_{\N_{\tau}} \dr \tau 
\lesssim_{p,q}^{\mathstrut} \int_{\N_{\tau_1}} \int_{\C} r^p \frac{|v_{\uu}|^{q}}{|v_t|^{q}} |h|  |v_{\uu}|  \dr \m_{\N_{\tau_1}} +\int_{\Sigma_{\tau_1}} \rho \Big[ |h| \Big] \dr \m_{\Sigma_{\tau_1}} .$$
By making use of these hierarchies in $(p,q)$, we will then be able to convert the weight $r|v_{\uu}|$ into $\tau$ decay, providing us in particular time decay for the region of bounded $r$. In fact, by fully exploiting the hierarchies given by the inequality of Proposition \ref{rphierar}, we will extract a decay rate of $\tau^{-2}$ from $r^2 |v_{\uu}|$.
\end{Rq}
\begin{Rq}\label{energydecayintRq}
For the purpose of establishing pointwise decay estimates, we will also prove that
$$ \sup_{\underline{w} \geq \tau_2+u_0+2R_0^*} \int_{\underline{\N}_{\underline{w}}} \mathds{1}_{\tau_1 \leq u-u_0 \leq \tau_2} \int_{\C} r^p \frac{|v_{\uu}|^q}{|v_t|^q} |h||v_{u}| \dr \m_{\C} \dr \m_{\underline{\N}_{\underline{w}}} 
\lesssim_{q}^{\mathstrut} \! \int_{\N_{\tau_1}} \int_{\C} r^p \frac{|v_{\uu}|^{q}}{|v_t|^{q}} |h|  |v_{\uu}|  \dr \m_{\N_{\tau_1}} +\int_{\Sigma_{\tau_1}} \rho \Big[ |h| \Big] \dr \m_{\Sigma_{\tau_1}}.$$
\end{Rq}
\begin{proof}
Let $0 \leq \tau_1 \leq \tau_2$ and introduce
$$ \mathcal{D}_{\tau_1}^{\tau_2} \, := \, \Rm_{\tau_1}^{\tau_2} \cap \{ r \geq R_0 \} \, = \,  \bigsqcup_{\tau_1 \leq \tau \leq \tau_2} \N_{\tau} \, = \, \{ (t,r^*,\omega) \in \R \times \R \times \mathbb{S}^2 \, /  \, r^* \geq R_0^*, \, \tau_1 \leq t-r^*-u_0 \leq \tau_2  \}.$$
Using Lemma \ref{vol}, we have $\dr \m_{\mathcal{D}_{\tau_1}^{\tau_2}} = \gamma_0(r) \dr \tau \wedge \N_{\tau}$ where $0<C^{-1} \leq \gamma_0 \leq C$, so, for any sufficiently regular function $H: \widehat{\Rm}_{\tau_1}^{\tau_2} \rightarrow \R$,
\begin{equation}\label{equiint}
 \int_{\tau=\tau_1}^{\tau_2} \int_{\N_{\tau}} \int_{\C} |H| \dr \m_{\C} \dr \m_{\N_{\tau}} \dr \tau \, \lesssim \, \int_{\mathcal{D}_{\tau_1}^{\tau_2}} |H| \dr \m_{\C} \dr \m_{\mathcal{D}_{\tau_1}^{\tau_2}}  \, \lesssim \, \int_{\tau=\tau_1}^{\tau_2} \int_{\N_{\tau}}\int_{\C} |H| \dr \m_{\C} \dr \m_{\N_{\tau}} \dr \tau.
\end{equation}
Remark also that for any $\underline{w} \geq \tau_2+u_0+2R_0^*$, the boundary of $\mathcal{D}_{\tau_1}^{\tau_2} \cap \{ \uu \leq \underline{w} \}$ is composed by
\begin{align*}
\mathcal{N}_{\tau_1} \cap \{ \uu \leq \underline{w} \}, \qquad  \mathcal{N}_{\tau_2} \cap \{ \uu \leq \underline{w} \}, \qquad  \underline{\N}_{\underline{w}} \cap \{ \tau_1 \leq  u-u_0 \leq \tau_2 \}, \qquad   \{ r=R_0 \} \cap \{  \tau_1 \leq u-u_0 \leq \tau_2 \}.
\end{align*}
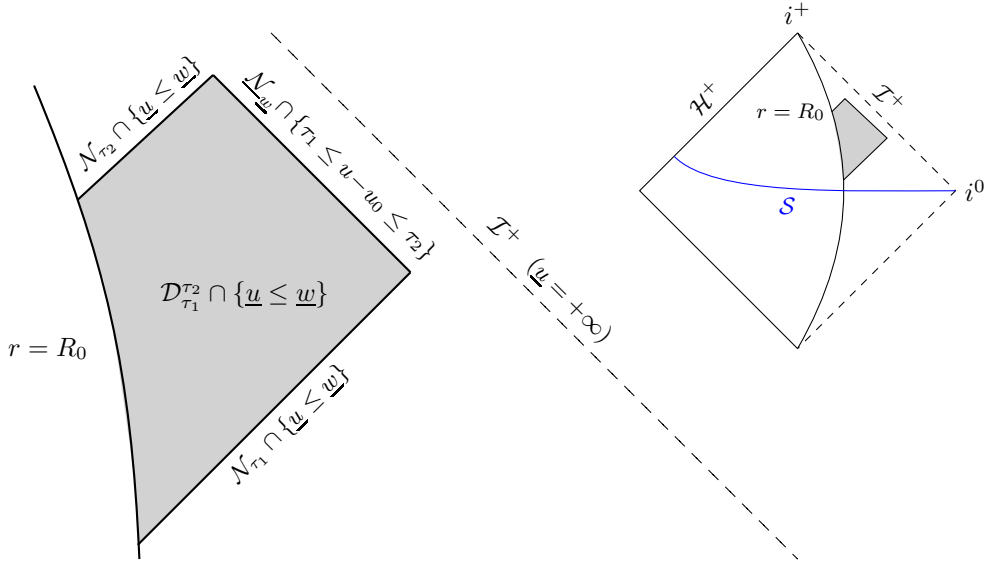
\begin{figure}[H]
\begin{center}
\begin{tikzpicture}[scale=0.7]
\node (II)   at (-0.5,0)   {};
\path  
  (II) +(90:10)  coordinate  (IItop)
       +(0:10)   coordinate  (IIright)
       ;
\draw[loosely dashed] 
      (IItop) --
          node[midway, above, sloped] {$\cal{I}^+$ \, \small{($\underline{u}=+\infty$)}}
      (IIright) -- cycle;
\fill[color=gray!35] (-3.03,0.27)--(2.15,5.45)--(-1.6,9.2)--(-4.17,6.83)--(-3.68,5.3)--(-3.53,4.5)--(-3.4,3.7)--(-3.1,1.6)--(-3.04,0.56);
\draw[thick] (-3.03,0.27) -- node[midway,below, sloped]{\small{$\N_{\tau_1}\cap \{ \uu \leq \underline{w} \}$}} (2.15,5.45);
\draw[thick] (-4.17,6.83) -- node[midway,above, sloped]{\; \;\small{$\N_{\tau_2}\cap \{ \uu \leq \underline{w} \}$}} (-1.6,9.2);
\draw[thick] (2.15,5.45) -- node[midway,above,sloped]{\; \, \small{$\underline{\N}_{\underline{w}} \! \cap  \{ \tau_1 \leq u \! - \! u_0 \leq \tau_2 \}$}} (-1.6,9.2);
\draw[thick]   (-3,0) to [bend right=10] (-5,9);
\draw (-3.8,4) node[left]{$r=R_0$};
\draw (-1,5) node{$\mathcal{D}_{\tau_1}^{\tau_2} \cap \{ \uu \leq \underline{w} \}$};

\node (I)   at (9.5,7)   {};
\path  
  (I) +(90:3)  coordinate[label=90:$i^+$]  (Itop)
       +(-90:3) coordinate (Ibot)
       +(180:3) coordinate (Ileft)
       +(0:3)   coordinate[label=0:$i^0$]  (Iright)
       ;
\draw (Ibot) --      (Ileft) -- 
        node[midway, above, sloped] {\small{$\cal{H}^+$}}
      (Itop);
\draw[ dashed]   (Itop)    --
          node[midway, above, sloped] {\small{$\cal{I}^+$} }
      (Iright) -- (Ibot);  
\fill[color=gray!35] (10.15,8.5)--(10.4,8.75)--(11.2,8)--(10.37,7.2)--(10.28,7.85);   
\draw  (10.15,8.5)--(10.4,8.75)--(11.2,8)--(10.37,7.2); 
\draw (9.5,10) to [bend left] (9.5,4);
\draw (10.2,8.5) node[left]{\footnotesize{$r=R_0$}};
\draw[thin,blue]  (7.1575,7.6525) .. controls  (7.9025,6.9025) and (10.37,7).. node[below]{\small{$\mathcal{S}$}}   (12.5,7)  ;
\end{tikzpicture}
\end{center}
\caption{The set $\mathcal{D}_{\tau_1}^{\tau_2} \cap \{ \uu \leq \underline{w} \}$ and its boundary.}
\end{figure}

\noindent Fix $0 \leq p \leq 2q$ and apply the divergence theorem to the current $N \big[ r^p \frac{|v_{\uu}|^q}{|v_t|^q}|h|  \big]_{\mu}$ in the domain $\mathcal{D}_{\tau_1}^{\tau_2} \cap \{ \uu \leq \underline{w} \}$. This leads, as $v \cdot \n_{\N_{\tau}} = v_{\uu} \leq 0$ and $v \cdot \n_{\underline{\N}_{\underline{w}}} = v_u \leq 0$, to 
\begin{multline}\label{divaux}
\int_{\N_{\tau_2}} \mathds{1}_{ \uu \leq \underline{w}} \int_{\C}
 r^p \frac{|v_{\uu}|^q}{|v_t|^q} |h||v_{\uu}| \dr \m_{\C} \dr \m_{\N_{\tau_2}}+\int_{\underline{\N}_{\underline{w}}} \mathds{1}_{\tau_1 \leq u-u_0 \leq \tau_2} \int_{\C}
 r^p \frac{|v_{\uu}|^q}{|v_t|^q} |h||v_{u}| \dr \m_{\C} \dr \m_{\underline{\N}_{\underline{w}}} \\ -\int_{\N_{\tau_1}} \mathds{1}_{ \uu \leq \underline{w}} \int_{\C}
 r^p \frac{|v_{\uu}|^q}{|v_t|^q} |h||v_{\uu}| \dr \m_{\C} \dr \m_{\N_{\tau_1}}-\int_{t=\tau_1+u_0+R_0^*}^{\tau_2+u_0+R_0^*}  \int_{\mathbb{S}^2}\int_{\C} |R_0|^p \frac{|v_{\uu}|^q}{|v_t|^q} |h|  v_{r^*} \, \dr \m_{\C} \, R_0^2 \dr \m_{\mathbb{S}^2} \dr t  \\
= \, \int_{\mathcal{D}_{\tau_1}^{\tau_2}} \mathds{1}_{\uu \leq \underline{w}} \int_{\C} \T \left( r^p \frac{|v_{\uu}|^q}{|v_t|^q} |h| \right)   \dr \m_{\C} \dr \m_{\mathcal{D}_{\tau_1}^{\tau_2}} .
\end{multline}
Let
$$\mathbf{K}^p \, := \, \sup_{\underline{w} \geq \tau_2+u_0+2R_0^*} \int_{\underline{\N}_{\underline{w}}} \mathds{1}_{\tau_1 \leq u-u_0 \leq \tau_2} \int_{\C} r^p \frac{|v_{\uu}|^q}{|v_t|^q} |h||v_{u}| \dr \m_{\C} \dr \m_{\underline{\N}_{\underline{w}}} \geq 0$$
be the quantity that we need to bound in order to prove Remark \ref{energydecayintRq}. According to Beppo-Levi's theorem, we have for $\tau \in \{ \tau_1 , \tau_2 \}$,
\begin{align*}
\lim_{\underline{w} \to + \infty} \,  \int_{\N_{\tau}} \mathds{1}_{ \uu \leq \underline{w}} \int_{\C}
 r^p \frac{|v_{\uu}|^q}{|v_t|^q} |h||v_{\uu}| \dr \m_{\C} \dr \m_{\N_{\tau}} \, = \, \int_{\N_{\tau}}    \int_{\C}  r^p \frac{|v_{\uu}|^{q}}{|v_t|^{q}} |h| |v_{\uu}| \dr \m_{\C} \dr \m_{\N_{\tau}} .
\end{align*}
It will also be convenient to introduce
\begin{equation}\label{defJq}
\mathbf{J} := \int_{t=\tau_1+u_0+R_0^*}^{\tau_2+u_0+R_0^*}  \int_{\mathbb{S}^2}\int_{\C} \frac{|v_{\uu}|^q}{|v_t|^q} |h|  v_{r^*} \, \dr \m_{\C} \, R_0^2 \dr \m_{\mathbb{S}^2} \dr t .
\end{equation}
Note also that according to Lemma \ref{Lemhierar}, $ -\T \! \left( \! r^p \frac{|v_{\uu}|^q}{|v_t|^q}  \right)|h|$ is the sum of two nonnegative terms and a nonpositive one. Consequently, Beppo-Levi's theorem yields
$$ \lim_{\underline{w} \to + \infty} \! -\int_{\mathcal{D}_{\tau_1}^{\tau_2}} \! \mathds{1}_{\uu \leq \underline{w}} \! \int_{\C} \! \T \! \left( \! r^p \frac{|v_{\uu}|^q}{|v_t|^q} \! \right) \!   |h| \dr \m_{\C} \dr \m_{\mathcal{D}_{\tau_1}^{\tau_2}} = \! \int_{\mathcal{D}_{\tau_1}^{\tau_2}} \! \int_{\C} \! pr^{p-1} \! \frac{|v_{\uu}|^{q+1}}{|v_t|^{q}}|h|   \dr \m_{\C} \dr \m_{\mathcal{D}_{\tau_1}^{\tau_2}}+\frac{2q\!-\!p}{4}\mathbf{I}^p-(3q-p)\frac{M}{2} \mathbf{I}^{p-1}\!,$$
where $\mathbf{I}^s$, introduced in order to clearly identify the hierarchy between the $r$-weighted energy estimates, is defined for any $0 \leq s \leq 2q$ as
$$ \mathbf{I}^s \, := \, \int_{\mathcal{D}_{\tau_1}^{\tau_2}} \int_{\C} r^{s-1}\frac{|v_{\uu}|^{q-1}|\slashed{v}|^2}{|v_t|^{q}r^2}|h|   \dr \m_{\C} \dr \m_{\mathcal{D}_{\tau_1}^{\tau_2}}.$$
We then deduce, using also $\T(|h|)=0$, that for any $0 \leq p \leq 2q$,
\begin{multline}\label{eq:eizh}
\int_{\N_{\tau_2}}  \int_{\C}  r^p \frac{|v_{\uu}|^{q}}{|v_t|^{q}} |h| |v_{\uu}| \dr \m_{\C} \dr \m_{\N_{\tau_2}}+\int_{\mathcal{D}_{\tau_1}^{\tau_2}} \int_{\C} pr^{p-1}\frac{|v_{\uu}|^{q+1}}{|v_t|^{q}}|h|  \dr \m_{\C} \dr \m_{\mathcal{D}_{\tau_1}^{\tau_2}}+\mathbf{K}^p  + \frac{2q\!-\!p}{4}\mathbf{I}^p   \\ \leq \, (3q-p)\frac{M}{2} \mathbf{I}^{p-1}+ \int_{\N_{\tau_1}}   \int_{\C}  r^p \frac{|v_{\uu}|^{q}}{|v_t|^{q}} |h| |v_{\uu}| \dr \m_{\C} \dr \m_{\N_{\tau_1}}+|R_0|^p \cdot \mathbf{J}  .
\end{multline}
The remainder of the proof is composed of two steps. The first one consists in controlling sufficiently well $\mathbf{I}^{p-1}$ by taking advantage of the hierarchy in $p$ between these $r$-weighted energy estimates. Then, it will remain us to control $ \mathbf{J}$.

In order to initialize an induction, we need first to improve \eqref{eq:eizh} for $p \leq 1$. For this, note first that we have, as $R_0 >3M$,
$$ q \, \mathbf{I}^0 \, \lesssim \,  \frac{q}{2}\int_{\mathcal{D}_{\tau_1}^{\tau_2}} \int_{\C} \frac{1}{r} \left( 1 - \frac{3M}{r}\right) \frac{|v_{\uu}|^{q-1}|\slashed{v}|^2}{|v_t|^{q-1}r^2}|h|   \dr \m_{\C} \dr \m_{\mathcal{D}_{\tau_1}^{\tau_2}} \, = \, \frac{2q}{4}\mathbf{I}^0 - 3q\frac{M}{2} \mathbf{I}^{-1},$$
so that $\mathbf{I}^0$ can be bounded applying \eqref{eq:eizh} for $p=0$. Now remark that for all $0 \leq p \leq 1$, $\mathbf{I}^{p-1} \leq (2M)^{-1+p} \mathbf{I}^0$ and \eqref{eq:eizh} yields, for any $0 \leq p \leq \min (1,2q)$,
\begin{multline}\label{eq:nav}
\int_{\N_{\tau_2}}  \int_{\C}  r^p \frac{|v_{\uu}|^{q}}{|v_t|^{q}} |h| |v_{\uu}| \dr \m_{\C} \dr \m_{\N_{\tau_2}}+\int_{\mathcal{D}_{\tau_1}^{\tau_2}} \int_{\C} pr^{p-1}\frac{|v_{\uu}|^{q+1}}{|v_t|^{q}}|h|  \dr \m_{\C} \dr \m_{\mathcal{D}_{\tau_1}^{\tau_2}} +\mathbf{K}^p+ (2q-p)\mathbf{I}^p \\  \lesssim^{\mathstrut}_q  \int_{\N_{\tau_1}}   \int_{\C} (1+ r^p) \frac{|v_{\uu}|^{q}}{|v_t|^{q}} |h| |v_{\uu}| \dr \m_{\C} \dr \m_{\N_{\tau_1}}+ \mathbf{J} .
\end{multline}
If $2q >1$, it remains to improve \eqref{eq:eizh} for $1 < p \leq 2q$. We then assume that $2q >1$ and we fix $1 < p' \leq 2q$. Applying \eqref{eq:eizh} to $p=p'-n$, for all $n \in \llbracket 0, \lfloor p' \rfloor-1 \rrbracket$, and \eqref{eq:nav} to $p=p'- \lfloor p' \rfloor$, we obtain, since $r^s \leq (2M)^{s-p'}r^{p'}$ for all $0 \leq s \leq p'$,
\begin{multline}\label{eq:dek}
\int_{\N_{\tau_2}}  \int_{\C}  r^{p'} \frac{|v_{\uu}|^{q}}{|v_t|^{q}} |h| |v_{\uu}| \dr \m_{\C} \dr \m_{\N_{\tau_2}}+\int_{\mathcal{D}_{\tau_1}^{\tau_2}} \int_{\C} p'r^{p'-1}\frac{|v_{\uu}|^{q+1}}{|v_t|^{q}}|h|  \dr \m_{\C} \dr \m_{\mathcal{D}_{\tau_1}^{\tau_2}} +\mathbf{K}^{p'}+ (2q-p')\mathbf{I}^{p'} \\  \lesssim^{\mathstrut}_q  \int_{\N_{\tau_1}}   \int_{\C} r^{p'} \frac{|v_{\uu}|^{q}}{|v_t|^{q}} |h| |v_{\uu}| \dr \m_{\C} \dr \m_{\N_{\tau_1}}  + \mathbf{J} .
\end{multline}
In view of \eqref{eq:nav}, the estimate \eqref{eq:dek} holds for any $0 \leq p' \leq 2q$. We can then conclude the proof by using \eqref{equiint}, in order to bound by below the left hand side of the previous inequality, provided that
\begin{equation}\label{boundJq}
\mathbf{J} \, \lesssim^{\mathstrut}_q  \int_{\Sigma_{\tau_1}} \rho \Big[ |h| \Big] \dr \m_{\Sigma_{\tau_1}}.
\end{equation}
In order to prove this last inequality, recall the definition \eqref{defJq} of $\mathbf{J}$ and apply \eqref{divaux} for $p=0$ and, say, $\underline{w}=\tau_2+u_0+2R_0^*$. As $|v_{\uu}| \leq |v_t|$, $\N_{\tau} \subset \Sigma_{\tau}$, $\mathcal{D}_{\tau_1}^{\tau_2} \subset \Rm_{\tau_1}^{\tau_2}$ and $\T(|h|)=0$, this gives
\begin{align*}
\left| \mathbf{J}\right| \,  \lesssim^{\mathstrut}_q & \int_{\Sigma_{\tau_2}} \rho \Big[ |h| \Big] \dr \m_{\Sigma_{\tau_2}}+\mathbf{K}^0+\int_{\Sigma_{\tau_1}} \rho \Big[ |h| \Big] \dr \m_{\Sigma_{\tau_1}} + \int_{\mathcal{D}_{\tau_1}^{\tau_2}} \int_{\Pm} \left| \T \left( \frac{|v_{\uu}|^q}{|v_t|^q} \right) \right| |h|  \dr \m_{\Pm} \dr \m_{\mathcal{D}_{\tau_1}^{\tau_2}}.
\end{align*}
Note now that according to the energy inequality of Proposition \ref{divergence0} and Remark \ref{divergence0Rq},
$$\int_{\Sigma_{\tau_2}} \rho \Big[ |h| \Big] \dr \m_{\Sigma_{\tau_2}}+\int_{\underline{\N}_{\underline{w}}} \mathds{1}_{\tau_1 \leq u-u_0 \leq \tau_2}   \int_{\C}  |h|  |v_u| \dr \m_{\C} \dr \m_{\underline{\N}_{\underline{w}}} \, \leq \, \int_{\Sigma_{\tau_1}} \rho \Big[ |h| \Big] \dr \m_{\Sigma_{\tau_1}} .$$
Moreover, using Lemma \ref{Lemhierar}, $|v_{\uu}| \leq |v_t|$ and $R_0 > 3M$, we obtain
$$ \forall \, r \geq R_0, \qquad \left|\T \left( \frac{|v_{\uu}|^q}{|v_t|^q} \right) \right| = \frac{q}{2r}\left(1-\frac{3M}{r} \right) \frac{|v_{\uu}|^{q-1}}{|v_t|^q} \frac{|\slashed{v}|^2}{r^2|v_t|} \lesssim^{\mathstrut}_q \frac{(r-3M)^2}{r^3}\frac{|\slashed{v}|^2}{r^2},$$
so that, as $\mathcal{D}_{\tau_1}^{\tau_2} \subset \Rm_{\tau_1}^{\tau_2}$ and according to Proposition \ref{ILEDdeg}, applied to $|v_t|^{-1} h$,
$$ \int_{\mathcal{D}_{\tau_1}^{\tau_2}} \int_{\Pm} \left| \T \left( \frac{|v_{\uu}|^q}{|v_t|^q} \right) \right| |h|  \dr \m_{\Pm} \dr \m_{\mathcal{D}_{\tau_1}^{\tau_2}} \, \lesssim^{\mathstrut}_q  \int_{\Sigma_{\tau_1}} \rho \Big[ |h| \Big] \dr \m_{\Sigma_{\tau_1}}.$$
The last inequalities imply \eqref{boundJq}, which concludes the proof.
\end{proof}

\subsection{The particular case $(p,q)=(2,1)$}

Contrary to the region $\{ r^* \geq R_0^*+t \}$, which is studied below in Section \ref{sec6}, one cannot derive decay on $\int_{\Sigma_{\tau}} \! \rho \big[ |f| |v^{\mathstrut}_{\mathrm{N}}| \big] \dr \m_{\N_{\tau}}$ through a direct application of the $r^p |v_{\uu}|^q$-weighted energy inequalities. Instead, the decay will be obtained from integrated energy estimates, which will be proved using Propositions \ref{ILED} and \ref{rphierar}. Before considering a wider framework in the next subsection, we illustrate the strategy by treating the particular case $(p,q)=(2,1)$. We prove first a technical result, which will be useful for the general case as well.
\begin{Lem}\label{interpol}
Let $h_1$ and $h_2$ be two sufficiently regular functions defined on $\widehat{\Rm}_{-\infty}^{+\infty}$ and $s>1$. Then, for any integers $0 \leq k \leq i$ and any $\tau \in \R$, there holds
\begin{align*}
 \left| \int_{\Sigma_{\tau}} \rho \Big[ |h_1| |h_2|^{s^k-1}\Big] \dr \m_{\Sigma_{\tau}} \right|^{\frac{1}{s^k}} \, \lesssim^{\mathstrut}_i & \int_{\Sigma_{\tau}} \rho \Big[ |h_1| \Big] \dr \m_{\Sigma_{\tau}}+\left| \int_{\Sigma_{\tau}} \rho \Big[ |h_1| |h_2|^{s^i-1} \Big] \dr \m_{\Sigma_{\tau}} \right|^{\frac{1}{s^i}} .
 \end{align*}
\end{Lem}
\begin{proof}
As the inequality is straightforward for $k \in \{ 0,i \}$, we assume that $0<k<i$ and we consider $p= \frac{s^i-1}{s^k-1}$. Using Hölder's inequality, we get
$$ \left| \int_{\Sigma_\tau} \rho \Big[ |h_1| |h_2|^{s^k-1} \Big] \dr \m_{\Sigma_\tau} \right|^{\frac{1}{s^k}} \, \lesssim^{\mathstrut}_i  \left| \int_{\Sigma_\tau} \rho \Big[ |h_1| \Big] \dr \m_{\Sigma_\tau} \right|^{\frac{p-1}{ps^k}}  \left| \int_{\Sigma_\tau} \rho \Big[ |h_1| |h_2|^{s^i-1} \Big] \dr \m_{\Sigma_\tau} \right|^{\frac{1}{ps^k}}.$$
In then remains to apply the inequality $A^{\frac{1}{q}}B^{\frac{q-1}{q}} \leq \frac{1}{q}A+\frac{q-1}{q}B$, with $q=\frac{p}{p-1}s^k$. Indeed,
$$ \frac{(q-1)p s^k}{q} = (p-1)(q-1) = ps^k-p+1 = \frac{s^{i+k}-s^k-(s^i-1)+s^k-1}{s^k-1}= \frac{s^{i+k}-s^i}{s^k-1}=s^i.$$
\end{proof}
It will be convenient to use the following notation.
\begin{Def}\label{Defenergynorm}
We define, for any $0\leq p \leq 2q$, $s \geq 1$, $a \in \R_+$, $\tau \in \R$ and any sufficiently regular function $f : \widehat{\Rm}_{-\infty}^{+\infty}\to \R$,
$$ \mathbb{E}_{p,q}^{s,a}[f](\tau) := \int_{\Sigma_{\tau}}  \rho \left[ \left( 1+r^p \frac{|v_{\underline{u}}|^q}{|v_t|^q}\right) \langle v_t \rangle^{(4+a)(s-1)}  |f|^s |v^{\mathstrut}_{\mathrm{N}}|^a\right] \dr \m_{\Sigma_{\tau}}.$$
\end{Def}
\begin{Pro}\label{Proenergynorm}
Let $0\leq p \leq 2q$, $s>1$, $a \in \R_+$ and $f : \widehat{\Rm}_{-\infty}^{+\infty}\to \R$ be a sufficiently regular solution to the Vlasov equation. Then, for any $0 \leq k \leq i$,
\begin{align*} \forall \, \tau \in \R, \qquad  \left|  \mathbb{E}_{p,q}^{s^k,a}[f](\tau)  \right|^{\frac{1}{s^k}} & \lesssim_i  \mathbb{E}_{p,q}^{1,a}[f](\tau) + \left|  \mathbb{E}_{p,q}^{s^i,a}[f](\tau)  \right|^{\frac{1}{s^i}}, \\
\forall \, \tau_2 \geq \tau_1 \geq 0 , \qquad \mathbb{E}_{p,q}^{s^k,a}\big[f  \big](\tau_2) & \lesssim_{a,q} \mathbb{E}_{p,q}^{s^k,a} \big[f \big](\tau_1).
\end{align*}
\end{Pro}
\begin{proof}
The first inequality follows from the previous Lemma \ref{interpol}. For the energy estimate, note first that $\T(\langle v_t \rangle^{(4+a)(s-1)} f)=0$ according to Lemma \ref{Comuprop} and recall that $v_{\mathrm{N}}=v_t$ for $r \geq R_0$. It remains to combine Propositions \ref{energyredshift} with \ref{rphierar}. 
\end{proof}
We now prove an energy decay statement for solutions to the massless Vlasov equation.
\begin{Pro}\label{casep=12}
Let $f : \widehat{\Rm}_0^{+\infty} \rightarrow \R$ be a sufficiently regular solution to $\T(f)=0$. Then, for any $1<s \leq 6/5$, we have for all $\tau \geq 0$,
\begin{align*}
 \int_{\Sigma_{\tau}} \rho \left[ |f| |v^{\mathstrut}_{\mathrm{N}}| \right] \dr \m_{\Sigma_{\tau}} \lesssim_s \, & \frac{1}{(1+\tau)^2} \left( \mathbb{E}_{2,1}^{1,1}[f](0)+\left| \mathbb{E}_{2,1}^{s^3,1}[f](0) \right|^{\frac{1}{s^3}} \right)
 \end{align*}
\end{Pro}
\begin{proof}
Consider $s \in ]1,\frac{6}{5}]$ and $0 \leq \tau_1 < \tau_2$. Applying Proposition \ref{rphierar} to $h=f|v_t|$, with $(p,q)=(2,1)$, gives
\begin{align*}
\int_{\tau=\tau_1}^{\tau_2} \int_{\N_{\tau}} \int_{\C} r |f||v_{\uu}|^2  \dr \m_{\C} \dr \m_{\N_{\tau}} \dr \tau \, & \leq \, \int_{\tau=0}^{\tau_2} \int_{\N_{\tau}} \int_{\C} r |f||v_{\uu}|^2  \dr \m_{\C} \dr \m_{\N_{\tau}} \dr \tau \\
& \lesssim \, \int_{\N_{0}} \int_{\C} r^2|f| |v_{\uu}|^2  \dr \m_{\N_{0}} +\int_{\Sigma_{0}} \rho \Big[ |f| |v_t| \Big] \dr \m_{\Sigma_{0}}.
\end{align*}
Applying this time Proposition \ref{rphierar} with $(p,q)=(1,1)$, between times $\tau$ and $\tau_2$, we obtain
$$ \int_{\tau=\tau_1}^{\tau_2} \int_{\N_{\tau_2}} \int_{\C} r |f||v_{\uu}|^2  \dr \m_{\C} \dr \m_{\N_{\tau_2}} \dr \tau  \lesssim  \int_{\tau=\tau_1}^{\tau_2} \bigg( \int_{\N_{\tau}} \int_{\C} r |f||v_{\uu}|^2  \dr \m_{\C} \dr \m_{\N_{\tau}}+\int_{\Sigma_{\tau}} \rho \Big[ |f| |v_t| \Big] \dr \m_{\Sigma_{\tau}} \bigg) \dr \tau .$$
We then deduce from these last two estimates and Proposition \ref{divergence0}, applied between times $\tau_1$ and $\tau$,
$$ ( \tau_2-\tau_1)\! \int_{\N_{\tau_2}} \! \int_{\C} r |f||v_{\uu}|^2  \dr \m_{\C} \dr \m_{\N_{\tau_2}}   \lesssim  \int_{\N_{0}} \! \int_{\C} r^2|f| |v_{\uu}|^2  \dr \m_{\N_{0}} +\int_{\Sigma_{0}} \rho \Big[ |f| |v_t| \Big] \dr \m_{\Sigma_{0}} +(\tau_2-\tau_1) \!\int_{\Sigma_{\tau_1}} \rho \Big[ |f| |v_t| \Big] \dr \m_{\Sigma_{\tau_1}} \! .$$
The idea, in order to derive decay from such an inequality, will be to apply it for a dyadic sequence of times, i.e. with $(\tau_1,\tau_2)=(2^j, 2^{j+1})$ so that $\tau_1 \sim \tau_2-\tau_1 \sim \tau_2$, and then to control the energy flux at time $\tau \sim 2^j$ by Proposition \ref{energyredshift}. In view of Remark \ref{rqcontrol}, Proposition \ref{rphierar}, applied to $h=f|v_t|$ with $(p,q)=(1,1)$, leads, for all $\tau_3 > \tau_2$, to
\begin{align*}
\int_{\tau=\tau_2}^{\tau_3}\int_{\N_{\tau}} \int_{\C}  |f||v_t||v_{\uu}|  \dr \m_{\C} \dr \m_{\N_{\tau}} \dr \tau  \, \lesssim \, &  \frac{1}{\tau_2-\tau_1} \int_{\N_{0}} \int_{\C} r^2|f| |v_{\uu}|^2  \dr \m_{\N_{0}} +\frac{1}{\tau_2-\tau_1}\int_{\Sigma_{0}} \rho \Big[ |f| |v_t| \Big] \dr \m_{\Sigma_{0}}\\
& +\int_{\Sigma_{\tau_1}} \rho \Big[ |f| |v_t| \Big] \dr \m_{\Sigma_{\tau_1}} +\int_{\Sigma_{\tau_2}} \rho \Big[ |f| |v_t| \Big] \dr \m_{\Sigma_{\tau_2}} .
\end{align*}
The last term is equal to the penultimate one according to Proposition \ref{divergence0}. Recall now that $v_{\uu}=v \cdot \n_{\Sigma_{\tau}}$ and $v_t = v^{\mathstrut}_{\mathrm{N}}$ for $r \geq R_0$ (see Proposition \ref{energyredshift}). Moreover, by \eqref{redshiftdef} we have $|v_t| \leq |v^{\mathstrut}_{\mathrm{N}}|$. Hence, combining the last inequality with the integrated local energy decay estimate of Proposition \ref{ILED}, we get, as $2(s-1) \leq \frac{1}{2}$,
\begin{align*}
\int_{\tau=\tau_2}^{\tau_3}\int_{\Sigma_{\tau}} \rho \Big[ |f| |v^{\mathstrut}_{\mathrm{N}}| \Big] \dr \m_{\Sigma_{\tau}} \dr \tau  \, \lesssim^{\mathstrut}_s \, &  \frac{1}{\tau_2-\tau_1} \int_{\Sigma_{0}} \rho \left[ \left( 1+ r^2 \frac{|v_{\uu}|}{|v_t|} \right) \!|f| |v^{\mathstrut}_{\mathrm{N}}| \right] \dr \m_{\Sigma_{0}} +\int_{\Sigma_{\tau_1}} \rho \Big[ |f| |v^{\mathstrut}_{\mathrm{N}}| \Big] \dr \m_{\Sigma_{\tau_1}}\\ &+ \left| \int_{\Sigma_{\tau_2}}  \rho \left[\left\langle \sqrt{r}\frac{|v_{\underline{u}}|}{|v_t|} \right\rangle \langle v_t \rangle^{5(s-1)}  |f|^s |v^{\mathstrut}_{\mathrm{N}}|\right] \dr \m_{\Sigma_{\tau_2}} \right|^{\frac{1}{s}}.
\end{align*}
Now, apply the energy estimate of Proposition \ref{energyredshift} between times $\tau$ and $\tau_3$ in order to bound below the left hand side of the last inequality. We bound above the last term on the right hand side using Hölder inequality. This gives, for all $\tau_3 > \tau_2$,
\begin{align}
\nonumber \int_{\Sigma_{\tau_3}} \rho  \Big[ |f|&|v^{\mathstrut}_{\mathrm{N}}| \Big]  \dr \m_{\Sigma_{\tau_3}}  \,  \lesssim^{\mathstrut}_s \,   \frac{1}{(\tau_3-\tau_2)(\tau_2-\tau_1)} \int_{\Sigma_{0}} \rho \left[ \left( 1+ r^2 \frac{|v_{\uu}|}{|v_t|} \right) \!|f| |v^{\mathstrut}_{\mathrm{N}}| \right] \dr \m_{\Sigma_{0}}+\frac{1}{\tau_3-\tau_2}\int_{\Sigma_{\tau_1}} \rho \Big[ |f| |v^{\mathstrut}_{\mathrm{N}}| \Big] \dr \m_{\Sigma_{\tau_1}}  \\  & +\frac{1}{\tau_3-\tau_2} \left| \int_{\Sigma_{\tau_2}}  \rho \left[\left( 1+r^2\frac{|v_{\underline{u}}|}{|v_t|} \right) \langle v_t \rangle^{5(s-1)}  |f|^s |v^{\mathstrut}_{\mathrm{N}}|\right] \dr \m_{\Sigma_{\tau_2}} \right|^{\frac{1}{4s}} \left| \int_{\Sigma_{\tau_2}}  \rho \left[ \langle v_t \rangle^{5(s-1)}  |f|^s |v^{\mathstrut}_{\mathrm{N}}|\right] \dr \m_{\Sigma_{\tau_2}} \right|^{\frac{3}{4s}}. \nonumber
\end{align}
Remark that $2^{j+1}-2^j =2^j$ for all $j \geq 0$. Note further, in view of Lemma \ref{Comuprop}, that $\T \left( \langle v_t \rangle^{5(s^n-1)}|f|^{s^n}  \right)=0$ and $\langle v_t \rangle^{5(s-1)}|\langle v_t \rangle^{5(s^n-1)}|f|^{s^n} |^s=\langle v_t \rangle^{5(s^{n+1}-1)}|f|^{s^{n+1}} $. We apply the last inequality, for any $j \in \mathbb{N}$,
\begin{itemize}
\item to $f$, for $(\tau_3,\tau_2,\tau_1)=(2^{j+4},2^{j+3},2^{j+2})$ and then for $(\tau_3,\tau_2,\tau_1)=(2^{j+2},2^{j+1},0)$.
\item To $\langle v_t \rangle^{5(s-1)}|f|^{s}$, for $(\tau_3,\tau_2,\tau_1)=(2^{j+3},2^{j+2},2^{j+1})$ and then for $(\tau_3,\tau_2,\tau_1)=(2^{j+1},2^{j},0)$.
\item To $\langle v_t \rangle^{5(s^2-1)}|f|^{s^2}$, for $(\tau_3,\tau_2,\tau_1)=(2^{j+2},2^{j+1},0)$.
\end{itemize} 
Using the norms introduced in Definition \ref{Defenergynorm}, which decrease with time by Proposition \ref{Proenergynorm}, this leads to
\begin{align*}
 \int_{\Sigma_{2^{j+4}}} \rho \Big[ |f||v^{\mathstrut}_{\mathrm{N}}| \Big] & \dr \m_{\Sigma_{2^{j+4}}}   \lesssim^{\mathstrut}_s  \left( 2^{-2j}+2^{-3j} \right) \mathbb{E}_{2,1}^{1,1}[f](0)  
 \\
 &+  \left(2^{-2j}+2^{-j-2\frac{3j}{4s}}+2^{-j-3\frac{3j}{4s}} \right) \! \left|  \mathbb{E}_{2,1}^{s,1}[f](0) \right|^{\frac{1}{s}}+2^{-j-2\frac{3j}{4s}}\left|  \mathbb{E}_{2,1}^{s,1}[f](0) \right|^{\frac{1}{4s}} \left|  \mathbb{E}_{2,1}^{s^2,1}[f](0) \right|^{\frac{3}{4s^2}} \\
&  +2^{-j-\frac{3j}{4s}-\frac{9j}{16s^2}}\left|  \mathbb{E}_{2,1}^{s,1}[f](0) \right|^{\frac{1}{4s}} \left(\left|  \mathbb{E}_{2,1}^{s^2,1}[f](0) \right|^{\frac{3}{4s^2}}+\left|  \mathbb{E}_{2,1}^{s^2,1}[f](0) \right|^{\frac{3}{16s^2}}\left|  \mathbb{E}_{2,1}^{s^3,1}[f](0) \right|^{\frac{9}{16s^3}} \right).
\end{align*}
As $s \leq \frac{6}{5}$, we finally obtain, using Young's inequality for products as well as Proposition \ref{Proenergynorm}, 
\begin{equation}
 2^{2j} \int_{\Sigma_{2^{j+4}}} \rho \Big[ |f||v^{\mathstrut}_{\mathrm{N}}| \Big] \dr \m_{\Sigma_{2^{j+4}}}  \, \lesssim^{\mathstrut}_s \, \mathbb{E}_{2,1}^{1,1}[f](0) +  \left|  \mathbb{E}_{2,1}^{s^3,1}[f](0) \right|^{\frac{1}{s^3}} \!. \label{eq:daouzek}
\end{equation}
Let $\tau \geq 2^4$ and consider $j \geq 0$ satisfying $2^{j+4} \leq \tau < 2^{j+5}$, so that $1+\tau \leq 2^6 \cdot 2^{j}$. We then deduce from Proposition \ref{energyredshift} that
$$ (1+\tau)^{2} \int_{\Sigma_{\tau}} \rho \Big[ |f| |v^{\mathstrut}_{\mathrm{N}}| \Big] \dr \m_{\Sigma_{\tau}}   \, \lesssim \, 2^{2j} \int_{\Sigma_{2^{j+4}}} \rho \Big[ |f||v^{\mathstrut}_{\mathrm{N}}| \Big] \dr \m_{\Sigma_{2^{j+4}}},$$
which, combined with \eqref{eq:daouzek}, implies the required estimate. If $\tau \leq 2^4$, $1 \lesssim (1+\tau)^{-2} $ and the estimate directly follows from Proposition \ref{energyredshift}.
\end{proof}

\subsection{Proof of Proposition \ref{energydecaysec4}}

Let, for all this subsection, $f : \widehat{\Rm}_0^{+\infty} \rightarrow \R$ be a sufficiently regular function satisfying $\T(f)=0$. Recall that
\begin{itemize}
\item $\N_{\tau} = \Sigma_{\tau} \cap \{ r \geq R_0 \}$, $v \cdot \n_{\N_{\tau}} = v_{\uu}$ and $v^{\mathstrut}_{\mathrm{N}}=v_t$ for all $r \geq R_0$.
\item $|v_{\uu}| \leq |v_t|$ and $r^p \lesssim_p^{\mathstrut} 1$ in the region $\{ r < R_0 \}$.
\end{itemize}
Consequently, for any $(a,p) \in \R_+^2$, we have
 $$\forall \, \tau \in \R_+, \qquad \int_{\Sigma_{\tau}} \rho \bigg[ r^p \frac{|v_{\uu}|^{p/2}}{|v_t|^{p/2}} |f| |v^{\mathstrut}_{\mathrm{N}}|^a \bigg] \dr \m_{\Sigma_{\tau}}  \lesssim^{\mathstrut}_{p} \int_{\Sigma_{\tau}} \rho \bigg[  |f| |v^{\mathstrut}_{\mathrm{N}}|^a \bigg] \dr \m_{\Sigma_{0}}+\int_{\N_{\tau}} \int_{\C} r^p \frac{|v_{\uu}|^{p/2}}{|v_t|^{p/2}} |f||v_t|^a |v_{\uu}| \dr \m_{\C} \dr \m_{\N_{\tau}}.$$
The first estimate of Proposition \ref{energydecaysec4} then ensues from the energy inequality of Proposition \ref{energyredshift} and Proposition \ref{rphierar}, applied to $h=f|v_t|^a$, with the parameters $(p,p/2)$ and between times $0$ and $\tau$.

We now turn to the second estimate. It will be convenient to use the following notation.
\begin{Def}\label{defsec5}
Let $x \in \R$ and recall the notations introduced in Definition \ref{defsec4}. We define $q_x \in \R$ as
\begin{itemize}
\item $q_x := \frac{x}{2} $ if $\lceil x \rceil \in 2 \mathbb{Z}$ is even. Note that if $x=2n$, $n \in \mathbb{Z}$, then $q_x=n$.
\item $q_x= \frac{x+1}{2} $ if $\lceil x \rceil \in 2 \mathbb{Z}+1$ is odd. In particular, if $x=2n-1$, $n \in \mathbb{Z}$, then $q_x=n$.
\end{itemize}
\end{Def}
Even if the main ideas are the same than those used in the proof of Proposition \ref{casep=12}, a new difficulty arises. Indeed, by iterating Proposition \ref{rphierar} we do not obtain an estimate of the energy norm of $f$ if $p \notin \mathbb{N}$ but merely on the ones of $r^{p-\lfloor p \rfloor}f$ and $r^{p-\lceil p \rceil}f$. This forces us to carefully choose $q$ and then to estimate the energy norm of $f$ through an interpolation. More precisely, we will apply Proposition \ref{rphierar} with $(p-n,q_{p-n})$, for all $n \in \llbracket 0, \lceil p \rceil \rrbracket$. For the purpose of proving and exploiting a hierarchy of $r^{p-n} |v_{\uu}|^{q_{p-n}}$-weighted energy estimates, the following properties, which can easily be checked, will be useful.
\begin{Lem}\label{forapplyrp}
For any $x \in \R_+^*$, we have,
\begin{itemize}
\item if $\lceil x \rceil \in 2 \mathbb{Z}$ is even, then $q^{\mathstrut}_{x-1}=q^{\mathstrut}_x$ and $0<x=2q_x$.
\item Otherwise $\lceil x \rceil \in 2 \mathbb{Z}+1$ is odd and $q^{\mathstrut}_{x-1}=q^{\mathstrut}_x -1$. Moreover, $2q_x-x=1$ so $0 \leq x < 2q_x$.
\end{itemize}
Consequently, we can apply Proposition \ref{rphierar} with the parameters $(x,q_x)$. Moreover, if $\lceil x \rceil$ is odd, we will also be able to use Remark \ref{rqcontrol}.
\end{Lem}

The first step of the proof consists in proving an integrated decay estimate for the region $\{r \geq R_0\}$.
\begin{Lem}\label{LemsolVlasovbis1}
Let $h : \widehat{\Rm}_0^{+\infty} \rightarrow \R$ be a sufficiently regular function, $p >0$ and define $\tau_j := 2^j$. For any $j \in \mathbb{N}$ and any $n \in \mathbb{N}^*$ such that $n \leq \lceil p \rceil$,
$$ \int_{\tau_{j+n-1}}^{\tau_{j+n}}\!\int_{\N_{\tau}}\! \int_{\C} r^{p-n} \frac{|v_{\uu}|^{q_{p-n}}}{|v_t|^{q_{p-n}}} |h| |v_{\uu}| \dr \m_{\C} \dr \m_{\N_{\tau}} \dr \tau  \lesssim^{\mathstrut}_p     \frac{1}{2^{(n-1) j}} \!\int_{\N_{\tau_j}} \! \int_{\C} r^{p} \frac{|v_{\uu}|^{q_p}}{|v_t|^{q_p}} |h| |v_{\uu}| \dr \m_{\C} \dr \m_{\N_{\tau_j}}  + \int_{\Sigma_{\tau_j}} \! \rho \Big[ |h| \Big] \dr \m_{\Sigma_{\tau_j}}\! .$$
Moreover, we have
$$ \int_{\tau=\tau_{j}}^{\tau_{j+1}} \!  \int_{\N_{\tau}}\! \int_{\C} r^{p} \frac{|v_{\uu}|^{q_{p}}}{|v_t|^{q_{p}}} |h| |v_{\uu}| \dr \m_{\C} \dr \m_{\N_{\tau}} \dr \tau  \lesssim^{\mathstrut}_p  2^j\! \int_{\N_{\tau_j}}\!  \int_{\C}\! r^{p} \frac{|v_{\uu}|^{q_p}}{|v_t|^{q_p}} |h| |v_{\uu}| \dr \m_{\C} \dr \m_{\N_{\tau_j}} + 2^j \! \int_{\Sigma_{\tau_j}} \rho \Big[ |h| \Big] \! \dr \m_{\Sigma_{\tau_j}}  .$$
\end{Lem}
\begin{proof}
Fix $j \in \mathbb{N}$, $p>0$ and let us perform an induction in order to prove the first estimate. In view of Lemma \ref{forapplyrp}, we obtain from Proposition \ref{rphierar}, applied with the parameters $(p,q_p)$, 
\begin{align*}
 \int_{\tau=\tau_j}^{\tau_{j+1}} \! \int_{\N_{\tau}}  \int_{\C} r^{p-1} \frac{|v_{\uu}|^{q_{p-1}}}{|v_t|^{q_{p-1}}} |h| |v_{\uu}| \dr \m_{\N_{\tau}} \dr \tau \, \lesssim^{\mathstrut}_p & \int_{\N_{\tau_j}} \int_{\C} r^p \frac{|v_{\uu}|^{q_p}}{|v_t|^{q_p}} |h| |v_{\uu}| \dr \m_{\N_{\tau_j}}+\int_{\Sigma_{\tau_j}} \rho \Big[ |h|  \Big] \dr \m_{\Sigma_{\tau_j}}.
 \end{align*}
The inequality then holds at the rank $n=1$. Assume now that $p >1$ and consider $ n \in \mathbb{N}^*$ satisfying $n \leq \lceil p \rceil-1$ and such that the result holds at the rank $n$. Applying Proposition \ref{rphierar} with the parameters $(p-n,q_{p-n})$ and between times $\tau$ and $\tau_{j+n}$ leads to
\begin{multline*}
\int_{\tau=\tau_{j+n-1}}^{\tau_{j+n}}\! \int_{\N_{\tau_{j+n}}} \!\int_{\C} r^{p-n} \frac{|v_{\uu}|^{q_{p-n}}}{|v_t|^{q_{p-n}}} |h|  |v_{\uu}| \dr \m_{\C} \dr \m_{\N_{\tau_{j+n}}} \dr \tau \\ \lesssim^{\mathstrut}_p \int_{\tau=\tau_{j+n-1}}^{\tau_{j+n}} \bigg( \int_{\N_{\tau}}  \int_{\C} r^{p-n} \frac{|v_{\uu}|^{q_{p-n}}}{|v_t|^{q_{p-n}}} |h|  |v_{\uu}| \dr \m_{\C} \dr   \m_{\N_{\tau}}+\int_{\Sigma_{\tau}} \rho \Big[ |h| \Big] \dr \m_{\Sigma_{\tau}} \bigg) \dr \tau.
\end{multline*}
Now note that $\tau_{j+n}-\tau_{j+n-1}=2^{n-1} \cdot 2^j$. Consequently, combining the last inequality with the energy estimate of Proposition \ref{divergence0}, applied between times $\tau_{j}$ and $\tau$, and the induction hypothesis, we obtain
$$2^j\! \int_{\N_{\tau_{j+n}}}\!  \int_{\C}  r^{p-n} \frac{|v_{\uu}|^{q_{p-n}}}{|v_t|^{q_{p-n}}} |h| |v_{\uu}| \dr \m_{\C} \dr \m_{\N_{\tau_{j+n}}} \! \lesssim^{\mathstrut}_p  \frac{1}{2^{(n-1) j}}\!\int_{\N_{\tau_j}} \!\int_{\C}  r^{p} \frac{|v_{\uu}|^{q_p}}{|v_t|^{q_p}} |h|  |v_{\uu}| \dr \m_{\C} \dr \m_{\N_{\tau_j}}\! + 2^j\!\int_{\Sigma_{\tau_j}}  \rho \Big[  |h|  \Big] \dr \m_{\Sigma_{\tau_j}} \! .$$
Applying again Proposition \ref{rphierar} with the parameters $(p-n,q_{p-n})$, between times $\tau_{j+n}$ and $\tau_{j+n+1}$, yields, in view of Lemma \ref{forapplyrp},
\begin{align*}
 \int_{\tau=\tau_{j+n}}^{\tau_{j+n+1}}\! \int_{\N_{\tau}} \int_{\C}  r^{p-n-1} \frac{|v_{\uu}|^{q_{p-n-1}}}{|v_t|^{q_{p-n-1}}} |h| |v_{\uu}| \dr \m_{\C} \dr \m_{\N_{\tau}} \dr \tau  \, \lesssim^{\mathstrut}_p \, &  \frac{1}{2^{n j}}\!\int_{\N_{\tau_j}}  \int_{\C} r^{p} \frac{|v_{\uu}|^{q_p}}{|v_t|^{q_p}} |h |v_{\uu}| \dr \m_{\C} \dr \m_{\N_{\tau_j}}  \\ &  +\int_{\Sigma_{\tau_j}}  \rho \Big[ |h| \Big] \dr \m_{\Sigma_{\tau_j}} +\int_{\Sigma_{\tau_{j+n}}}  \rho \Big[ |h| \Big] \dr \m_{\Sigma_{\tau_{j+n}}} .
\end{align*}
We then obtain the inequality at the rank $n+1$ by applying once again Proposition \ref{divergence0}, between times $\tau_j$ and $\tau_{j+n}$.

For the second part of the lemma, apply Proposition \ref{rphierar}, with the parameters $(p,q_p)$ and between times $\tau_j$ and $\tau$. This gives,
$$\int_{\N_{\tau}}  \int_{\C} r^{p} \frac{|v_{\uu}|^{q_{p}}}{|v_t|^{q_{p}}} |h| |v_{\uu}| \dr \m_{\N_{\tau}}  \, \lesssim^{\mathstrut}_p  \int_{\N_{\tau_j}} \int_{\C} r^p \frac{|v_{\uu}|^{q_p}}{|v_t|^{q_p}} |h| |v_{\uu}| \dr \m_{\N_{\tau_j}}\!+\int_{\Sigma_{\tau_j}} \rho \Big[ |h|  \Big] \dr \m_{\Sigma_{\tau_j}}.$$
It then remains to integrate the previous inequality between $\tau_j$ and $\tau_{j+1}$.
\end{proof}
\begin{Cor}\label{LemsolVlasovbis2}
Let $a \geq 0$ and $p >0 $. Then, for any $j \in \mathbb{N}$, we have
$$ \int_{\tau= \tau_{j+\lceil p \rceil -1}}^{\tau_{j+\lceil p \rceil}} \mathds{1}_{r \geq R_0} \! \int_{\Sigma_{\tau}} \rho \Big[  |f| |v^{\mathstrut}_{\mathrm{N}}|^a \Big] \dr \m_{\Sigma_{\tau}} \dr \tau  \lesssim^{\mathstrut}_{p}   \frac{1}{2^{(p-1) j}} \! \int_{\Sigma_0}  \rho \left[\!\left(\! 1+ r^{p} \frac{|v_{\uu}|^{q_p}}{|v_t|^{q_p}} \right) \!  |f| |v^{\mathstrut}_{\mathrm{N}}|^a \! \right] \! \dr \m_{\Sigma_0}  +  \int_{\Sigma_{j}} \! \rho \Big[ |f| |v_t|^a \Big] \dr \m_{\Sigma_{j}} \!.$$
\end{Cor}
\begin{proof}
Fix $j \in \mathbb{N}$ and apply Proposition \ref{rphierar} to $h=f|v_t|^a$ with the parameters $(p,q_p)$, between times $0$ and $\tau_j$. This gives
$$ \int_{\N_{\tau_j}}  \int_{\C} r^{p} \frac{|v_{\uu}|^{q_p}}{|v_t|^{q_p}} |f| |v_t|^a |v_{\uu}| \dr \m_{\C} \dr \m_{\N_{\tau_j}}  \lesssim^{\mathstrut}_p  \int_{\N_{0}} \int_{\C}  r^{p} \frac{|v_{\uu}|^{q_p}}{|v_t|^{q_p}} |f| |v_t|^a |v_{\uu}| \dr \m_{\C} \dr \m_{\N_{0}}+\int_{\Sigma_0}  \rho \Big[ |f| |v_t|^a \Big] \dr \m_{\Sigma_0}.$$
We then deduce, since $\Sigma_{\tau} \cap \{ r \geq R_0 \} = \N_{\tau}$, $v_{\uu} = v \cdot \n_{\N_{\tau}}$, $|v_t| \leq |v^{\mathstrut}_{\mathrm{N}}|$ and $v_t=v^{\mathstrut}_{\mathrm{N}}$ for $r \geq R_0$ (see \eqref{redshiftdef}), that
\begin{align}\label{mil}
& \int_{\N_{\tau_j}} \! \int_{\C} r^{p} \frac{|v_{\uu}|^{q_p}}{|v_t|^{q_p}} |f| |v_t|^a |v_{\uu}| \dr \m_{\C} \dr \m_{\N_{\tau_j}} \, \lesssim^{\mathstrut}_p \,  \int_{\Sigma_0}  \rho \left[\left( 1+ r^{p} \frac{|v_{\uu}|^{q_p}}{|v_t|^{q_p}} \right) \! |f| |v^{\mathstrut}_{\mathrm{N}}|^a  \right]  \dr \m_{\Sigma_0}  , \\
\nonumber & \mathbf{Q}  \, := \, \int_{\tau= \tau_{j+\lceil p \rceil -1}}^{\tau_{j+\lceil p \rceil}} \mathds{1}_{r \geq R_0} \int_{\Sigma_{\tau}} \rho \Big[  |f| |v^{\mathstrut}_{\mathrm{N}}|^a \Big] \dr \m_{\Sigma_{\tau}} \dr \tau  \, = \, \int_{\tau= \tau_{j+\lceil p \rceil -1}}^{\tau_{j+\lceil p \rceil}}  \int_{\N_{\tau}} \int_{\C}  |f| |v_t|^a |v_{\uu}| \dr \m_{\C} \dr \m_{\N_{\tau}} \dr \tau.
 \end{align}
Hence, if $p \in \mathbb{N}^*$, one only has to apply Lemma \ref{LemsolVlasovbis1} with $h=f|v_t|^a$ and $n=p$. We now assume $p \in \R_+^* \setminus \mathbb{N}$ and we use Hölder's inequality in order to obtain
\begin{multline*}
 \mathbf{Q} \, \lesssim^{\mathstrut}_{p} 
  \left| \int_{\tau= \tau_{j+\lceil p \rceil}+1}^{\tau_{j+\lceil p \rceil}}  \int_{\N_{\tau}} \int_{\C} r^{p-\lceil p \rceil} \frac{|v_{\uu}|^{\frac{p-\lceil p \rceil}{2}}}{|v_t|^{\frac{p-\lceil p \rceil}{2}}}  |f| |v_t|^a |v_{\uu}| \dr \m_{\C} \dr \m_{\N_{\tau}} \dr \tau \right|^{p-\lceil p \rceil+1} \\ \times  \left| \int_{\tau= \tau_{j+\lceil p \rceil}+1}^{\tau_{j+\lceil p \rceil}}  \int_{\N_{\tau}} \int_{\C} r^{p-\lceil p \rceil+1} \frac{|v_{\uu}|^{\frac{p-\lceil p \rceil+1}{2}}}{|v_t|^{\frac{p-\lceil p \rceil+1}{2}}}  |f| |v_t|^a |v_{\uu}| \dr \m_{\C} \dr \m_{\N_{\tau}} \dr \tau \right|^{\lceil p \rceil-p}.
  \end{multline*}
Remark now that, since $-1<p-\lceil p  \rceil <0$ and $0 < p-\lceil p \rceil+1< 1$, we have $2q_{p-\lceil p \rceil}= p-\lceil p \rceil$ and $2q_{p-\lceil p \rceil+1}=p-\lceil p \rceil+1$. So, using the last inequality, Lemma \ref{LemsolVlasovbis1}, applied for $n= \lceil p \rceil $ and $n=\lceil p \rceil-1$, as well as Young's inequality for products $A^{\frac{s-1}{s}}B^{\frac{1}{s}} \leq \frac{s-1}{s}A+ \frac{1}{s}B$, with $\frac{1}{s} = \lceil p \rceil-p$, we get
$$\mathbf{Q} \, \lesssim^{\mathstrut}_p  \frac{1}{2^{(p-1) j}}\int_{\N_{\tau_j}}  \int_{\C} r^{p} \frac{|v_{\uu}|^{q_p}}{|v_t|^{q_p}} |f| |v_t|^a |v_{\uu}| \dr \m_{\C} \dr \m_{\N_{\tau_j}}  + \max (1, 2^{(1-p)j} ) \int_{\Sigma_{\tau_j}}  \rho \Big[ |f| |v_t|^a \Big] \dr \m_{\Sigma_{\tau_j}}  .$$
The result then follows from \eqref{mil} and, if $p< 1$, Proposition \ref{divergence0}, applied between times $0$ and $\tau_j$.
\end{proof}
We are now able to conclude the proof of Proposition \ref{energydecaysec4}, which is implied by applying the following result for $n=\lceil p \rceil$ and by noticing that $|v_{\uu}|^{q_p}\leq |v_t|^{q_p-p/2} |v_{\uu}|^{p/2}$.
\begin{Pro}\label{rpgenera}
Let $(a,p) \in \R_+^2$ and $1<s \leq 1+\langle 5 p \rangle^{-2}$. For any $n \in \llbracket 0, \lceil p \rceil \rrbracket$, we have for all $\tau \in \R_+$, 
\begin{align*}
 \int_{\Sigma_{\tau}} \rho \Big [|f| |v^{\mathstrut}_{\mathrm{N}}|^a \Big] \dr \m_{\Sigma_{\tau}} \lesssim_{a,p,s} \, & \frac{1}{(1+\tau)^{\min(p, n)}}\int_{\Sigma_0} \rho \left[ \left(1+ r^{p} \frac{|v_{\uu}|^{q_{p}}}{|v_t|^{q_{p}}} \right)|f||v^{\mathstrut}_{\mathrm{N}}|^{a } \right] \dr \m_{\Sigma_0} \\
 & +\frac{1}{(1+\tau)^{\min(p, n)}} \left|\int_{\Sigma_0} \rho \left[ \left(1+ r^{p} \frac{|v_{\uu}|^{q_p}}{|v_t|^{q_p}} \right)\! |f|^{s^{2n}}\langle v_t \rangle^{(4+a)(s^{2n}-1)}|v^{\mathstrut}_{\mathrm{N}}|^{a } \right] \dr \m_{\Sigma_0} \right|^{s^{-2n}} .
\end{align*}
\end{Pro} 
\begin{proof}
Let $(a,p) \in \R_+^2$, $1<s \leq 1+\langle 5 p \rangle^{-2}$ and note that if $n=0$, it suffices to apply the energy estimate of Proposition \ref{energyredshift}. Assume then that $p>0$ and that the result holds at the rank $n \in \llbracket 0, \lceil p \rceil -1 \rrbracket$. According to Proposition \ref{ILED}, applied between times $\tau_{j+\lceil p \rceil}$ and $\tau_{j}$,
\begin{align*}
 \int_{\tau=\tau_{j+\lceil p \rceil -1}}^{\tau_{j+\lceil p \rceil }} \! \mathds{1}_{r \leq R_0} & \int_{\Sigma_{\tau}} \rho \Big[ |f| |v^{\mathstrut}_{\mathrm{N}}|^a \Big]  \dr \m_{\Sigma_{\tau}} \dr \tau \leq \int_{\tau=\tau_{j}}^{\tau_{j+\lceil p \rceil }} \! \mathds{1}_{r \leq R_0} \! \int_{\Sigma_{\tau}} \rho \Big[ |f| |v^{\mathstrut}_{\mathrm{N}}|^a \Big]  \dr \m_{\Sigma_{\tau}} \dr \tau \\ & \lesssim^{\mathstrut}_{s}  \int_{\Sigma_{\tau_j}} \! \rho \Big[ |f| |v^{\mathstrut}_{\mathrm{N}}|^a \Big] \dr \m_{\Sigma_{\tau_j}} + \left| \int_{\Sigma_{\tau_j}}  \rho \left[\left\langle r^{2(s-1)}\frac{|v_{\underline{u}}|}{|v_t|} \right\rangle \langle v_t \rangle^{(4+a)(s-1)}  |f|^s  |v^{\mathstrut}_{\mathrm{N}}|^{a}\right] \dr \m_{\Sigma_{\tau_j}} \right|^{\frac{1}{s}}\! .
 \end{align*}
Moreover, according to Proposition \ref{energyredshift}, applied between times $\tau$ and $\tau_{j+\lceil p \rceil}$, and $\tau_{j+\lceil p \rceil}-\tau_{j+\lceil p \rceil-1}= 2^{\lceil p \rceil-1} \cdot 2^j$,
$$  \int_{\Sigma_{\tau_{j+\lceil p \rceil}}} \rho \Big[ |f| |v^{\mathstrut}_{\mathrm{N}}|^a \Big]  \dr \m_{\Sigma_{\tau_{j+\lceil p \rceil}}}  \, \lesssim^{\mathstrut}_a \frac{1}{ 2^{\lceil p \rceil-1}  \cdot 2^j} \int_{\tau=\tau_{j+\lceil p \rceil-1}}^{\tau_{j+\lceil p \rceil}}   \int_{\Sigma_{\tau}} \rho \Big[ |f| |v^{\mathstrut}_{\mathrm{N}}|^a \Big]  \dr \m_{\Sigma_{\tau}} \dr \tau  .$$
Combining the last two estimates with Corollary \ref{LemsolVlasovbis2} yields, 
\begin{align} 
 \nonumber \int_{\Sigma_{\tau_{j+\lceil p \rceil}}} \! \rho \Big[ |f| |v^{\mathstrut}_{\mathrm{N}}|^a \Big] \dr \m_{\Sigma_{\tau_{j+\lceil p \rceil}}}    \lesssim^{\mathstrut}_{a,p,s} & \, \frac{1}{2^{pj}} \int_{\Sigma_{0}}  \rho \left[ \left( 1+r^{p} \frac{|v_{\uu}|^{q_p}}{|v_t|^{q_p}} \right) |f| |v^{\mathstrut}_{\mathrm{N}}|^a \right] \dr \m_{\Sigma_{0}}  + \frac{1}{2^j} \! \int_{\Sigma_{\tau_j}} \! \rho \Big[ |f| |v^{\mathstrut}_{\mathrm{N}}|^a \Big] \dr \m_{\Sigma_{\tau_j}} \\ 
 & + \frac{1}{2^{j}}  \left| \int_{\Sigma_{\tau_j}} \! \rho \left[\left(1+ r^{2(s-1)}\frac{|v_{\underline{u}}|}{|v_t|} \right) \langle v_t \rangle^{(4+a)(s-1)}  |f|^s  |v^{\mathstrut}_{\mathrm{N}}|^{a}\right] \! \dr \m_{\Sigma_{\tau_j}} \right|^{\frac{1}{s}} \! . \label{eq:trizek}
\end{align}
If $n=0$, it remains to bound the last two terms by applying Proposition \ref{Proenergynorm} between times $0$ and $\tau_j$. Otherwise, $n \geq 1$ and then $p \geq 1$. We apply the induction hypothesis to $f$ at the rank $n$ in order to bound the second term on the right hand side of \eqref{eq:trizek}. We get, using the energy norms introduced in Definition \ref{Defenergynorm} and applying Proposition \ref{Proenergynorm},
$$2^{nj} \!\int_{\Sigma_{\tau_j}} \! \rho \Big[ |f| |v^{\mathstrut}_{\mathrm{N}}|^a \Big] \dr \m_{\Sigma_{\tau_j}} \! \lesssim_{a,p,s}  \mathbb{E}_{p,q_p}^{1,a}[f](0)+\left| \mathbb{E}_{p,q_p}^{s^{2n},a}[f](0) \right|^{\frac{1}{s^{2n}}}\lesssim_p   \mathbb{E}_{p,q_p}^{1,a}[f](0)+\left| \mathbb{E}_{p,q_p}^{s^{2n+1},a}[f](0) \right|^{\frac{1}{s^{2n+1}}}  .$$
For the last term on the right hand side in \eqref{eq:trizek}, which is equal to $|\mathbb{E}_{2(s-1),1}^{s,a}[f](\tau_j)|^{1/s}$, we start by applying Hölder inequality with exponents $\frac{p}{2(s-1)}$ and $\frac{p}{p-2(s-1)}$. As $|v_{\uu}| \leq |v_t|$ and $q_p \leq \frac{p}{2(s-1)}$,

$$ |\mathbb{E}_{2(s-1),1}^{s,a}[f](\tau_j)|^{\frac{1}{s}} \leq \left| \mathbb{E}_{p,q_p}^{s,a}[f](\tau_j)\right|^{\frac{2(s-1)}{ps}} \, \left| \mathbb{E}_{0,0}^{s,a}[f](\tau_j) \right|^{\frac{p-2(s-1)}{ps}} \leq \left| \mathbb{E}_{p,q_p}^{s,a}[f](0)\right|^{\frac{2(s-1)}{ps}} \, \left| \mathbb{E}_{0,0}^{s,a}[f](\tau_j) \right|^{\frac{p-2(s-1)}{ps}}  .$$  
We then deduce that
\begin{align}
 \nonumber \int_{\Sigma_{\tau_{j+\lceil p \rceil}}} \! \rho \Big[ |f| |v^{\mathstrut}_{\mathrm{N}}|^a \Big] \dr \m_{\Sigma_{\tau_{j+\lceil p \rceil}}}   & \lesssim^{\mathstrut}_{a,p,s}  \, \frac{1}{2^{\min(p,(n+1))j}} \left( \mathbb{E}_{p,q_p}^{1,a}[f](0)+\left| \mathbb{E}_{p,q_p}^{s^{2n+1},a}[f](0) \right|^{\frac{1}{s^{2n+1}}} \right) \\
 & +\frac{1}{2^j}\left| \mathbb{E}_{p,q_p}^{s,a}[f](0)\right|^{\frac{2(s-1)}{ps}} \, \left| \int_{\Sigma_{\tau_{j}}} \! \rho \Big[\langle v_t \rangle^{(4+a)(s-1)} |f| |v^{\mathstrut}_{\mathrm{N}}|^a \Big] \dr \m_{\Sigma_{\tau_{j}}}  \right|^{\frac{p-2(s-1)}{ps}}\! . \label{eq:pevarzek}
\end{align}
The constant
$$ \alpha := \frac{p-2(s-1)}{ps}$$
will play an important role in the remainder of the proof. Moreover, it will be useful to remark that
$$ \forall \, k \in \mathbb{N}, \qquad \langle v_t \rangle^{(4+a)(s^k-1)}\big|\langle v_t \rangle^{(4+a)(s-1)}|f|^{s} \big|^{s^k}=\langle v_t \rangle^{(4+a)(s^{k+1}-1)}|f|^{s^{k+1}} .$$
Applying the induction hypothesis to $\langle v_t \rangle^{(4+a)(s-1)} f$, at the rank $n$, and then Proposition \ref{Proenergynorm} yields
\begin{equation}\label{eq:reutiliserla}
 \int_{\Sigma_{\tau_{j+\lceil p \rceil}}} \! \rho \Big[ |f| |v^{\mathstrut}_{\mathrm{N}}|^a \Big] \dr \m_{\Sigma_{\tau_{j+\lceil p \rceil}}} \!   \lesssim  \! \left( \frac{1}{2^{\min(p,(n+1))j}}+\frac{1}{2^{j+\alpha nj}} \right) \! \left( \mathbb{E}_{p,q_p}^{1,a}[f](0)+\left| \mathbb{E}_{p,q_p}^{s^{2n+1},a}[f](0) \right|^{\frac{1}{s^{2n+1}}} \right) \!.
\end{equation}
If $1+\alpha n \geq p$, then a new application of Proposition \ref{Proenergynorm} gives us the result at the rank $n+1$ for any $\tau \in \{ \tau_j \, , \, j \geq \lceil p \rceil\}$. Otherwise, we apply \eqref{eq:reutiliserla} to $\langle v_t \rangle^{(4+a)(s-1)} f$, so that we get from \eqref{eq:pevarzek},
\begin{align*}
 \int_{\Sigma_{\tau_{j+2\lceil p \rceil}}} \! \rho \Big[ |f| |v^{\mathstrut}_{\mathrm{N}}|^a \Big] \dr \m_{\Sigma_{\tau_{j+2\lceil p \rceil}}}    \lesssim^{\mathstrut}_{a,p,s}  &  \frac{1}{2^{\min(p,(n+1))j}} \left( \mathbb{E}_{p,q_p}^{1,a}[f](0)+\left| \mathbb{E}_{p,q_p}^{s^{2n+1},a}[f](0) \right|^{\frac{1}{s^{2n+1}}} \right)  \\
 & +\frac{1}{2^{j+\alpha(j+\alpha nj)}} \left( \mathbb{E}_{p,q_p}^{1,a}[f](0)+\left| \mathbb{E}_{p,q_p}^{s^{2n+2},a}[f](0) \right|^{\frac{1}{s^{2n+2}}} \right).
\end{align*}
Since $s \leq 1+\frac{1}{5p^2}$, we have $\alpha(j+\alpha nj) \geq nj$. Consequently, in any cases, we have
\begin{equation}\label{eq:reutiliserla2}
 \int_{\Sigma_{\tau_{j+2\lceil p \rceil}}} \! \rho \Big[ |f| |v^{\mathstrut}_{\mathrm{N}}|^a \Big] \dr \m_{\Sigma_{\tau_{j+2\lceil p \rceil}}}    \lesssim^{\mathstrut}_{a,p,s}  \! \frac{1}{2^{\min(p,(n+1))j}} \! \left( \mathbb{E}_{p,q_p}^{1,a}[f](0)+\left| \mathbb{E}_{p,q_p}^{s^{2n+2},a}[f](0) \right|^{\frac{1}{s^{2n+2}}} \right) \!.
\end{equation}
We are now able to prove the result at the rank $n+1$.
\begin{itemize}
\item If $\tau \geq \tau_{2\lceil p \rceil}$, there exists $j \in \mathbb{N}$ such that $\tau_{j+2\lceil p \rceil} \leq \tau < \tau_{j+2\lceil p \rceil+1}$. Then, by Proposition \ref{energyredshift},
$$ \int_{\Sigma_{\tau}} \rho \Big[ |f| |v^{\mathstrut}_{\mathrm{N}}|^a \Big]  \dr \m_{\Sigma_{\tau}} \, \lesssim_{a}^{\mathstrut} \,  \int_{\Sigma_{\tau_{j+2\lceil p \rceil}}} \rho \Big[ |f| |v^{\mathstrut}_{\mathrm{N}}|^a \Big] \dr \m_{\Sigma_{\tau_{j+\lceil p \rceil}}},  \qquad 1+\tau \lesssim^{\mathstrut}_p 2^j.$$
Combining this with \eqref{eq:reutiliserla2} yields the estimate at the rank $n+1$ in the case $\tau \geq \tau_{2\lceil p \rceil}$.
\item Otherwise, $1+\tau \leq 1+2^{2\lceil p \rceil}$ and one only has to apply Proposition \ref{energyredshift} between times $0$ and $\tau$.
\end{itemize}
\end{proof}
\begin{Rq}
In order to simplify the presentation, we did not consider the case where $\T(f)=F$, with $F$ a source term decaying sufficiently fast. Otherwise, we would have to add spacetime integrals such as
$$\frac{1}{(1+\tau)^p} \int_0^{\tau} \int_{\Sigma_{\tau'}} \int_{\C} r^{p-n} \frac{|v_{\uu}|^{q_{p-n}}}{|v_t|^{q_{p-n}}}(1+\tau)^n |\T(f)| |v^{\mathstrut}_{\mathrm{N}}|^a \dr \m_{\C} \dr \m_{\Sigma_{\tau'}} \dr \tau', \qquad n \in \mathbb{N},$$
to the right hand side of the estimates proved in this subsection.
\end{Rq}
\section{Pointwise decay estimates}\label{sec5}

\subsection{Preparatory results}
The purpose of this subsection is to express the radial derivative of $\rho[|f| |v^{\mathstrut}_{\mathrm{N}}|]$, for $f$ a solution to the massless Vlasov equation, in terms of derivatives of $f$ that we can control in $L^1 $ by applying the results obtained previously in this paper. This will allow us to prove pointwise decay estimates on the velocity average of $f$ through Sobolev inequalities.

We start with useful computations.
\begin{Lem}\label{calculvstar}
For any $i \in \llbracket 1, 3 \rrbracket$, there holds almost everywhere
$$ \widehat{\Omega}_i (v_{r^*})=0, \qquad \widehat{\Omega}_i (|\slashed{v}|)=0, \qquad \widehat{\Omega}_i (v_t)=0, \qquad \widehat{\Omega}_i (v^{\mathstrut}_{\mathrm{N}})=0.$$
For any $a \in \R_+$ and $0 \leq p \leq 2q$, we have on $\widehat{\Rm}_{-\infty}^{+\infty}$,
$$ \bigg|\partial_{r} \! \bigg( \frac{|v^{\mathstrut}_{\mathrm{N}}|^a}{r^2|v_t|^2} \bigg) \! \bigg|  \lesssim^{\mathstrut}_{a} \! \frac{|v^{\mathstrut}_{\mathrm{N}}|^a}{r^2|v_t|^2}, \qquad  \left|\partial_{\uu} \! \left( \frac{r^p|v_{\uu}|^{q+1}}{r^2|v_t|^3} \right) \!\right|  \lesssim^{\mathstrut}_{q} \! \frac{r^p|v_{\uu}|^{q+1}}{r^2|v_t|^3} , \qquad  \left|\partial_{u} \! \left( \frac{r^p|v_{\uu}|^{q}|v_u|}{r^2|v_t|^3} \right)\! \right|  \lesssim^{\mathstrut}_{q}  \! \frac{r^p|v_{\uu}|^{q}|v_u|}{r^2|v_t|^3}. $$
For any $s \in [1,+\infty[$, any vector field $V \in \{ \partial_t, \widehat{\Omega}_1 , \widehat{\Omega}_2, \widehat{\Omega}_3 , \T \}$ and any locally $W^{1,s}$ function $h$,
\begin{equation}\label{pderiv}
 \left| V (|h|^s ) \right| \, \leq \, |V(h)|^s + (s-1)|h|^s.
 \end{equation}
\end{Lem}
\begin{proof}
We start by the first part of the Lemma and we fix $1 \leq i \leq 3$. In view of the expression of $\widehat{\Omega}_i$ (see \eqref{def1}-\eqref{def3}), we have $\widehat{\Omega}_i (v_{r^*})=0$. For the second equality, since $(v_{\theta},v_{\varphi}) \neq 0$ almost everywhere, it is sufficient to prove $\widehat{\Omega}_i(|\slashed{v}|^2)=0$. Recall that $|\slashed{v}|^2:= |v_{\theta}|^2 + \frac{|v_{\varphi}|^2}{\sin^2 ( \theta)}$, so that
\begin{align*}
\widehat{\Omega}_1 (|\slashed{v}|^2) \hspace{1mm} & = \hspace{1mm} \sin (\varphi) \frac{\cot(\theta) |v_{\varphi}|^2}{\sin^2(\theta)}-\cos (\varphi)\frac{v_{\varphi}}{\sin^2 (\theta)}v_{\theta} +\cos (\varphi) v_{\theta}\frac{v_{\varphi}}{\sin^2 (\theta)}-\sin (\varphi) \cot(\theta) v_{\varphi} \frac{v_{\varphi}}{\sin^2 (\theta)} \hspace{1mm} = \hspace{1mm} 0,\\
\widehat{\Omega}_2 (|\slashed{v}|^2) \hspace{1mm} & = \hspace{1mm}  -\cos (\varphi) \frac{\cot(\theta) |v_{\varphi}|^2}{\sin^2(\theta)}-\sin (\varphi)\frac{v_{\varphi}}{\sin^2 (\theta)}v_{\theta} +\sin (\varphi) v_{\theta}\frac{v_{\varphi}}{\sin^2 (\theta)}+\cos (\varphi) \cot(\theta) v_{\varphi} \frac{v_{\varphi}}{\sin^2 (\theta)} \hspace{1mm} = \hspace{1mm} 0,\\
\widehat{\Omega}_3 (|\slashed{v}|^2) \hspace{1mm} & = \hspace{1mm} \partial_{\varphi} |\slashed{v}|^2 \hspace{1mm} = \hspace{1mm} 0.
\end{align*}
As $\widehat{\Omega}_i(r)=0$, this implies directly, in view of the definition of $v_t$ and $v^{\mathstrut}_{\mathrm{N}}$ (see \eqref{defv0}, \eqref{redshiftdef}), that $\widehat{\Omega}_i (v_t) = \widehat{\Omega} (v^{\mathstrut}_{\mathrm{N}})=0$. For \eqref{pderiv}, one only has to note that for any smooth vector field $V$, 
\begin{equation*}
 \left| V (|h|^s ) \right| \, = \, s |V (h) h|h|^{s-2} | \, = \, s|V(h)| |h|^{s-1} \, \leq \, |V(h)|^s + (s-1)|h|^s.
 \end{equation*}
We now turn to the second part of the Lemma. Recall that $\partial_{r^*}= \left(1-\frac{2M}{r} \right) \partial_r$. Hence, using the expression \eqref{defv0} of $v_t$ and then the co-mass shell condition \eqref{vu}, we get
\begin{align*}
\partial_{r^*} v_t \, = \, \left(1-\frac{2M}{r} \right) \partial_r \left( \frac{1-\frac{2M}{r}}{r^2} \right) \frac{|\slashed{v}|^2}{v_t} \, = \, \frac{6M-2r}{r^2}\left( 1-\frac{2M}{r} \right)\frac{|\slashed{v}|^2}{r^2 v_t} \, = \, \frac{8(3M-r)|v_u||v_{\uu}|}{r^2v_t}.
\end{align*}
Recall now that $v_{\underline{u}}=\frac{v_t+v_{r^*}}{2}\leq 0$ and $v_{u}=\frac{v_t-v_{r^*}}{2} \leq 0$, so that
\begin{align*}
\partial_{r^*} |v_{\underline{u}}| \, = \, \partial_{r^*} |v_{u}| \, = \, \frac{4(r-3M)|v_u||v_{\uu}|}{r^2v_t}, \qquad  \quad \partial_{r^*}\! \left( \frac{ |v_u|}{1-\frac{2M}{r}} \right) \, = \, \left(\frac{4(r-3M)|v_{\uu}|}{r^2v_t}- \frac{2M}{r^2} \right) \frac{|v_u|}{1-\frac{2M}{r}}.
\end{align*}
Together with $|r-3M| \leq r$, $\frac{1}{r} \lesssim 1$, $v_t \leq 0$, $|v_u| \leq |v_t|$, $|v_{\uu}| \leq |v_t|$ and the definition \eqref{redshiftdef} of $v^{\mathstrut}_{\mathrm{N}}$, this leads to
$$|\partial_{r^*} |v_t| | \lesssim  |v_t|, \quad |\partial_{r^*} |v_u| | \lesssim |v_u|, \quad |\partial_{r^*} |v_{\uu}| | \lesssim |v_{\uu}|, \quad |\partial_{r^*} |v^{\mathstrut}_{\mathrm{N}}| | \lesssim |v^{\mathstrut}_{\mathrm{N}}|.$$
As $|\partial_{r^*}(r^{s})| \leq sr^{s-1} \lesssim_s r^s$, this implies the result.
\end{proof}
The following commutation property between the angular derivatives and the averaging in $v$ will be useful.
\begin{Lem}\label{Sobsphere}
Let $ h : \widehat{\Rm}_{-\infty}^{+\infty} \rightarrow \R$ be a sufficiently regular function. We have
\begin{align}
\nonumber \forall \, i \in \llbracket 1,3 \rrbracket, \qquad \left| \Omega_i \int_{\C} |h|  \dr \m_{\C} \right| \, & \leq \,   \int_{\C} \left| \widehat{\Omega}_i h \right| \dr \m_{\C} .
\end{align}
Moreover, for any $s \in [1,+\infty[$, there exists a constant $C_s>0$ depending only on $s$ such that
$$ \left\| \int_{\C} |h|^s \dr \m_{\C} \right\|_{L^{\infty}( \mathbb{S}^2)} \, \leq \, C_s \, \sum_{|I| \leq 2} \int_{\mathbb{S}^2} \int_{\C} \left| \widehat{\Omega}^I h \right|^s \dr \m_{\C} \dr \m_{\mathbb{S}^2}.$$
\end{Lem}
\begin{proof}
For simplicity, we introduce $\Omega_i^v :=\widehat{\Omega}_i-\Omega_i$. Note first that using Lemma \ref{calculvstar}, we have $\big| \widehat{\Omega}_i\big(|h| |v_t|^{-1} \big)\big| = \big| \widehat{\Omega}_i(h)\frac{h}{|h|}|v_t|^{-1} \big|=\big| \widehat{\Omega}_i(h) \big| |v_t|^{-1} $ in $W^{1,1}$. Hence, since $\dr \m_{\C}=|v_t|^{-1}r^{-2} \sin^{-1} (\theta ) \dr v_{r^*} \dr v_{\theta}\dr v_{\varphi} $, there holds
\begin{align*}
\left| \Omega_i \int_{\C} |h| \dr \m_{\C} \right| \, & = \, \left| \int_{\C} \! \left(\widehat{\Omega}_i \! \left(|h| |v_t|^{-1} \right)\!-\Omega_i^v \!\left( |h| |v_t|^{-1} \right) \! \right) \! \frac{\dr v_{r^*} \dr v_{\theta} \dr v_{\varphi} }{r^2 \sin (\theta)}+\int_{\C}\! |h||v_t|^{-1} \Omega_i \left( \frac{1}{\sin \theta } \right) \! \frac{\dr v_{r^*} \dr v_{\theta} \dr v_{\varphi} }{r^2  } \right| \\
& \leq \,  \int_{\C} \left| \widehat{\Omega}_i(h) \right|  \dr \m_{\C}+ \left| \int_{\C}  \left( |h||v_t|^{-1} \Omega_i \left( \frac{1}{\sin \theta } \right) -\frac{\Omega_i^v \big(|h| |v_t|^{-1}\big)}{\sin (\theta)} \right) \frac{\dr v_{\theta} \dr v_{\varphi} \dr v_{r^*}}{r^2 } \right|.
\end{align*}
It remains to notice that the last term on the right hand side of the last inequality vanishes. Indeed, this is straightforward if $i=3$ since $\Omega_3=\partial_{\varphi}$ and $\widehat{\Omega}_{\varphi}=\partial_{\varphi}$. If $i=1$, we have $\Omega_1=- \sin \varphi \, \partial_{\theta}-\cot \theta \cos \varphi \, \partial_{\varphi}$, so that $\Omega_1 \left( \frac{1}{\sin \theta } \right)=\frac{\sin ( \varphi) \cot (\theta)}{\sin(\theta)}$, and $\Omega^v_1= - \cos \varphi \frac{v_{\varphi}}{\sin^2 \theta} \partial_{v_{\theta}}+\left( \cos \varphi \, v_{\theta}-\sin \varphi \cot \theta \, v_{\varphi} \right)\partial_{v_{\varphi}} $, so by integration by parts in\footnote{Note that there is no boundary terms since $\{v_{r^*}=0\}$ has Lebesgue measure $0$.} $(v_{\theta}, v_{\varphi})$,
$$ -\int_{\C}   \frac{\Omega_i^v \big(|h||v_t|^{-1} \big)}{\sin (\theta)} \frac{\dr v_{\theta} \dr v_{\varphi} \dr v_{r^*}}{r^2 } = -\int_{\C} \frac{ \sin (\varphi) \cot ( \theta )}{\sin (\theta)}|h| |v_t|^{-1}\frac{\dr v_{\theta} \dr v_{\varphi} \dr v_{r^*}}{r^2 }.$$
The case $i=2$ can be treated similarly. We then deduce from Remark \ref{exprdtheta} and \eqref{othersphecoord} the following inequalities, for the derivatives of the spherical coordinate systems $(\theta, \varphi)$ and $(\widetilde{\theta} , \widetilde{\varphi})$,
\begin{align}
\left| \partial_{\theta} \int_{\C} |h|  \dr \m_{\C} \right|+\left| \partial_{\varphi} \int_{\C} |h| \dr \m_{\C} \right|  +
\left| \partial_{\widetilde{\theta}} \int_{\C} |h| \dr \m_{\C} \right|+\left| \partial_{\widetilde{\varphi}} \int_{\C} |h| \dr \m_{\C} \right|  \,  \leq \, 2\sum_{j=1}^3 \left| \int_{\C} \left| \widehat{\Omega}_j h \right|  \dr \m_{\C} \right| . \label{angularcom}
 \end{align}
We now turn to the Sobolev inequality. Let $\mathbb{V}$ be the subset of $\mathbb{S}^2$ containing all the points $\omega$ which are in the set covered by the coordinate system $(\theta , \varphi)$ and satisfying $\frac{\pi}{4} \leq \theta (\omega)  \leq \frac{3\pi}{4}$. We start by considering $\omega \in \mathbb{V}$. Using first a local one dimensional Sobolev inequality and then \eqref{angularcom}, applied to $|h|^s$, together with \eqref{pderiv}, we get
\begin{align*}
 \int_{\C} |h|^s \dr \m_{\C} \Big\vert_{(\theta(\omega), \varphi ( \omega))} \, & \lesssim \, \int_{\theta=\theta(\omega)}^{\theta(\omega)+\frac{\pi}{12}} \int_{\C} |h|^s \dr \m_{\C} \Big\vert_{(\theta, \varphi ( \omega))} +\left| \left(\partial_{\theta} \int_{\C} |h|^s \dr \m_{\C} \right) \Big\vert_{(\theta, \varphi ( \omega))} \right| \dr \theta \\
 &  \lesssim^{\mathstrut}_s \, \sum_{|I| \leq 1} \int_{\theta=\frac{\pi}{4}}^{\frac{5\pi}{6}} \int_{\C} \left|\widehat{\Omega}^I h \right|^s  \dr \m_{\C} \Big\vert_{(\theta, \varphi ( \omega))} \dr \theta.
\end{align*}
Applying again a one dimensional Sobolev inequality and then using \eqref{angularcom}, \eqref{pderiv} as well as $\sin (\theta) \geq \frac{1}{2}$ on the domain of integration, we obtain
\begin{align*}
 \int_{\C} |h|^s \dr \m_{\C} \Big\vert_{(\theta(\omega), \varphi ( \omega))} \, & \lesssim^{\mathstrut}_s \, \sum_{|I| \leq 1} \int_{\varphi=0}^{2 \pi} \int_{\theta=\frac{\pi}{4}}^{\frac{5\pi}{6}} \int_{\C} \left|\widehat{\Omega}^I h \right|^s \dr \mu_{\C}+ \left| \partial_{\varphi}\int_{\C} \left|\widehat{\Omega}^I h \right|  \dr \m_{\C} \right|^s \dr \theta \dr \varphi \\
  & \lesssim^{\mathstrut}_s \,  \sum_{|J| \leq 2} \int_{\varphi=0}^{2 \pi}\int_{\theta=\frac{\pi}{4}}^{\frac{5\pi}{6}} \int_{\C} \! \left|\widehat{\Omega}^J h \right|^s \dr \m_{\C}  \sin(\theta) \dr \theta \dr \varphi \, \leq \, \sum_{|J| \leq 2} \int_{\mathbb{S}^2} \int_{\C} \! \left|\widehat{\Omega}^J h \right|^s  \dr \m_{\C}  \dr \m_{\mathbb{S}^2}\!.
\end{align*}
The case of the points $\omega \in \mathbb{S}^2 \setminus \mathbb{V}$ can be handled similarly using this time the coordinate system $(\widetilde{\theta} , \widetilde{\varphi})$ instead of $(\theta , \varphi)$ and \eqref{angularcom}. Indeed, in view of \eqref{othersphecoord}, if $\omega \in \mathbb{S}^2 \setminus \mathbb{V}$, then $\frac{\pi}{4} \leq \widetilde{\theta} (\omega) \leq \frac{3 \pi}{4}$.
\end{proof}

The following result will be crucial in order to derive boundedness in $L^s(\Sigma_t)$ for quantities involving $|v_{r^*}|^2  \partial_{r^*}f$, where $f$ is a solution to the massless Vlasov equation. The main idea of the proof consists in rewritting $v_{r^*} \partial_{r^*} f$ using the operator $\T$ and then to deal with the terms containing $v$ derivatives of $f$ by integration by parts.
\begin{Lem}\label{rdeveriv}
Let $a \in \R_+$ and $h : \widehat{\Rm}_{-\infty}^{+\infty} \rightarrow \R$ be a sufficiently regular function. On $\Rm_{-\infty}^{+\infty}$, we have
$$\left| \frac{1}{1-\frac{2M}{r}} \partial_{u}  \int_{\C} \frac{|v_{r^*}|^2}{|v_t|^2} |h| |v_t| |v^{\mathstrut}_{\mathrm{N}}|^{a} \dr \m_{\C} \right| \,  \lesssim^{\mathstrut}_{a}  \sum_{i=1}^3 \int_{\C} \left( |h| +|\partial_t h| + \big|\widehat{\Omega}_i h\big|\right)|v^{\mathstrut}_{\mathrm{N}}|^{a+1} \dr \m_{\C}+\int_{\C} \left| \T (h) \right| |v^{\mathstrut}_{\mathrm{N}}|^a  \dr \m_{\C} .$$
For any $0 \leq p \leq 2q$, there holds on $\Rm_{-\infty}^{+\infty} \cap \{ r \geq R_0 \}$,
\begin{align*}
\left| \partial_{\uu}  \int_{\C} \frac{|v_{r^*}|^2}{|v_t|^2} r^p |v_{\uu}|^q|h| |v_{\uu}|  \dr \m_{\C} \right| \, & \lesssim^{\mathstrut}_{q}  \sum_{i=1}^3 \int_{\C} r^p |v_{\uu}|^q \left( |h| +|\partial_t h| + \big|\widehat{\Omega}_i h\big|\right) \! |v_{\uu}| \dr \m_{\C} +\int_{\C} r^p\frac{|v_{\uu}|^q}{|v_t|}\left| \T (h) \right| |v_{\uu}|  \dr \m_{\C} ,
\\
\left| \partial_{u}  \int_{\C} \frac{|v_{r^*}|^2}{|v_t|^2} r^p |v_{\uu}|^q|h| |v_u|  \dr \m_{\C} \right| \, & \lesssim^{\mathstrut}_{q}  \sum_{i=1}^3 \int_{\C} r^p |v_{\uu}|^q \left( |h| +|\partial_t h| + \big|\widehat{\Omega}_i h\big|\right)\! |v_{u}| \dr \m_{\C} +\int_{\C} r^p\frac{|v_{\uu}|^q}{|v_t|}\left| \T (h) \right| |v_{u}|  \dr \m_{\C} .
\end{align*}
\end{Lem}
\begin{proof}
Since $\dr \m_{\C}= r^{-2} \sin^{-1}(\theta)|v_t|^{-1} \dr v_{\theta} \dr v_{\varphi} \dr v_{r^*}$ and $2 \partial_u = \partial_t-\left(1 - \frac{2M}{r} \right) \partial_r$, we have 
$$ \left| \frac{1}{1-\frac{2M}{r}} \partial_{u} \int_{\C} \frac{|v_{r^*}|^2}{|v_t|} |h| |v^{\mathstrut}_{\mathrm{N}}|^a  \dr \m_{\C} \right|   =  \left| \int_{\C} \left( \frac{|v_{r^*}|^2}{1-\frac{2M}{r}} \partial_{u} ( |h| )  \frac{|v^{\mathstrut}_{\mathrm{N}}|^a}{r^2|v_t|^2} -\frac{1}{2} |v_{r^*}|^2|h| \partial_{r} \! \left( \frac{|v^{\mathstrut}_{\mathrm{N}}|^a}{r^2|v_t|^2} \right)  \right) \frac{\dr v_{\theta} \dr v_{\varphi} \dr v_{r^*}}{ \sin ( \theta )} \right| .$$ 
The second term of the integrand on the right hand side can be bounded using Lemma \ref{calculvstar}. For the first one, note that $2v_{r^*} \partial_u =-v_{r^*} \partial_{r^*}+ v_t \partial_t-2v_u \partial_t$ and $\frac{|v_u|}{1-\frac{2M}{r}} \leq |v^{\mathstrut}_{\mathrm{N}}|$ (see \eqref{defvtbar}). Consequently,
\begin{align*}
\left| \frac{1}{1-\frac{2M}{r}} \partial_{u}\! \int_{\C} \! \frac{|v_{r^*}|^2}{|v_t|} |h| |v^{\mathstrut}_{\mathrm{N}}|^a  \dr \m_{\C} \right| \! & \lesssim \! \left| \int_{\C} \! \frac{v_t \partial_t (|h|)\!-\!v_{r^*} \partial_{r^*}  \! ( |h| )}{1-\frac{2M}{r}} \frac{v_{r^*}|v^{\mathstrut}_{\mathrm{N}}|^a}{|v_t|^2} \frac{\dr v_{\theta} \dr v_{\varphi} \dr v_{r^*}}{ r^2\sin ( \theta )} \right|\! +\! \hspace{-0.5mm} \int_{\C}\! \left(  |h| \! + \! |\partial_t h | \right)\! |v^{\mathstrut}_{\mathrm{N}}|^{a+1} \! \dr \m_{\C}. 
\end{align*}
Similarly, using Lemma \ref{calculvstar}, $1-\frac{2M}{r} \leq 1$ as well as $2v_{r^*} \partial_u =-v_{r^*} \partial_{r^*}+ v_t \partial_t-2v_u \partial_t$ or $2v_{r^*} \partial_{\uu} =v_{r^*} \partial_{r^*}- v_t \partial_t+2v_{\uu} \partial_t$, we get
\begin{multline*}
 \left| \partial_{\uu} \int_{\C} \frac{|v_{r^*}|^2}{|v_t|^2} r^p|v_{\uu}|^q |h||v_{\uu}|  \dr \m_{\C} \right| \,  = \, \left| \int_{\C} \frac{r^p|v_{\uu}|^{q+1}}{r^2|v_t|^3} |v_{r^*}|^2 \partial_{\uu} ( |h| ) +  |v_{r^*}|^2|h| \partial_{r^*} \! \left( \frac{r^p|v_{\uu}|^{q+1}}{r^2|v_t|^3} \right)  \frac{\dr v_{\theta} \dr v_{\varphi} \dr v_{r^*}}{ \sin ( \theta )} \right| \\
 \lesssim^{\mathstrut}_q \, \left| \int_{\C} \frac{v_t \partial_t (|h|)-v_{r^*} \partial_{r^*}  (|h|)}{1-\frac{2M}{r}} \frac{v_{r^*} r^p|v_{\uu}|^{q+1}}{|v_t|^3} \frac{\dr v_{\theta} \dr v_{\varphi} \dr v_{r^*}}{ r^2\sin ( \theta )} \right| + \int_{\C} r^p|v_{\uu}|^q \left(  |h| +|\partial_t h | \right) |v_{\uu}| \dr \m_{\C}
 \end{multline*}
 and 
 \begin{align*}
 \left| \partial_{u} \int_{\C} \frac{|v_{r^*}|^2}{|v_t|^2}  r^p|v_{\uu}|^q |h||v_{u}| \dr \m_{\C} \right| \,  \lesssim^{\mathstrut}_q \, & \left| \int_{\C}\frac{v_t \partial_t (|h|)-v_{r^*} \partial_{r^*}  (|h|)}{1-\frac{2M}{r}} \frac{ v_{r^*} r^p|v_{\uu}|^q|v_{u}|}{|v_t|^3} \frac{\dr v_{\theta} \dr v_{\varphi} \dr v_{r^*}}{ r^2\sin ( \theta )} \right| \\
 &   + \int_{\C} r^p|v_{\uu}|^q\left(  |h| +|\partial_t h | \right) |v_{u}|  \dr \m_{\C}.  
\end{align*}
In order to unify the treatment of these three cases, we introduce $ \alpha (v)  \, \in \,  \left\{ |v^{\mathstrut}_{\mathrm{N}}|^a \, , \, r^p\frac{|v_{\uu}|^{q+1}}{|v_t|} \, , \,  r^p\frac{|v_{\uu}|^{q}}{|v_t|}|v_{u}| \right\}$
and let us prove
\begin{align}
\nonumber \left| \int_{\C} \frac{v_t \partial_t (h)-v_{r^*} \partial_{r^*}  (h)}{1-\frac{2M}{r}} \frac{v_{r^*}\alpha(v)}{|v_t|^2} \frac{\dr v_{\theta} \dr v_{\varphi} \dr v_{r^*}}{ r^2\sin ( \theta )} \right| \,  \lesssim^{\mathstrut}_{a,q} \, & \sum_{i=1}^3 \int_{\C} \left( |h|  + \left|\widehat{\Omega}_i h\right|\right) \! |v^{\mathstrut}_{\mathrm{N}}| \alpha(v) \dr \m_{\C} \\
&  +\int_{\C} \left| \T (h) \right| \alpha(v)  \dr \m_{\C}. \label{kev:unan}
\end{align}
As $v^{\mathstrut}_{\mathrm{N}}= v_t$ in the region $r \geq R_0$ according to Proposition \ref{energyredshift}, this will imply the three estimates of the Lemma. The starting point consists in noticing, using the definition \eqref{defT} of $\T$, that
$$\frac{v_{r^*} \partial_{r^*}  (|h|)-v_t \partial_t (|h|)}{1-\frac{2M}{r}} = \T\left( |h| \right)-\frac{v_{\theta}}{r^2}\partial_{\theta}(|h|)-\frac{v_{\varphi}}{r^2 \sin^2 ( \theta)} \partial_{\varphi}( |h| ) -\frac{\cot (\theta)  |v_{\varphi}|^2}{r^2 \sin^2 ( \theta)} \partial_{v_{\theta}}(|h|) -\frac{r-3M}{r^4} |\slashed{v}|^2 \partial_{v_r^*}(|h|).$$
Next, we perform the following estimates. 
\begin{itemize}
\item We have $ \big|  \T (|h|) \frac{v_{r^*}\alpha(v)}{|v_t|^2} \big| = \big| \frac{h}{|h|} \T (h) \frac{v_{r^*}\alpha(v)}{|v_t|^2} \big| \leq  | \T (h)|\frac{\alpha(v)}{|v_t|}$.
\item The angular part of the Vlasov operator can be written as
$$ \frac{v_{\theta}}{r^2}\partial_{\theta}(|h|)+\frac{v_{\varphi}}{r^2 \sin^2 ( \theta)} \partial_{\varphi}( |h| ) +\frac{\cot (\theta)  |v_{\varphi}|^2}{r^2 \sin^2 ( \theta)} \partial_{v_{\theta}}(|h|)=\frac{1}{r^2} \sum_{1 \leq i \leq 3} v\big(\Omega_i \big) \widehat{\Omega}_i(|h|) .$$
Moreover, there holds
$$ \forall \, 1 \leq i \leq 3, \qquad \frac{|v_{r^*} v\big(\Omega_i \big)|}{r|v_t|^2} \leq \frac{|\slashed{v}|}{r|v_t|} \lesssim \frac{|v_{\mathrm{N}}|}{|v_t|}.$$
\item By an integration by parts, 
$$  \left| \int_{\C} \frac{r-3M}{r^4} |\slashed{v}|^2 \partial_{v_r^*}(|h|) \frac{v_{r^*}\alpha(v)}{|v_t|^2} \frac{\dr v_{\theta} \dr v_{\varphi} \dr v_{r^*}}{ r^2\sin ( \theta )} \right| \, = \, \left| \int_{\C} \frac{r-3M}{r^4} |\slashed{v}|^2 |h| \partial_{v_r^*} \left( \frac{v_{r^*}\alpha(v)}{|v_t|^2} \right)\frac{\dr v_{\theta} \dr v_{\varphi} \dr v_{r^*}}{ r^2\sin ( \theta )} \right|. $$
Now, recall the definitions of $v_t$, $v_{\uu}$, $v_{u}$ and $v^{\mathstrut}_{\mathrm{N}}$ (see \eqref{defv0}, \eqref{vu} and \eqref{redshiftdef}). Then, remark
$$ \partial_{v_{r^*}}v_t \, = \, \frac{v_{r^*}}{v_t}, \qquad \partial_{v_{r^*}} v_{\uu}  \, = \, \frac{1}{2}\partial_{v_{r^*}}(v_t+v_{r^*})= \frac{v_{\uu}}{v_t}, \qquad  \partial_{v_{r^*}} v_{u} \,  = \, \frac{1}{2}\partial_{v_{r^*}}(v_t-v_{r^*})= -\frac{v_{u}}{v_t},$$
which implies $|\partial_{v_{r^*}} v^{\mathstrut}_{\mathrm{N}}| \leq \frac{|v^{\mathstrut}_{\mathrm{N}}|}{|v_t|}$ and $|\partial_{v_{r^*}} \alpha(v)| \lesssim_{a,q} \frac{\alpha(v)}{|v_t|}$. We then deduce, as $ \frac{|\slashed{v}|^2}{r^2} = 4\frac{|v_u|}{ 1-\frac{2M}{r}}|v_{\uu}| \lesssim |v^{\mathstrut}_{\mathrm{N}}||v_t|  $,
$$ \frac{|\slashed{v}|^2}{r^2}\left|\partial_{v_{r^*}} \left(\frac{ v_{r^*} \alpha(v)}{ |v_t|^2}  \right) \right| \, \lesssim^{\mathstrut}_{a,q} \, \frac{|v^{\mathstrut}_{\mathrm{N}}|\alpha(v)}{|v_t|}. $$
\end{itemize}
The previous estimates and $\dr \m_{\C}= r^{-2} \sin^{-1}(\theta)|v_t|^{-1} \dr v_{\theta} \dr v_{\varphi} \dr v_{r^*}$ imply \eqref{kev:unan}, which concludes the proof.
\end{proof}
\subsection{Weighted $L^s-L^{\infty}$ estimates for the velocity average of massless Vlasov fields}

The purpose of this subsection is to estimate pointwise the non-degenerate energy density $\mathbb{T}[f](\mathrm{N}, \mathrm{N})$ by applying Sobolev inequalities. For the reason mentionned below in Remark \ref{Rqnonlin} and in order to simplify the presentation, we will only work with solutions to $\T(f)=0$ but we could in fact prove functional inequalities adapted to the study of solutions to massless Vlasov equations in the spirit of \cite[Proposition $2.11$]{rVP}. In view of the previous subsection, we start by estimating $ \int_{\C} \frac{|v_{r^*}|^2}{|v_t|^2} |f|  |v_t||v^{\mathstrut}_{\mathrm{N}}| \dr \m_{\C}$. The next result will also be useful in order to deal with the region $\Rm_{-\infty}^0$.
\begin{Lem}\label{decayvrstar}
Let $(a,p,q) \in \R_+^3$ such that $0 \leq p \leq 2q$, $s \in [1,+\infty[$ and $f : \widehat{\Rm}_{-\infty}^{+\infty} \rightarrow \R$ be a sufficiently regular function satisfying $\T(f)=0$. For all $\tau \in \R$ and $(t,r^*,\omega) \in \Sigma_{\tau}$ such that $r^* < R_0^*$, we have
$$ \int_{\C} \frac{|v_{r^*}|^2}{|v_t|^2}  |f|^s  |v_t||v^{\mathstrut}_{\mathrm{N}}|^a \dr \m_{\C}\Big\vert_{(t,r^*,\omega)} \, \lesssim^{\mathstrut}_{a,s}  \frac{1}{r^{2}} \sum_{\substack{n+|I| \leq 3 \\ n \leq 1}} \int_{\Sigma_{\tau}} \rho \bigg[ \left|\partial_t^n \widehat{\Omega}^I f\right|^s |v^{\mathstrut}_{\mathrm{N}}|^a  \bigg]  \dr \m_{\Sigma_{\tau}} .$$
For all $\tau \in \R$ and $(t,r^*,\omega) \in \Sigma_{\tau}$ such that $r^* \geq R_0^*$, there holds
\begin{multline*}
 \int_{\C} \frac{|v_{r^*}|^2}{|v_t|^2} \frac{|v_{\uu}|^{q}}{|v_t|^{q}} |f|^s  |v_t||v^{\mathstrut}_{\mathrm{N}}|^a \dr \m_{\C}\Big\vert_{(t,r^*,\omega)} \, \lesssim^{\mathstrut}_{q,s}  \frac{1}{r^{2+p}} \sum_{\substack{n+|I| \leq 3 \\ n \leq 1}} \int_{\Sigma_{\tau}} \rho \left[ r^p \frac{|v_{\uu}|^{q}}{|v_t|^{q}}  \left|\partial_t^n \widehat{\Omega}^I f\right|^s |v_t|^a  \right]  \dr \m_{\Sigma_{\tau}} \\ 
+ \frac{1}{r^{2+p}}\sum_{\substack{n+|I| \leq 3 \\ n \leq 1}}  \int_{\underline{\N}_{t+r^*\!}}  \mathds{1}_{\tau \leq u-u_0 \leq \tau+1} \int_{\C}  r^p \frac{|v_{\uu}|^{q}}{|v_t|^{q}} \left| \partial_t^n \widehat{\Omega}^I  h \right|^s\! |v_t|^a|v_{u}| \dr \m_{\C} \dr \m_{\underline{\N}_{t+r^*\!}}.
 \end{multline*}
\end{Lem}
\begin{Rq}
Replacing $\frac{|v_{r^*}|^2}{|v_t|^2}$ by $\frac{|v_{r^*}|^{s}}{|v_t|^{s}}$, we could only require to control $\lceil \frac{3}{s} \rceil$ derivatives of $f$ but this would not provide us any additional useful information (because, in particular, of Proposition \ref{decayangu}).
\end{Rq}
\begin{proof}
Fix $(t,x^{*},\omega_0) \in \Rm_{-\infty}^{+\infty}$ and consider $\tau \in \R$ such that $(t,x^*,\omega_0) \in \Sigma_{\tau}$. In order to avoid any confusion, we emphasize that during this proof, $f$ will be viewed as a function of the variables $(t,r^*,\theta , \varphi,v_{r^*}, v_{\theta},v_{\varphi})$ and $r$ as a function of the coordinate $r^*$. 

Recall from Lemma \ref{calculvstar} that $\widehat{\Omega}^J(v_{r^*})=\widehat{\Omega}^J(v_t)=\widehat{\Omega}^J(v^{\mathstrut}_{\mathrm{N}})=0$. Consequently, Lemma \ref{Sobsphere} yields
\begin{equation}\label{eq:Sobspherep}
 \int_{\C} \frac{|v_{r^*}\!|^2}{|v_t|^2} \frac{|v_{\uu}|^{q}}{|v_t|^{q}} |f|^s  |v_t||v^{\mathstrut}_{\mathrm{N}}|^a \dr \m_{\C}\Big\vert_{(t,x^*\!, \omega_0)}  \lesssim_s^{\mathstrut} \! \sum_{|J| \leq 2} \int_{\omega \in \mathbb{S}^2} \! \int_{\C}\frac{|v_{r^*}\!|^2}{|v_t|^2}  \frac{|v_{\uu}|^{q}}{|v_t|^{q}} \left| \widehat{\Omega}^J f \right|^s \! |v_t||v^{\mathstrut}_{\mathrm{N}}|^a \dr \m_{\C}\Big\vert_{(t,x^*\!, \omega)} \dr \m_{\mathbb{S}^2}.
 \end{equation}
Fix $\omega \in \mathbb{S}^2$, $|J| \leq 2$ and denote $\widehat{\Omega}^J f$ by $h$, so that $\T(h)=0$ according to Lemma \ref{Comuprop}. Assume first that $x^* < R_0^*$ and recall from Lemma \ref{vol} that $\Sigma_{\tau} \cap \{ r^* < R_0^* \}$ can be parameterized by $(\underline{U}(r)-r^*(r)+\tau,r^*(r),\omega')$, with $2M<r < R_0$ and $\omega' \in \mathbb{S}^2$. By a $1$-dimensional Sobolev inequality, we obtain
\begin{align*}
\int_{\C} \frac{|v_{r^*}\!|^2}{|v_t|^2} |h|^s |v_t||v^{\mathstrut}_{\mathrm{N}}|^a \dr \m_{\C}\Big\vert_{(t,x^*\!, \omega)} \, \lesssim^{\mathstrut}_s \sum_{0 \leq n \leq 1}\int_{r=2M}^{R_0}  \left|\partial_{r}^n\left( \int_{\C} \frac{|v_{r^*}\!|^2}{|v_t|^2} |h|^s |v_t||v^{\mathstrut}_{\mathrm{N}}|^a  \dr \m_{\C}\Big\vert_{(\underline{U}(r)-r^*+\tau,r^*\!,\omega)} \right) \right|\dr r .
\end{align*}
Since $\underline{U}$ is a smooth function on $[2M,R_0]$, we have $|\underline{U}'| \lesssim 1$. As $\partial_r r^*= \left( 1 - \frac{2M}{r} \right)^{-1}$, we get
\begin{align*}
 \left|\partial_{r} \! \left( \int_{\C} \frac{|v_{r^*}\!|^2}{|v_t|^2} |h|^s |v_t||v^{\mathstrut}_{\mathrm{N}}|^a  \dr \m_{\C}\Big\vert_{(\underline{U}(r)-r^*+\tau,r^*\!,\omega)} \right)\right| &  \lesssim     \left|\left( \partial_t \! \int_{\C} \frac{|v_{r^*}\!|^2}{|v_t|^2} |h|^s |v_t||v^{\mathstrut}_{\mathrm{N}}|^a  \dr \m_{\C} \right) \!\Big\vert_{(\underline{U}(r)-r^*+\tau,r^*\!,\omega)}\right| \\ &+ \left|\left(\frac{1}{1-\frac{2M}{r}} \partial_u \! \int_{\C} \frac{|v_{r^*}\!|^2}{|v_t|^2} |h|^s |v_t||v^{\mathstrut}_{\mathrm{N}}|^a \dr \m_{\C} \right) \! \Big\vert_{(\underline{U}(r)-r^*+\tau,r^*\!,\omega)}\right|.
\end{align*}
Applying Lemma \ref{rdeveriv} to $|h|^s$ then gives, in view of \eqref{pderiv} and since $\T(h)=0$,
\begin{align*}
\int_{\C}\frac{|v_{r^*}\!|^2}{|v_t|^2} |h|^s |v_t||v^{\mathstrut}_{\mathrm{N}}|^a \dr \m_{\C}\Big\vert_{(t,x^*\!, \omega)}  \lesssim^{\mathstrut}_{a,s} \! \sum_{n+|I| \leq 1} \int_{r=2M}^{R_0} \! \int_{\C}  \left|\partial_t^n \widehat{\Omega}^I h \right|^s \! |v^{\mathstrut}_{\mathrm{N}}|^{a+1} \dr \m_{\C}\Big\vert_{(\underline{U}(r)-r^*+\tau,r^*\!,\omega)}\dr r .
\end{align*}
Combining the last inequality with \eqref{eq:Sobspherep}, we obtain, since $2M < r,r(x^*) \leq R_0$ on the domain considered,
\begin{align*}
|r(x^*\!)|^2\! \int_{\C} \frac{|v_{r^*}\!|^2}{|v_t|^2} |f|^s |v_t||v^{\mathstrut}_{\mathrm{N}}|^a \dr \m_{\C}\Big\vert_{(t,x^*\!, \omega)} \! \lesssim^{\mathstrut}_{a,s} \! \sum_{\substack{n+|I| \leq 3 \\ n \leq 1}} \int_{r=2M}^{R_0} \! \int_{\mathbb{S}^2} \! \int_{\C} \left|\partial_t^n \widehat{\Omega}^I f \right|^s \!|v^{\mathstrut}_{\mathrm{N}}|^{a+1}\dr \m_{\C} \dr \m_{\mathbb{S}^2}\Big\vert_{(\underline{U}(r)-r^*+\tau,r^*)} r^2\dr r .
\end{align*}
Finally, as $|v^{\mathstrut}_{\mathrm{N}}| \lesssim | v \cdot \n_{\Sigma_{\tau}}|$ on $\{r^* < R_0^*\}$ by Lemma \ref{comparo} and in view of the expression of $\dr \m_{\Sigma_{\tau}}$ (see Lemma \ref{vol}), we get
\begin{align*}
|r(x^*\!)|^2\! \int_{\C} \frac{|v_{r^*}\!|^2}{|v_t|^2} |f|^s |v_t||v^{\mathstrut}_{\mathrm{N}}|^a \dr \m_{\C}\Big\vert_{(t,x^*\!, \omega)}  \lesssim^{\mathstrut}_{a,s} \! \sum_{\substack{n+|I| \leq 3 \\ n \leq 1}} \int_{\Sigma_{\tau}} \mathds{1}_{r \leq R_0} \, \rho \left[  \left|\partial_t^n \widehat{\Omega}^I f \right|^s |v^{\mathstrut}_{\mathrm{N}}|^a \right] \dr \m_{\Sigma_{\tau}} .
\end{align*}
Suppose now that $x^* \geq R_0^*$. Here, we cannot merely integrate on $\Sigma_{\tau}$ since $v \cdot \n_{\N_{\tau}}= v_{\uu}$ does not control $|v_t|$. Instead, we use  $v_t=v_{\uu}+v_u$ and $v^{\mathstrut}_{\mathrm{N}}=v_t$ for $r \geq R_0$ (see Proposition \ref{energyredshift}), so that
\begin{align}
& \int_{\C} \frac{|v_{r^*}\!|^2}{|v_t|^2} \frac{|v_{\uu}|^{q}}{|v_t|^{q}} |h|^s |v_t||v^{\mathstrut}_{\mathrm{N}}|^a \dr \m_{\C}\Big\vert_{(t,x^*\!, \omega)} \, \leq \,  H_1(t,x^*\!,\omega) + H_2(t,x^*\!,\omega), \label{eq:daouugent} \\ \nonumber
 & H_1(t,x^* \!, \omega)    := \! \int_{\C} \! \frac{|v_{r^*}\!|^2}{|v_t|^2}  \frac{|v_{\uu}|^{q}}{|v_t|^{q}} |h|^s |v_t|^a |v_{\uu}| \dr \m_{\C}\Big\vert_{(t,x^*\!, \omega)}\!, \qquad H_2(t,x^* \!, \omega)  := \! \int_{\C} \! \frac{|v_{r^*}\!|^2}{|v_t|^2}  \frac{|v_{\uu}|^{q}}{|v_t|^{q}} |h|^s |v_t|^a|v_u| \dr \m_{\C}\Big\vert_{(t,x^*\!, \omega)}
\end{align}
and we bound $H_1$ and $H_2$ separately. By one dimensional Sobolev inequalities, we have
\begin{align*}
  H_1(t,x^*\!, \omega)  & \, = \, -  \int_{\underline{u}=t+x^*}^{+\infty} \! \partial_{\uu} \left( H_1 \Big(  \frac{1}{2}\uu + \frac{1}{2}(t-x^*),\frac{1}{2}\uu-\frac{1}{2}(t-x^*), \omega \Big) \right)  \dr \uu , \\
  H_2 (t,x^*\!, \omega) & \, \lesssim \, \sum_{n \leq 1}  \int_{u=t-x^*}^{t+x^*\!+1} \! \left| \partial^n_{u} \left( H_2 \Big(  \frac{1}{2}(t+x^*)+ \frac{1}{2}u,\frac{1}{2}(t+x^*)-\frac{1}{2}u, \omega \Big) \right) \right| \dr u.
\end{align*}
Note that $r(x^*) \leq r$ on the domain of integration of the first integral and $r^* \geq x^*-\frac{1}{2}$, so that $R_0 \leq r(x^*) \lesssim  r$, on the one of the second integral. Thus, applying Lemma \ref{rdeveriv} to $|h|^s|v_t|^{a}$, we obtain, using \eqref{pderiv} and since $\partial_t (v_t)=\widehat{\Omega}_i(v_t)=\T(|h|^s|v_t|^a)=0$,
\begin{align*}
|r(x^*)|^{2+p} H_1 ( t,x^* \! , \omega)  & \lesssim^{\mathstrut}_{q,s} \! \sum_{n+|I| \leq 1 }  \int_{\underline{u}=t+x^*}^{+\infty} \! \int_{\C}   r^p \frac{|v_{\uu}|^{q}}{|v_t|^{q}} \! \left| \partial_t^n \widehat{\Omega}^I h \right|^s \! |v_t|^a|v_{\uu}| \dr \m_{\C}\Big\vert_{\big(\frac{1}{2}\uu + \frac{1}{2}(t-x^*),\frac{1}{2}\uu-\frac{1}{2}(t-x^*), \omega \big)} r^2 \dr \uu, \\
|r(x^*)|^{2+p} H_2 ( t,x^* \! , \omega)  & \lesssim^{\mathstrut}_{q,s} \! \sum_{n+|I| \leq 1 }  \int_{u=t-x^*}^{t-x^*+1} \! \hspace{-0.4mm} \int_{\C}   r^p \frac{|v_{\uu}|^{q}}{|v_t|^{q}} \! \left| \partial_t^n \widehat{\Omega}^I h \right|^s \! |v_t|^a|v_{u}| \dr \m_{\C}\Big\vert_{\big(\frac{1}{2}(t+r^*) + \frac{1}{2}u,\frac{1}{2}(t+r^*)-\frac{1}{2}u, \omega \big)} r^2 \dr u.
\end{align*}
Note now that $t-x^*=\tau+u_0$ since $x^* \geq R_0$ and $(t,x^* , \omega_0) \in \Sigma_{\tau}$. Recall from Subsection \ref{subsec1} that
\begin{itemize}
\item $\Sigma_{\tau} \cap \{ r^* \geq R_0^* \} = \N_{\tau}$, $v \cdot \n_{\N_{\tau}} =v_{\uu}$, $\dr \N_{\tau} = r^2 \dr \underline{u} \wedge \dr \m_{\mathbb{S}^2}$ and
\item $\dr \underline{\N}_{t+x^*} = r^2 \dr u \wedge \dr \m_{\mathbb{S}^2}$. Moreover, $t-x^* \leq u \leq t-x^*+1 \Leftrightarrow \tau \leq u-u_0 \leq \tau+1$. 
\end{itemize}
We then deduce the second estimate of the Proposition from \eqref{eq:Sobspherep}, \eqref{eq:daouugent} and the last two inequalities.
\end{proof}
This allows us to deduce the following estimate for the region $\tau \geq 0$.
\begin{Pro}\label{decayint}
Let $(a,p,q) \in \R_+^3$ such that $0 \leq p \leq 2q$ and $s \in [1,+\infty[$. Consider a sufficiently regular function $f : \widehat{\Rm}_{0}^{+\infty} \rightarrow \R$ satisfying $\T(f)=0$. There holds, for all $\tau \geq 0$ and $(t,r^*,\omega) \in \Sigma_{\tau}$,
$$ \int_{\C} \frac{|v_{r^*}|^2}{|v_t|^2} \frac{|v_{\uu}|^q}{|v_t|^q} |f|^s  |v_t||v^{\mathstrut}_{\mathrm{N}}|^a \dr \m_{\C}\Big\vert_{(t,r^*,\omega)} \, \lesssim^{\mathstrut}_{a,q,s}  \frac{1}{r^{2+p}} \sum_{\substack{n+|I| \leq 3 \\ n \leq 1}} \int_{\Sigma_{\tau}} \rho \left[ \left( 1+r^p\frac{|v_{\uu}|^q}{|v_t|^q} \right) \left|\partial_t^n \widehat{\Omega}^I f\right|^s |v^{\mathstrut}_{\mathrm{N}}|^a  \right]  \dr \m_{\Sigma_{\tau}} .$$
\end{Pro}
\begin{proof}
If $r^* < R_0^*$, then remark that $|v_{\uu}|^q \leq |v_t|^q$, so that the result ensues from Lemma \ref{decayvrstar} and $r^{-2} \lesssim^{\mathstrut}_q r^{-2-p}$. Otherwise $r^* \geq R_0^*$ and we also apply Lemma \ref{decayvrstar}. We bound the flux through the piece of the hypersurface $\underline{\N}_{t+r^*}$ using
\begin{itemize}
\item Remark \ref{divergence0Rq}, applied with $\tau_1=\tau$ and $\tau_2$ sufficiently large, if $r^* \leq R_0^*+1$ since $r^p |v_{\uu}|^q \lesssim^{\mathstrut }_q |v_t|^q$.
\item Remark \ref{energydecayintRq}, applied with $\tau_1=\tau$ and $\tau_2=\tau+1$, if $r^* \geq R_0^*+1$, since in that case we have the inclusion $\underline{\N}_{t+r^*} \cap \{ \tau \leq u-u_0 \leq \tau+1 \} \subset \{ r \geq R_0 \}$ or, equivalently, $t+r^* \geq \tau+1+u_0+2R_0^*$.
\end{itemize}
It then remains to use that $|v_t| \leq |v^{\mathstrut}_{\mathrm{N}}|$ (see \eqref{redshiftdef}).
\end{proof}
The next lemma will permit us to estimate $ \int_{\C} |f|  |v^{\mathstrut}_{\mathrm{N}}|^2 \dr \mu_{\C}$ by $L^s$ norms of $f$  through an application of the previous result.
\begin{Lem}\label{decayangu}
Let $s \in \, ]3,4]$, $\delta>0$ and $f : \widehat{\Rm}_{-\infty}^{+\infty} \rightarrow \R$ be a sufficiently regular function. Then, for any $q \in \R_+$,
$$ \int_{\C}  \frac{|v_{\uu}|^q}{|v_t|^q} |f| |v^{\mathstrut}_{\mathrm{N}}|^2 \dr \m_{\C} \, \lesssim^{\mathstrut}_{s,\delta} \,  \frac{r^{\frac{4}{s}}}{r^{2\frac{s-1}{s}}} \left| \int_{\C} \frac{|v_{r^*}|^2}{|v_t|^2}  \frac{|v_{\uu}|^{sq+2}}{|v_t|^{sq+2}} \left| f \, |v_t|^{\frac{4-s}{s}} \langle v_t \rangle^{\frac{s-3+\delta}{s}} \langle \slashed{v} \rangle^{2\frac{s-3+\delta}{s}} \right|^{s}\! |v_t| |v^{\mathstrut}_{\mathrm{N}}|^{2(s+1)}  \dr \m_{\C} \right|^{\frac{1}{s}} \! .$$
\end{Lem}
\begin{proof}
We start by applying the Hölder inequality. As $\frac{s}{s-1}$ is the conjugate exponent of $s$, we have
\begin{equation}\label{eq:seitek}
 \int_{\C}  \frac{|v_{\uu}|^q}{|v_t|^q} |f| |v^{\mathstrut}_{\mathrm{N}}|^2 \dr \m_{\C} \, \lesssim \, \left| \mathcal{I} \right|^{\frac{1}{s}} \cdot \left| \mathcal{J} \right|^{\frac{s-1}{s}},
 \end{equation}
where
\begin{align*}
\mathcal{I} \, & := \,  \int_{\C} \left|  \frac{|v_{\uu}|^q}{|v_t|^q} |f| |v^{\mathstrut}_{\mathrm{N}}|^2 \right|^s \frac{|v_{r^*}|^2 |\slashed{v}|^4(1+|v_{r^*}|)^{s-3+\delta}(1+|\slashed{v}|)^{2(s-3+\delta)}}{|v_t|^{s-1}} \dr \m_{\C} , \\
\mathcal{J} \, & := \,  \int_{\C} \frac{|v_t|}{|v_{r^*}|^{\frac{2}{s-1}}(1+|v_{r^*}|)^{\frac{s-3+\delta}{s-1}} |\slashed{v}|^{\frac{4}{s-1}}(1+|\slashed{v}|)^{2\frac{s-3+\delta}{s-1}}} \dr \m_{\C} .
\end{align*}
Note now that since $\dr \m_{\C} = r^{-2} \sin^{-1} (\theta) |v_t|^{-1} \dr v_{r^*} \dr v_{\theta} \dr v_{\varphi}$ and $\sqrt{2}|\slashed{v}| \geq |v_{\theta}|+\sin^{-1}(\theta) |v_{\varphi}|$, 
\begin{align*}
\mathcal{J} \, \lesssim \, \frac{1}{r^2} \int_{v_{r^*} \in \R} \frac{\dr v_{r^*}}{|v_{r^*}|^{\frac{2}{s-1}} (1+|v_{r^*}|)^{\frac{s-3+\delta}{s-1}} } \int_{(v_{\theta}, v_{\varphi}) \in \R^2}\frac{\dr v_{\theta} \dr \frac{v_{\varphi}}{\sin (\theta)}}{ \big(|v_{\theta}|+\frac{|v_{\varphi}|}{\sin (\theta)} \big)^{\frac{4}{s-1}} \big( 1+|v_{\theta}|+\frac{|v_{\varphi}|}{\sin (\theta)} \big)^{2\frac{s-3+\delta}{s-1}}} .
\end{align*}
We then deduce, as $s>3$, that $\mathcal{J} \lesssim_{s,\delta} r^{-2}$. We now turn to $\mathcal{I}$ and we recall that $|v_{r^*}| \leq |v_t|$ and $\left( 1 - \frac{2M}{r} \right)\frac{|\slashed{v}|^2}{r^2} =4|v_{\uu}| |v_u|$ (see \eqref{vu}). Hence, in view of the definition \eqref{redshiftdef} of $v^{\mathstrut}_{\mathrm{N}}$, we have $|\slashed{v}|^2 \lesssim r^2 |v_{\uu}| |v^{\mathstrut}_{\mathrm{N}}|$, so that
$$ \mathcal{I} \, \lesssim^{\mathstrut} \, \int_{\C} \left|  \frac{|v_{\uu}|^q}{|v_t|^q} |f| |v^{\mathstrut}_{\mathrm{N}}|^2 \right|^s \frac{|v_{r^*}|^2}{|v_t|^2} \frac{r^4|v_{\uu}|^2}{|v_t|^2} |v_t||v^{\mathstrut}_{\mathrm{N}}|^2 \frac{(1+|v_t|)^{s-3+\delta}(1+|\slashed{v}|)^{2(s-3+\delta)}}{|v_t|^{s-4}} \dr \m_{\C} .$$
The result then follows from \eqref{eq:seitek},  $\mathcal{J} \lesssim_{s,\delta} r^{-2}$ and this last estimate.
\end{proof}
We are now able to prove pointwise decay estimates for the non-degenerate energy density $\mathbb{T}[f](\mathrm{N}, \mathrm{N})$.
\begin{Pro}\label{pointwise}
Let $q \in \R_+$, $z \in \, ]3,4]$, $\delta>0$ and $f : \widehat{\Rm}_{0}^{+\infty} \rightarrow \R$ be a sufficiently regular function. Then, for all $\tau \geq 0$ and $(t,r^*,\omega) \in \Sigma_{\tau}$,
\begin{align*}
 \int_{\C} & \frac{|v_{\uu}|^q}{|v_t|^q} |f| |v^{\mathstrut}_{\mathrm{N}}|^2 \dr \m_{\C}\Big\vert_{(t,r^*,\omega)} \\ &\, \lesssim^{\mathstrut}_{q,z,\delta}   \frac{1}{r^{2+2q}} \sum_{\substack{n+|I| \leq 3 \\ n \leq 1}} \left| \int_{\Sigma_{\tau}} \rho  \left[ \left(  1+ r^{2zq+4} \frac{|v_{\uu}|^{zq+2}}{|v_t|^{zq+2}} \right) \! \left| \partial_t^n \widehat{\Omega}^I f \right|^z  \big\langle |v_t|+|\slashed{v}| \big\rangle^{3(z-3+\delta)} |v_t|^{4-z} |v^{\mathstrut}_{\mathrm{N}}|^{2(z+1)}  \right]  \dr \m_{\Sigma_{\tau}}\right|^{\frac{1}{z}} \! .
 \end{align*}
\end{Pro}
\begin{proof}
Apply first Lemma \ref{decayangu} to $f$ with $s=z$ and then Proposition \ref{decayint} to $f |v_t|^{\frac{z-4}{z}} \langle v_t \rangle^{\frac{z-3+\delta}{s}} \langle \slashed{v} \rangle^{2\frac{z-3+\delta}{z}}$, for the parameters $s=z$, $a=2(z+1)$, $p=2zq+4$ and, by an abuse of notation, $q=zq+2$. Finally, use $\langle v_t \rangle \langle \slashed{v} \rangle^2 \lesssim \langle |v_t|+|\slashed{v}| \rangle^3$.
\end{proof}
\begin{Rq}\label{Rqplustardaussi}
Note that the assumption $z \leq 4$ is important. Indeed, if $z>4$, $|v_t|$ would carry a negative exponent and this would force us to assume stronger vanishing properties in the variable $v_{\uu}$ on $f$ near the event horizon.

By a similar heuristic analysis as the one carried out in Remark \ref{Rqlabelplustard}, one can check that
\begin{itemize}
\item $\int_{\Sigma_{\tau}} \! \rho [ |f| |v^{\mathstrut}_{\mathrm{N}}| ] \dr \m_{\Sigma_{\tau}} \! < +\infty$ schematically implies that, for $|v_{r^*}|+|\slashed{v}| \to + \infty$, $|f| = \mathcal{O} ( (|v_{r^*}|+|\slashed{v}|)^{-4} )$ and, for $r \to +\infty$, $|f|= \mathcal{O} (r^{-1})$.
\item The finiteness of the integral on the right hand side of the estimate of Proposition \ref{pointwise} provides, in the case $q=0$, $|f| = \mathcal{O} \big( (|v_{r^*}|+|\slashed{v}|)^{-4-\frac{3\delta}{z}} \big)$ and $|f|=\mathcal{O} (r^{-3+\frac{4+3\delta}{z}})$.
\end{itemize} 
\end{Rq}
\begin{Rq}\label{Rqnonlin}
In the context of a non linear problem, $\T(f)$ could be schematically of the form $\psi \cdot h$, where $\psi$ is a solution to a wave equation. Consequently, it could be difficult to deal with $|\psi|^s$ for $s >2$ if $\psi$ cannot be estimated pointwise. For this reason, we considered only solutions to $\T(f)=0$. Note however that with more informations on $\T(f)$, we could extend our results to a more general setting. For instance, if $s=1$, $p=0$ and $a=1$, one can obtain by following the proof of Proposition \ref{decayvrstar} that, for all $\tau \geq 0$ and $(t,r^*,\omega) \in \Sigma_{\tau}$, 
\begin{align*}
r^2 \! \int_{\C} \frac{|v_{r^*}|^2}{|v_t|^2} |f|  |v_t||v^{\mathstrut}_{\mathrm{N}}| \dr \m_{\C}\Big\vert_{(t,r^*,\omega)}\lesssim  \sum_{n+|I| \leq 3 } \int_{\Sigma_{\tau}} \rho \left[  \left|\partial_t^n \widehat{\Omega}^I f\right| |v^{\mathstrut}_{\mathrm{N}}| \right]  \dr \m_{\Sigma_{\tau}} +\sum_{|J| \leq 2}  \int_{\Sigma_{\tau}} \int_{\C}  \left|\T \left( \widehat{\Omega}^J f\right) \right| |v^{\mathstrut}_{\mathrm{N}}| \dr \m_{\C} \dr \m_{\Sigma_{\tau}} \\ + \sum_{n+|I| \leq 3 } \sum_{|J| \leq 2}\int_{\tau'=\tau}^{\tau+1} \! \int_{\Sigma_{\tau'}} \int_{\C} \left( \left|\T \left( \partial_t^n \widehat{\Omega}^I f\right) \right| |v_t|+ \left|\T \left( \T \left( \widehat{\Omega}^J f \right)\right) \right| \right) \dr \m_{\C} \dr \m_{\Sigma_{\tau'}} \dr \tau'.
\end{align*}
The bulk integral arises from an application of Remark \ref{divergence0Rq} in order to bound the flux on $\underline{\N}_{t+r^*}$.
\end{Rq}

\section{The region $r^* \geq R_0+t$}\label{sec6}

As in Minkowski spacetime, Vlasov fields behave better in such a region, which corresponds to the exterior of a light cone. One can already see that if $f$ satisfies $\T(f)=0$ and is initially compactly supported in $\mathcal{S} \cap \{r^* \leq R_0^* \}$ as, in that case, $f$ vanishes in $\Rm_{-\infty}^{0}$. This means in particular that $\int_{\Sigma_{\tau}} \rho \big[ |f| |v^{\mathstrut}_{\mathrm{N}}| \big] \dr \m_{\Sigma_{\tau}} = 0$ for all $\tau \leq 0$. More generally, solutions to the Vlasov equation can be studied in $\{r^* \geq R_0^*+t \}$ without any information on their behaviour in the remaining part of Schwarzschild spacetime since this region is globally hyperbolic. This will simplify the analysis compared to the domain $\{ r^* \leq R_0^*+t\}$. In particular, no difficulty related to the photon sphere or the event horizon will arise here.

\begin{Pro}[energy decay estimates for the region $\tau \leq 0$]\label{energydecayexter}
Let $(d,p) \in \R_+^2$, $a \in \R$ and $f : \widehat{\Rm}_{-\infty}^0 \rightarrow \R$ be a sufficiently regular function such that $\T(f)=0$. For all $\tau \leq 0$, there holds
$$\int_{\N_{\tau}} \int_{\C} r^{p} \frac{|v_{\uu}|^{p/2}}{|v_t|^{p/2}} |f||v_t|^a |v_{\uu}|  \dr \m_{\N_{\tau}} \, \lesssim^{\mathstrut}_{d,p} \frac{1}{(1+|\tau|)^{d}} \int_{\mathcal{S}} \mathds{1}_{r \geq R_0} \int_{\C} r^{p}  \frac{|v_{\uu}|^{p/2}}{|v_t|^{p/2}}(1+|u|)^d |f||v_t|^a | v \cdot \n_{\mathcal{S}}| \dr \m_{\C}  \dr \m_{\mathcal{S}}.$$
\end{Pro}
\begin{Rq}\label{energydecayextRq} We will also use the following inequality. For all $\tau' \leq 1$,
$$ \sup_{\underline{w} \in \R}   \int_{\underline{\N}_{\underline{w}}} \! \mathds{1}_{u-u_0 \leq \tau'}  \int_{\C}  r^{p} \frac{|v_{\uu}|^{\frac{p}{2}}}{|v_t|^{\frac{p}{2}}} |f||v_t|^a |v_{u}|  \dr \m_{\underline{\N}_{\underline{w}}}  \lesssim^{\mathstrut}_{d,p} \frac{1}{(1+|\tau'|)^{d}}  \int_{\mathcal{S}}  \! \int_{\C} \! r^{p}  \frac{|v_{\uu}|^{\frac{p}{2}}}{|v_t|^{\frac{p}{2}}} (1+u_-)^{d} |f||v_t|^a | v \cdot \n_{\mathcal{S}}| \dr \m_{\C}  \dr \m_{\mathcal{S}},$$
where $u_- := \max(0,-u)$.
\end{Rq}
\begin{proof}
Even if $\tau =u-u_0$ in the region studied here, we will abusively consider $\tau$ as a fixed parameter during this proof. Note that since $h=f|v_t|^a$ also satisfies $\T(h)=0$ in view of Lemma \ref{Comuprop}, it suffices to treat the case $a=0$. Fix $\tau \leq 0$, $\underline{w} \in \R$ and introduce the set
$$ \mathcal{D}( \tau, \underline{w} ) \, = \, \{ (t,r^*,\omega) \in \R \times \R \times \mathbb{S}^2 \, / \,  \, u \leq \tau+u_0 , \, \underline{u} \leq \underline{w} \} \cap \Rm_{-\infty}^{+\infty},$$
which is non empty if $\underline{w}$ is sufficiently large, and remark that its boundary is composed by
\begin{align*}
 \mathcal{N}_{\tau} \cap \{ \uu \leq \underline{w} \}, \qquad  \underline{\N}_{\underline{w}} \cap \{ u \leq \tau +u_0\}, \qquad   \mathcal{S}^{\mathstrut} \cap \{\uu \leq \underline{w}, \, u \leq \tau+u_0 \}.
\end{align*}

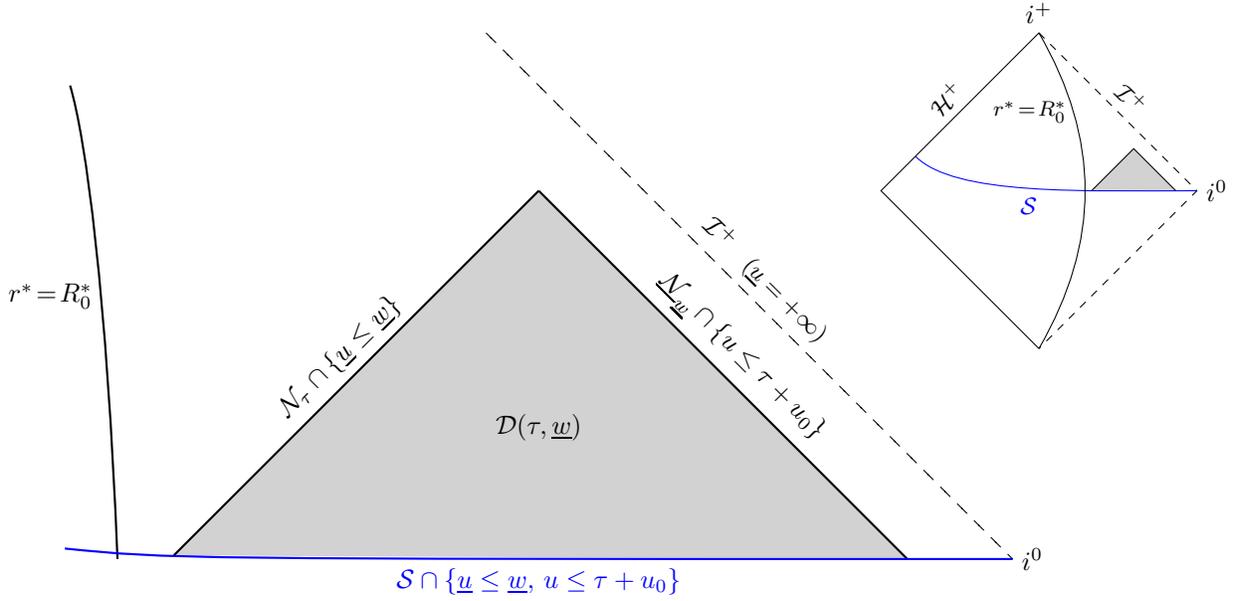
\begin{figure}[H]
\begin{center}
\begin{tikzpicture}[scale=0.7]
\node (II)   at (-1,0)   {};
\path  
  (II) +(90:10)  coordinate  (IItop)
       +(0:10)   coordinate[label=0:$i^0$]  (IIright)
       ;
\draw[loosely dashed] 
      (IItop) --
          node[midway, above, sloped] {$\cal{I}^+$ \, \small{($\underline{u}=+\infty$)}}
      (IIright) -- cycle;
\fill[color=gray!35] (-6.93,0.07) --(0,7)--(7,0);
\draw[thick] (-6.93,0.07) -- node[midway,above,sloped]{$\N_{\tau}\cap \{ \uu \leq \underline{w} \}$} (0,7) ;
\draw[thick] (0,7) -- node[midway,above,sloped]{$\underline{\N}_{\underline{w}}\cap \{ u \leq \tau+u_0 \}$} (7,0);
\draw[thick]   (-8,0) .. controls (-8.2,5) and (-8.6,8) .. (-8.9,9);
\draw[thin,blue] (0,0) node[below]{$\mathcal{S} \cap \{ \uu \leq \underline{w}, \, u \leq \tau+u_0 \}$};
\draw (0,2.5) node{$\mathcal{D}(\tau,\underline{w})$};
\draw (-8.3,5) node[left]{$r^*\!=\!R_0^*$};
\draw[blue,thick]   (-9,0.2) .. controls (-7,0)  .. (9,0);

\node (I)   at (9.5,7)   {};
\path  
  (I) +(90:3)  coordinate[label=90:$i^+$]  (Itop)
       +(-90:3) coordinate (Ibot)
       +(180:3) coordinate (Ileft)
       +(0:3)   coordinate[label=0:$i^0$]  (Iright)
       ;
\draw (Ibot) --      (Ileft) -- 
        node[midway, above, sloped] {\small{$\cal{H}^+$}}
      (Itop);
\draw[ dashed]   (Itop)    --
          node[midway, above, sloped] {\small{$\cal{I}^+$} }
      (Iright) -- (Ibot);  
\fill[color=gray!35] (10.5,7)--(11.3,7.8)--(12.1,7);   
\draw  (10.5,7)--(11.3,7.8)--(12.1,7); 
\draw[thin,blue]  (7.1575,7.6525) .. controls  (7.9025,6.9025) and (10.37,7).. node[below]{\small{$\mathcal{S}$}}   (12.5,7)  ;
\draw (9.5,10) to [bend left] (9.5,4);
\draw (10.2,8.5) node[left]{\footnotesize{$r^*\!=\! R_0^*$}};
\end{tikzpicture}
\end{center}
\caption{The set $\mathcal{D}(\tau,\underline{w})$ and its boundary.}
\end{figure}
We fix, for all the proof, $q \in \R_+$ and we consider $0 \leq p \leq 2q$. Recall that $v \cdot \n_{\N_{\tau}} = v_{\uu} \leq 0$, $v \cdot \n_{\underline{\N}_{\underline{w}}} = v_{u} \leq 0$ and $v \cdot \n_{\mathcal{S}} < 0$ since $\n_{\mathcal{S}}$ is timelike and future oriented. An application of the divergence theorem to the current $N \big[ r^p \frac{|v_{\uu}|^q}{|v_t|^q}|f|  \big]_{\mu}$ in the domain $\mathcal{D}(\tau_,\underline{w})$ then leads to
\begin{multline*}
\int_{\N_{\tau}} \mathds{1}_{ \uu \leq \underline{w}} \int_{\C} r^p \frac{|v_{\uu}|^q}{|v_t|^q} |f||v_{\uu}| \dr \m_{\C} \dr \m_{\N_{\tau}}+\int_{\underline{\N}_{\underline{w}}} \mathds{1}_{u-u_0 \leq \tau}  \int_{\C} r^p \frac{|v_{\uu}|^q}{|v_t|^q} |f||v_u| \dr \m_{\C} \dr \m_{\underline{\N}_{\underline{w}}} \\ - \int_{\mathcal{S}} \mathds{1}_{ \uu \leq \underline{w}} \, \mathds{1}_{u-u_0 \leq \tau} \int_{\C} r^p \frac{|v_{\uu}|^q}{|v_t|^q} |f||v \cdot \n_{\mathcal{S}}| \dr \m_{\C} \dr \m_{\mathcal{S}} \, = \,
\int_{\mathcal{D}(\tau,\underline{w})} \int_{\C} \T \left( r^p \frac{|v_{\uu}|^q}{|v_t|^q} \,|f| \right)   \dr \m_{\C} \dr \m_{\mathcal{D}(\tau,\underline{w})} .
\end{multline*}
Since $u-u_0 \leq \tau \leq 0$ implies $|u|+|u_0| \geq - \tau \geq 0$ as well as $r \geq R_0$ on $\mathcal{S }$, there holds
\begin{multline*}
 \int_{\mathcal{S}} \mathds{1}_{ \uu \leq \underline{w}} \, \mathds{1}_{u-u_0 \leq \tau} \int_{\C} r^p \frac{|v_{\uu}|^q}{|v_t|^q} |f||v \cdot \n_{\mathcal{S}}| \dr \m_{\C} \dr \m_{\mathcal{S}} \, 
  \lesssim^{\mathstrut}_d   \int_{\mathcal{S}} \frac{(1+|u|)^d}{(1+|\tau|)^d} \mathds{1}_{u-u_0 \leq \tau} \int_{\C} r^p \frac{|v_{\uu}|^q}{|v_t|^q} |f||v \cdot \n_{\mathcal{S}}| \dr \m_{\C} \dr \m_{\mathcal{S}} \\
\lesssim^{\mathstrut}_{d,q}   \mathbf{I}_{\mathcal{S}} \, := \,\frac{1}{(1+|\tau|)^d}\int_{\mathcal{S}} \mathds{1}_{ r \geq R_0} \int_{\C} r^{2q} \frac{|v_{\uu}|^q}{|v_t|^q}(1+|u|)^d |f||v \cdot \n_{\mathcal{S}}| \dr \m_{\C} \dr \m_{\mathcal{S}}.
\end{multline*}
 In order to lighten the notations, we also introduce
 $$ \mathbf{F}^{p,q}_{\tau, \underline{w}} \, := \, \int_{\underline{\N}_{\underline{w}}} \mathds{1}_{u-u_0 \leq \tau}  \int_{\C} r^p \frac{|v_{\uu}|^q}{|v_t|^q} |f||v_u| \dr \m_{\C} \dr \m_{\underline{\N}_{\underline{w}}}.$$
As $\T(|f|)= 0$, we then deduce, for all $p \in [0,2q]$, the following $r$-weighted energy estimates
$$\int_{\N_{\tau}} \mathds{1}_{ \uu \leq \underline{w}} \int_{\C} r^p \frac{|v_{\uu}|^{q}}{|v_t|^{q}} |f| |v_{\uu}| \dr \m_{\C} \dr \m_{\N_{\tau}}+ \mathbf{F}^{p,q}_{\tau, \underline{w}}-\int_{\mathcal{D}(\tau,\underline{w})} \int_{\C} \T \left( r^p \frac{|v_{\uu}|^q}{|v_t|^q}  \right)  |f| \dr \m_{\C} \dr \m_{\mathcal{D}(\tau,\underline{w})} \, \lesssim^{\mathstrut}_q \mathbf{I}_{\mathcal{S}} .$$
We now apply these last inequalities together with Lemma \ref{Lemhierar} in three different settings. First, if $p=2q$,
$$\int_{\N_{\tau}} \mathds{1}_{ \uu \leq \underline{w}} \int_{\C} r^{2q} \frac{|v_{\uu}|^{q}}{|v_t|^{q}} |f| |v_{\uu}| \dr \m_{\C} \dr \m_{\N_{\tau}}  + \mathbf{F}^{2q,q}_{\tau,\underline{w}}
 \lesssim^{\mathstrut}_q \int_{\mathcal{D}(\tau,\underline{w})} \int_{\C} r^{2q-1} \frac{|v_{\uu}|^{q-1}}{|v_t|^q} \frac{|\slashed{v}|^2}{r^2} |f| \dr \m_{\C} \dr \m_{\mathcal{D}(\tau,\underline{w})}+\mathbf{I}_{\mathcal{S}}  .$$
The goal now is to bound sufficiently well the first term on the right hand side of the previous inequality. For this, we exploit the hierarchy in $p$ given by the $r$-weighted energy estimates. Applying it for $p = 2q -n$, with $n \in \llbracket 1, \lfloor 2q \rfloor \rrbracket$, so that $2q-p \geq 1$ and $3q-p \leq 3q$, we get
$$ \int_{\mathcal{D}(\tau,\underline{w})} \int_{\C} r^{p-1} \frac{|v_{\uu}|^{q-1}}{|v_t|^q} \frac{|\slashed{v}|^2}{r^2} |f| \dr \m_{\C} \dr \m_{\mathcal{D}(\tau,\underline{w})} \lesssim^{\mathstrut}_q \int_{\mathcal{D}(\tau,\underline{w})} \int_{\C} r^{p-2} \frac{|v_{\uu}|^{q-1}}{|v_t|^q} \frac{|\slashed{v}|^2}{r^2} |f| \dr \m_{\C} \dr \m_{\mathcal{D}(\tau,\underline{w})}+ \mathbf{I}_{\mathcal{S}} $$
and applying it for $p=0$, we obtain, since $r^{2q-\lfloor 2q \rfloor-2} \lesssim r^{-1}$,
$$\int_{\mathcal{D}(\tau,\underline{w})}  \int_{\C} r^{2q-\lfloor 2q \rfloor-2} \frac{|v_{\uu}|^{q-1}}{|v_t|^q} \frac{|\slashed{v}|^2}{r^2} |f| \dr \m_{\C} \dr \m_{\mathcal{D}(\tau,\underline{w})} \lesssim \int_{\mathcal{D}(\tau,\underline{w})}  \int_{\C} r^{-1} \frac{|v_{\uu}|^{q-1}}{|v_t|^q} \frac{|\slashed{v}|^2}{r^2} |f| \dr \m_{\C} \dr \m_{\mathcal{D}(\tau,\underline{w})}  \lesssim^{\mathstrut}_q  \mathbf{I}_{\mathcal{S}} .$$
The combination of the last estimates yields to
\begin{equation*}
\int_{\N_{\tau}} \mathds{1}_{ \uu \leq \underline{w}} \int_{\C} r^{2q} \frac{|v_{\uu}|^{q}}{|v_t|^{q}} |f| |v_{\uu}| \dr \m_{\C} \dr \m_{\N_{\tau}}+\mathbf{F}^{p,q}_{\tau,\underline{w}} \, \lesssim^{\mathstrut}_q \mathbf{I}_{\mathcal{S}} .
\end{equation*}
An application of Beppo-Levi's theorem provides us the estimate of the proposition. As this inequality holds for all $\underline{w} \in \R$, we obtain the estimate of Remark \ref{energydecayextRq} for $\tau' \leq 0$. One can prove it for $0 < \tau' \leq 1$ by applying the results proved here to the foliation $(\widetilde{\N}_{\widetilde{\tau}})_{\widetilde{\tau} \leq 0}$, defined similarly as $(\N_{\tau})_{\tau \leq 0}$ but where $\widetilde{u}_0=u_0+1$.
\end{proof}
We now turn to the pointwise decay estimates.
\begin{Pro}\label{pointwiseext}
Let $(d,q) \in \R_+^2$, $z >3$, $\delta>0$ and $f : \widehat{\Rm}_{-\infty}^{+\infty} \rightarrow \R$ be a sufficiently regular function. Then, for all $\tau < 0$ and $(t,r^*,\omega) \in \Sigma_{\tau}$,
\begin{align*}
& \int_{\C} \frac{|v_{\uu}|^q}{|v_t|^q} |f| |v^{\mathstrut}_{\mathrm{N}}|^2 \dr \m_{\C}\Big\vert_{(t,r^*,\omega)} \, \lesssim^{\mathstrut}_{q,z,\delta} \\
 &  \frac{1}{r^{2+2q}(1+|\tau|)^{d}}  \sum_{\substack{n+|I| \leq 3 \\ n \leq 1}}   \left| \int_{\mathcal{S}}  \! \int_{\C}  r^{2zq+4} \frac{|v_{\uu}|^{zq+2}}{|v_t|^{zq+2}}(1+u_-)^{zd} \left| \partial_t^n \widehat{\Omega}^I f \right|^z \! \big\langle |v_t|+|  \slashed{v}| \rangle^{3(z-3+\delta)}  |v^{\mathstrut}_{\mathrm{N}}|^{z+6}  | v \cdot \n_{\mathcal{S}}| \dr \m_{\C}  \dr \m_{\mathcal{S}} \right|^{\frac{1}{z}} \! .
 \end{align*}
\end{Pro}
\begin{Rq}
By a slight modification of our proof, we could control the velocity average of $f$ by weighted $L^s$ norms of $f$ over the domain $\mathcal{S} \cap \{ r \geq R_0 \}$.
\end{Rq}
\begin{proof}
Fix $\tau <0$ and consider $(t,r^*,\omega) \in  \Sigma_{\tau}=\N_{\tau}$. Apply first Lemma \ref{decayangu} to $f$ with $s=z$ and then Proposition \ref{decayvrstar} to $f |v_t|^{\frac{z-4}{z}} \langle v_t \rangle^{\frac{z-3+\delta}{z}} \langle \slashed{v} \rangle^{2\frac{z-3+\delta}{z}}$, for the parameters $s=z$, $a=2(z+1)$, $p=2zq+4$ and, by an abuse of notation, $q=zq+2$. Since $\langle v_t \rangle \langle \slashed{v} \rangle^2 \lesssim \langle |v_t|+|\slashed{v}| \rangle^3$, we can then estimate $r^{2+2q}\int_{\C}  \frac{|v_{\uu}|^q}{|v_t|^q} |f| |v^{\mathstrut}_{\mathrm{N}}|^2 \dr \m_{\C}\big\vert_{(t,r^*,\omega)}$ by
\begin{align*}
&\sum_{\substack{n+|I| \leq 3 \\ n \leq 1}} \left| \int_{\Sigma_{\tau}} \rho \left[ r^{2zq+4} \frac{|v_{\uu}|^{zq+2}}{|v_t|^{zq+2}}  \left|\partial_t^n \widehat{\Omega}^I f \right|^z  |v_t|^{4-z} \big\langle |v_t|+| \slashed{v}| \big \rangle^{3(z-3+\delta)}  |v_t|^{2(z+1)} \right] \dr \m_{\Sigma_{\tau}} \right|^{\frac{1}{z}} \\ 
&+ \!\sum_{\substack{n+|I| \leq 3 \\ n \leq 1}}  \int_{\underline{\N}_{t+r^*\!}} \! \mathds{1}_{\tau \leq u-u_0 \leq \tau+1} \! \int_{\C}  r^{2zq+4} \frac{|v_{\uu}|^{zq+2}}{|v_t|^{zq+2}}  \left|\partial_t^n \widehat{\Omega}^I f \right|^z  |v_t|^{4-z} \big\langle |v_t|+| \slashed{v}| \big \rangle^{3(z-3+\delta)}  |v_t|^{2(z+1)} |v_u| \dr \m_{\C} \dr \m_{\underline{\N}_{t+r^*\!}}.
 \end{align*}
Then, we bound these two terms as follows.
 \begin{itemize}
 \item For the first one, note that for $\tau < 0$, $\Sigma_{\tau} = \N_{\tau}$ and $v \cdot \n_{\Sigma_{\tau}} = v_{\uu}$. It remains to apply Proposition \ref{energydecayexter} to $\big|\partial_t^n \widehat{\Omega}^I f \big|^z \langle |v_t|+| \slashed{v}|  \rangle^{3(z-3+\delta)}$, for $p =2zq+4$, $a=2(z+1)+4-z$ and, by an abuse of notation, $d=zd$.
 \item The second one can be controlled similarly, using Remark \ref{energydecayextRq} instead of Proposition \ref{energydecayexter}.
 \end{itemize}
\end{proof}

\section{Proof of Theorem \ref{theorem}}\label{sec7}

Let $f: \widehat{\Rm}_{-\infty}^{+\infty} \rightarrow \R$ be a solution to the massless Vlasov equation $\T(f)=0$. We introduce for any sufficiently regular function $h : \widehat{\Rm}_{-\infty}^{+\infty} \rightarrow \R$, $(a,q) \in \R_+^2$ and $s \geq 1$, the initial energy norms
\begin{align*}
\mathbf{E}^{a,q}_s[h] \, := \, & \int_{\mathcal{S}} \int_{\C} \Big( r^q \frac{|v_{\uu}|^{q/2}}{|v_t|^{q/2}}+(1+u_-)^q \Big) |h| |v^{\mathstrut}_{\mathrm{N}}|^a |v \cdot \n_{\mathcal{S}} | \dr \m_{\C} \dr \m_{\mathcal{S}} \\
& + \left| \int_{\mathcal{S}} \int_{\C} \Big( r^q \frac{|v_{\uu}|^{q/2}}{|v_t|^{q/2}}+(1+u_-)^q \Big) |h|^{s^{\lceil q \rceil}} \langle v_t \rangle^{(4+a)(s^{\lceil q \rceil} -1)} |v^{\mathstrut}_{\mathrm{N}}|^a |v \cdot \n_{\mathcal{S}} | \dr \m_{\C} \dr \m_{\mathcal{S}} \right|^{s^{-\lceil q \rceil}}.
\end{align*}
The next lemma is a direct consequence of Propositions \ref{energydecaysec4} and \ref{energydecayexter}.
\begin{Lem}\label{LemforproofTheorem}
Let $(a,q) \in \R_+ \times \R_+$ and $s>1$ such that $s^2 \leq 1+\langle 5p \rangle^{-2}$. Consider further $h : \widehat{\Rm}_{-\infty}^{+\infty} \rightarrow \R$ a solution to $\T(h)=0$ satisfying $\mathbf{E}^{a,q}_s[h] < +\infty$. Then, for any $d \in [0,q]$,
\begin{equation}\label{eq:pevarugent}
 \forall \, \tau \in \R, \qquad \qquad \int_{\Sigma_{\tau}} \rho \Bigg[ r^{q-d} \frac{|v_{\uu}|^{\frac{q-d}{2}}}{|v_t|^{\frac{q-d}{2}}} |h| |v^{\mathstrut}_{\mathrm{N}}|^a \Bigg] \dr \m_{\Sigma_{\tau}} \, \lesssim^{\mathstrut}_{a,q,s} \,   \frac{\mathbf{E}^{a,q}_s[f]}{(1+|\tau|)^d} . 
\end{equation}
\end{Lem}
\begin{proof}
We will use all along the proof that, in view of Lemma \ref{Comuprop}, $|h|^{s^{ \lceil q \rceil}} \langle v_t \rangle^{(4+a)(s^{\lceil q \rceil} -1)} $ is a solution to the massless Vlasov equation. Note first that
\begin{equation}\label{eq:unhapevarugent}
 \forall \, d \in [0,q], \qquad r^{q-d} \frac{|v_{\uu}|^{\frac{q-d}{2}}}{|v_t|^{\frac{q-d}{2}}}(1+u_-)^d \, \leq \, r^{q} \frac{|v_{\uu}|^{\frac{q}{2}}}{|v_t|^{\frac{q}{2}}}+(1+u_-)^q.
 \end{equation}
As $\Sigma_{\tau}= \N_{\tau}$, $v \cdot \n_{\N_{\tau}} = v_{\uu}$ and $v_t = v^{\mathstrut}_{\mathrm{N}}$ for $\tau <0$, we obtain by applying Proposition \ref{energydecayexter} that \eqref{eq:pevarugent} holds for all $\tau \leq 0$. Moreover, since $\Sigma_0= \big(\mathcal{S}\cap \{r < R_0 \} \big) \sqcup \N_0$, Proposition \ref{energydecayexter}, applied to $h$ and $|h|^{s^{ \lceil q \rceil}} \langle v_t \rangle^{(4+a)(s^{\lceil q \rceil} -1)} $ with $d \in \{ 0,q \}$, also provides us
$$ \int_{\Sigma_0} \! \rho \bigg[ \Big( 1+r^{q} \frac{|v_{\uu}|^{q/2}}{|v_t|^{q/2}} \Big)|h| |v^{\mathstrut}_{\mathrm{N}}|^a \bigg] \dr \m_{\Sigma_{0}}+ \left| \int_{\Sigma_0} \rho \bigg[ \Big( 1+r^{q} \frac{|v_{\uu}|^{q/2}}{|v_t|^{q/2}} \Big)|h|^{s^{ \lceil q \rceil}} \! \langle v_t \rangle^{(4+a)(s^{\lceil q \rceil} -1)}  |v^{\mathstrut}_{\mathrm{N}}|^a \bigg] \dr \m_{\Sigma_{0}} \right|^{s^{-\lceil q \rceil}} \! \lesssim^{\mathstrut}_{q} \! \mathbf{E}^{a,q}_s[h].  $$
According to Proposition \ref{energydecaysec4} and Remark \ref{energydecaysec4Rq}, we then obtain that \eqref{eq:pevarugent} also holds for all $\tau \geq 0$.
\end{proof}
We now start the proof of Theorem \ref{theorem}. Let $s>1$ and remark that we have $\T (|f|^s \langle v_t \rangle^{5(s-1)} ))=0$ by Lemma \ref{Comuprop}. Hence, as $\Sigma_0=\big( \mathcal{S} \cap \{ r < R_0 \} \big) \sqcup \N_0$ and $v \cdot \n_{\N_0}=v_{\uu}$, we obtain from Proposition \ref{energydecayexter} that
\begin{align*}
& \int_{\Sigma_0} \rho \Big[ |f| |v^{\mathstrut}_{\mathrm{N}}| \Big] \dr \m_{\Sigma_0}+ \left| \int_{\Sigma_{0}}  \rho \left[\left(1+ r^{2(s-1)}\frac{|v_{\underline{u}}|}{|v_t|} \right) \langle v_t \rangle^{5(s-1)}  |f|^s  |v^{\mathstrut}_{\mathrm{N}}| \right] \dr \m_{\Sigma_{0}} \right|^{\frac{1}{s}} \\
& \qquad \qquad  \lesssim_s \int_{\mathcal{S}} \int_{\C}  |f| |v^{\mathstrut}_{\mathrm{N}}| |v \cdot \n_{\mathcal{S}} | \dr \m_{\C} \dr \m_{\mathcal{S}}  + \left| \int_{\mathcal{S}} \int_{\C} \left(1+ r^{2(s-1)} \frac{|v_{\uu}|^{s-1}}{|v_t|^{s-1}} \right) \langle v_t \rangle^{5(s -1)} |f|^s |v^{\mathstrut}_{\mathrm{N}}| |v \cdot \n_{\mathcal{S}} | \dr \m_{\C} \dr \m_{\mathcal{S}} \right|^{\frac{1}{s}} \! .
\end{align*}
Then, the integrated local energy decay estimate follows from Proposition \ref{ILED}.

For $p \in \R_+$, the boundedness of the $r$-weighted energy norm
$$  \int_{\Sigma_{\tau}} \int_{\C} r^{p} \frac{| v_{\uu}|^{p/2}}{| v_t|^{p/2}}|f| |v^{\mathstrut}_{\mathrm{N}}| |v \cdot \n_{\Sigma_{\tau}} | \dr \m_{\C} \dr \m_{\Sigma_{\tau}} \, \lesssim^{\mathstrut}_p \, \int_{\mathcal{S}} \int_{\C} \bigg(1+ r^{p} \frac{|v_{\uu}|^{p/2}}{|v_t|^{p/2}} \bigg)  |f| |v^{\mathstrut}_{\mathrm{N}}| |v \cdot \n_{\mathcal{S}} | \dr \m_{\C} \dr \m_{\mathcal{S}}$$
is a direct consequence of Propositions \ref{energydecayexter}, \ref{energydecaysec4} and that $v^{\mathstrut}_{\mathrm{N}}=v_t$ for all $r \geq R_0$.

Assume now that $s>1$ satisfies $s^2 \leq 1+\langle 5 p \rangle^{-2}$. If $\mathbf{E}_s^{1,p}[f] < + \infty$, we obtain from Lemma \ref{LemforproofTheorem} the following decay estimate for the energy flux
$$ \forall \, \tau \in \R, \qquad  \int_{\Sigma_{\tau}} \rho \Big[  |f| |v^{\mathstrut}_{\mathrm{N}}| \Big] \dr \m_{\Sigma_{\tau}} \, \lesssim^{\mathstrut}_{p,s}  \frac{\mathbf{E}^{1,p}_s[f]}{(1+|\tau|)^p}.$$
It remains to prove the pointwise decay estimates. We suppose now that $1<s^2 \leq 1+\langle 20(p+2)\rangle^{-2}$. The main idea consists in proving decay estimates for well-chosen energy norms of $f$ and then to apply Propositions \ref{pointwise} and \ref{pointwiseext}. In order to reduce the number of parameters, we will apply these last two results with $z=2+s$ and $\delta=s-1$ but these restrictions are not necessary. Introduce then the energy norm
\begin{equation}\label{mathcalE}
 \mathcal{E}^p_s[f] \, := \,  \sum_{n \leq 1} \, \sum_{n+|I| \leq 3}  \left| \mathbb{E}^{8+s, \,(2+s)p+4}_s \left[ \big| \partial_t^n \widehat{\Omega}^I f \big|^{2+s} \big \langle |v_t|+|\slashed{v}| \big \rangle^{6(s-1) } \right] \right|^{\frac{1}{2+s}} 
 \end{equation}
and assume that $\mathcal{E}_s^p[f] < +\infty$. Start by noticing that, according to Lemma \ref{Comuprop}, we have 
$$ \forall \, n \leq 1, \; |I| \leq 3-n, \qquad  \T \big( \, \big| \partial_t^n \widehat{\Omega}^I f \big|^{2+s} \big \langle |v_t|+|\slashed{v}| \big \rangle^{6(s-1) } \, \big)=0.$$ 
Now, we prove 
\begin{equation}\label{eq:daoumil}
\forall \, \tau \in \R, \; (t,r^*,\omega) \in \Sigma_{\tau},  \quad \qquad \int_{\C} |f| |v^{\mathstrut}_{\mathrm{N}}|^2 \dr \m_{\C} \Big\vert_{(t,r^*,\omega)} \lesssim_{p,s} \frac{\mathcal{E}^p_s[f]}{r^2(1+|\tau|)^p} .
\end{equation}
\begin{itemize}
\item Indeed, if $(t,r^*,\omega) \in \Sigma_{\tau}$ with $\tau <0$, this directly follows from Proposition \ref{pointwiseext}, applied with $z=2+s$, $\delta = s-1$, $q = 0$ and $d=(2+s)p$, together with \eqref{eq:unhapevarugent}, applied with $q=(2+s)p+4$ and $d=(2+s)p$.
\item Otherwise $(t,r^*,\omega) \in \Rm_{0}^{+\infty}$ and use first Proposition \ref{pointwise}, with $z=2+s$, $\delta = s-1$ and $q = 0$. It then remains to bound $|v_t|^{2-s}|v^{\mathstrut}_{\mathrm{N}}|^{2s+6}$ by $|v^{\mathstrut}_{\mathrm{N}}|^{8+s}$ and to apply Lemma \ref{LemforproofTheorem} to $\big| \partial_t^n \widehat{\Omega}^I f \big|^{2+s} \big \langle |v_t|+|\slashed{v}| \big\rangle^{6(s-1)}$, with $a=8+s$, $q=(2+s)p+4$ and $d=(2+s)p$.
 \end{itemize}
Similarly, we have
$$\forall \, \tau \in \R, \; (t,r^*,\omega) \in \Sigma_{\tau},  \quad \qquad \int_{\C} \frac{|v_{\uu}|^{p/2}}{|v_t|^{p/2}} |f| |v^{\mathstrut}_{\mathrm{N}}|^2 \dr \m_{\C} \Big\vert_{(t,r^*,\omega)} \lesssim_{p,s} \frac{\mathcal{E}_s^p[f]}{r^2(r+|\tau|)^p} .$$
As $|v_{\uu}| \leq |v_t|$ and since \eqref{eq:daoumil} holds, it remains to bound the left hand side by $r^{-2-p} \mathcal{E}_s^p[f]$. For this,
\begin{itemize}
\item If $(t,r^*,\omega) \in \Sigma_{\tau}$, with $\tau <0$, we apply Proposition \ref{pointwiseext} for $z=2+s$, $\delta = s-1$, $q = p/2$ and $d=0$.
\item Otherwise $(t,r^*,\omega) \in \Rm_0^{+\infty}$ and use first Proposition \ref{pointwise}, with $z=2+s$, $\delta = s-1$ and $q = p/2$. Then, apply Lemma \ref{LemforproofTheorem} to $\big| \partial_t^n \widehat{\Omega}^I f \big|^{2+s} \big \langle |v_t|+|\slashed{v}| \big\rangle^{6(s-1)}$, with $a=8+s$, $q=(2+s)p+4$ and $d=0$.
 \end{itemize}
As $1+|t+r^*| \lesssim r+|\tau|$ by Lemma \ref{vol}, this concludes the proof of Theorem \ref{theorem}.

\subsection*{Acknowledgements} 
I would like to thank Dejan Gajic and Mihalis Dafermos for insightful discussions. This material is based upon work supported by the Swedish Research Council under grant no. 2016-06596 while I was in residence at Institut Mittag-Leffler in Djursholm, Sweden during the fall semester 2019. I also acknowledge the support of partial funding by the ERC grant MAFRAN 2017-2022.

\appendix

\section{Basic properties of the foliation $(\Sigma_{\tau})_{\tau \in \R}$}\label{ape}

The purpose of this section is to prove Lemma \ref{vol}. As the Regge-Wheeler coordinates degenerates at the horizon, it will be convenient to use the coordinate system $(\underline{u} , r , \theta , \varphi) \in \R \times \R_+^* \times ]0,\pi[ \times ]0,2 \pi [$, which covers the region $\mathscr{B} \cup \mathcal{H}^+ \cup \mathscr{D}$ (see Figure \ref{Penrose1}) and which is then regular on the event horizon $\mathcal{H}^+$, where $r=2M$. Recall that the metric takes the following form
$$ g \, = \, -\left( 1 - \frac{2M}{r} \right) \dr \underline{u}^2+2 \dr \underline{u} \dr r+r^2 \dr \m_{\mathbb{S}^2}$$
and that, denoting by $\partial_{\underline{u}}'$ and $\partial_{r}'$ the differentiation with respect to $\underline{u}$ and $r$ in the coordinate system $(\underline{u} , r , \theta , \varphi)$, we have in $\mathscr{D}$,
\begin{equation}\label{linkderiv}
\partial_t \, = \,  \partial_{\underline{u}}'  ,  \qquad \partial_{\underline{u}} \, = \, \partial_{\underline{u}}'+\frac{1}{2}\left( 1 - \frac{2M}{r} \right) \partial_r', \qquad \partial_u \, = \, -\frac{1}{2}\left( 1 - \frac{2M}{r} \right) \partial_r' .
\end{equation}
In particular, $ \frac{ 1}{1-\frac{2M}{r}} \partial_u$ can be extended as a smooth vector field on $\mathscr{B} \cup \mathcal{H}^+ \cup \mathscr{D}$. The invariant volume element $\dr \m_{\mathcal{M}}$ induced by $g$ on $\mathcal{M}$ can be expressed, in the region $\Rm_{-\infty}^{+\infty} \subset \mathscr{D}$, in the following three different ways
$$\dr \m_{\Rm^{+\infty}_{-\infty}} \,:= \, \dr \m_{\mathcal{M}} \Big\vert_{\Rm_{-\infty}^{+\infty}} = \, \left( 1 -\frac{2M}{r} \right)r^2 \dr t \wedge \dr r^* \wedge \dr \m_{\mathbb{S}^2} = r^2 \dr \underline{u} \wedge \dr r \wedge \dr \m_{\mathbb{S}^2} \, = \, \left( 1 -\frac{2M}{r} \right)\frac{r^2}{2} \dr u \wedge \dr \underline{u} \wedge \dr \m_{\mathbb{S}^2}, $$
where $ \dr \m_{\mathbb{S}^2}$ is the standard volume form on the unit sphere $\mathbb{S}^2$. As the hypersurfaces $\N_{\tau}$ are null, there is no canonical choice of normal vector $\n_{\N_{\tau}}$. Since
$$ \dr \m_{\Rm_{-\infty}^{+\infty}} \, = \, r^2g(\partial_{\uu} , \cdot ) \wedge \dr \uu \wedge \dr \mu^{\mathstrut}_{\mathbb{S}^2},$$
we choose $\n_{\N_{\tau}}:=\partial_{\uu}$ and the induced volume form on $\N_{\tau}$ is $\dr \m_{\N_{\tau}} := r^2 \dr \uu \wedge \dr \mu^{\mathstrut}_{\mathbb{S}^2}$. As $\tau=u-u_0$ for $r \geq R_0$ and $\Sigma_{\tau} \cap \{ r\geq R_0 \} =\N_{\tau}$, we have
$$\forall \, r \geq R_0, \qquad |\tau-u|=|u_0| \qquad \text{and} \qquad \dr \m_{\Rm^{+\infty}_{-\infty}} \, = \, \gamma_0(r) \dr \tau \wedge \dr \m_{\Sigma_{\tau}}, \quad \gamma_0(r) :=\frac{1}{2} \left( 1- \frac{2M}{r} \right).$$
We have then obtained all the results of Lemma \ref{vol} which concern the region $r \geq R_0$ and we now focus on the domain $2M < r < R_0$, where $\Sigma_{\tau}$ is spacelike. We start by proving the following result.
\begin{Lem}
There exists a smooth function $\underline{U} : [2M, R_0 ] \rightarrow \R$ such that, for any $\tau \geq 0$, $\Sigma_{\tau} \cap \{ r < R_0 \}$ can be parameterized, in the system of coordinates $(\underline{u},r,\theta, \varphi)$, by
$$ (r,\omega ) \in \, ]2M,R_0[ \times \mathbb{S}^2 \mapsto ( \underline{U}(r)+\tau , r, \omega).$$
Moreover, the vector field $\vec{n}:= g^{-1} ( \dr \uu - \underline{U}'(r) \dr r, \cdot )$ is normal to $\Sigma_{\tau} \cap \{ r < R_0 \}$ and timelike in the region $\{ 2M \leq r \leq R_0 \}$.
\end{Lem} 
 \begin{proof}
It will be convenient to work here with the coordinate system $(\underline{u},r,\theta , \varphi)$. Since $\Sigma_{\tau} \cap \{ r < R_0 \} = \varphi_{\tau} \left( \Sigma_{0} \cap \{ r < R_0 \} \right)$, where $\varphi_{\tau}$ is the flow generated by the Killing vector field $\partial_t=\partial_{\uu}'$, it suffices to prove the result for $\tau=0$. Then, recall that by construction, $\Sigma_0 \cap \{ 2M< r < R_0 \}= \accentset{\circ}{\mathcal{S}} \cap \{ 2M < r < R_0 \}$, where $\accentset{\circ}{\mathcal{S}}$ is a spacelike hypersurface crossing $\mathcal{H}^+$ to the future of the bifurcation sphere. Hence, there exists $\delta < 0$ such that $\widetilde{\mathcal{S}}:=\accentset{\circ}{\mathcal{S}} \cap \{ 2M - \delta < r \} \subset \mathscr{B} \cup \mathcal{H}^+ \cup \mathscr{D}$. As $\widetilde{\mathcal{S}}$ is  a smooth spherically symmetric hypersurface, for any $x \in \widetilde{\mathcal{S}}$, there exist an open set $O$ containing $x$ and a smooth function $F_O$ such that 
 $$ \widetilde{\mathcal{S}} \cap O  \, = \, \left\{ (\underline{u},r , \omega ) \in \R \times \R_+^* \times \mathbb{S}^2 \, / \, F_O(\underline{u},r) \, = \, 0 \right\}.$$
 Since $\widetilde{\mathcal{S}}$ is spacelike, $\partial_t=\partial_{\underline{u}}' \notin T_x \widetilde{\mathcal{S}}$ for any $x$ in the region $r \geq 2M$, so that $\partial'_{\underline{u}} F_O$ does not vanish on $\widetilde{\mathcal{S}} \cap O \cap \{ r \geq 2M \}$. According to the implicit function theorem, we obtain that for any $x \in \widetilde{\mathcal{S}} \cap \{ r \geq 2M \}$, there exist an open set $O'$ containing $x$ and a smooth function $\underline{U}_{O'}$ such that 
 $$ \widetilde{\mathcal{S}} \cap O'  \, = \, \left\{ (\underline{u},r , \omega ) \in \R \times \R_+^* \times \mathbb{S}^2 \, / \, \underline{u}=\underline{U}_{O'}(r) \right\}.$$
 By a connexity argument, this implies, if $\delta >0$ is chosen small enough, that there exists a smooth function $\underline{U} : \, ]2M-\delta,+\infty[ \rightarrow \R$ such that
 $$ \widetilde{\mathcal{S}}  \, = \, \left\{ (\underline{u},r , \omega ) \in \R \times ]2M-\delta , +\infty [ \times \mathbb{S}^2 \, / \, \underline{u}=\underline{U}(r) \right\}.$$
We then deduce that the $1$-form $\dr \uu - \underline{U}'(r) \dr r$ is normal to the spacelike hypersurface $\widetilde{\mathcal{S}}$. This implies that the $\varphi_{\tau}$-invariant vector field $\vec{n}$ is normal to $\Sigma_{\tau} \cap \{ r^* < R_0^* \}$ and timelike on $\mathcal{H}^+ \cup \mathscr{D}$.
 \end{proof}
As we have $t=\underline{u} -r^*$ in the region $\mathscr{D}$, this implies that the parameterization of $\Sigma_{\tau} \cap \{ r < R_0 \}$ in the Regge-Wheeler coordinates given in Lemma \ref{vol} holds. In order to prove the remaining four properties, notice first that, as $\vec{n}$ is $\varphi_{\tau}$-invariant, $SO_3(\R)$-invariant and timelike on $\{ 2M \leq r \leq R_0 \}$, we can define the smooth function
$$ \gamma : [2M , R_0] \rightarrow \R_+^*, \quad \gamma(r) = |g(\vec{n},\vec{n})|^{\frac{1}{2}}, \qquad \text{ which satisfies} \quad \exists \, C \geq 1, \; \forall \, 2M \leq r \leq R_0, \quad \frac{1}{C} \leq \frac{1}{\gamma(r)} \leq C.$$
Note now that the future oriented normal vector along $\Sigma_{\tau} \cap \{ r < R_0\}$ is equal either to $\frac{\vec{n}}{\gamma(r)}$ or to $-\frac{\vec{n}}{\gamma(r)}$. So, as the induced volume form on $\Sigma_{\tau}$ satisfies
$$ \dr \m_{\Rm^{+\infty}_{-\infty}} \, = \, -g ( \n_{\Sigma_{\tau}} , \cdot ) \wedge \dr \m_{\Sigma_{\tau}} , \qquad \text{we have} \quad \dr \m_{\Sigma_{\tau}} \Big\vert_{\{2M < r < R_0\}} \, = \, \gamma(r) r^2 \dr r \wedge \dr \m_{\mathbb{S}^2}.$$
Since $ \tau =  \uu -\underline{U}(r)$ for $2M < r < R_0$, we have
$$ \forall \, 2M < r < R_0, \qquad |\tau - \uu| \leq \|\underline{U}\|^{\mathstrut}_{L^{\infty}} \qquad \text{and} \qquad  \dr \m_{\Rm^{+\infty}_{-\infty}} \, = \, \gamma_0(r) \dr \tau \wedge \dr \m_{\Sigma_{\tau}}, \quad \gamma_0(r) = \frac{1}{\gamma(r)}.  $$
Finally, in view of the properties of $\vec{n}$, there exist smooth functions $\xi, \, \zeta :[2M,R_0] \rightarrow \R$ such that $\n_{\Sigma_{\tau}}= \xi(r)\partial_{\uu}'+\zeta(r) \partial_r'$. We then deduce from \eqref{linkderiv} that there exist smooth functions $\alpha, \, \beta :[2M,R_0] \rightarrow \R$ satisfying $\n_{\Sigma_{\tau}}= \alpha(r)\partial_{\uu}+\frac{\beta(r)}{1-\frac{2M}{r}} \partial_u$. As $\n_{\Sigma_{\tau}}$ is unitary and timelike, we have using \eqref{metricuuu} that $\alpha(r) \beta(r)=1$ for all $2M < r < R_0$. By continuity, the relation holds for all $r \in [2M,R_0]$ and this implies, since $\n_{\Sigma_{\tau}}$ is future oriented, that $\beta$ and $\alpha=1/\beta$ are both strictly positives on $[2M,R_0]$. This concludes the proof of Lemma \ref{vol}.
\section{Controlling all the components of the energy-momentum tensor}

We prove here a general result which motivates the introduction of the red-shift vector field $\mathrm{N}$. More precisely, we prove that if $T$ and $\widetilde{T}$ are strictly timelike vector fields, then, on any compact set, $\mathbb{T}[f](T,\widetilde{T})$ controls uniformly all the components of the energy-momentum tensor $\mathbb{T}[f]$ of the Vlasov field $f$.

Let $(\mathcal{M},g)$ be a smooth time-oriented and oriented $4$-dimensional Lorentzian manifold and consider the bundle of future light cones
$$ \mathcal{P} \, := \, \bigcup_{x \in \mathcal{M}} \mathcal{P}_x, \qquad \mathcal{P}_x\, := \, \{ (x,v) \, / \, v  \in T^{\star}_x \mathcal{M} , \quad g^{-1}_x(v,v)=0, \quad \text{$v$ future oriented} \}.$$
Given a coordinate system $(U,x^0,x^1,x^2, x^3)$ on $\mathcal{M}$, for any $y \in U \subset \mathcal{M}$, we can decompose any $v \in T^{\star}_y \mathcal{M}$ as $v= v_{\alpha} \dr {x^{\alpha}}\vert_{ y}$. Consequently, $(x^{\alpha},v_{\alpha})$ is a coordinate system on $T^{\star} \mathcal{M}$, called conjugates to $(x^0,x^1,x^2 , x^3)$. The metric $g$ induces the invariant volume element $\dr \m_{T^{\star}_x \mathcal{M}}= |\det g^{-1}_x|^{\frac{1}{2}} \dr v_0 \wedge \dr v_1 \wedge \dr v_2 \wedge \dr v_3$ on $T^{\star}_x \mathcal{M}$. It induces a volume form on $\mathcal{P}_x$, satisfying $\dr \m_{\mathcal{P}_x} = \dr q \wedge \dr \m_{T^{\star}_x \mathcal{M}}$ with $q(v):= \frac{1}{2}g^{-1}_x(v,v)$. We can then define the energy-momentum tensor of any sufficiently regular function $f : \mathcal{P} \rightarrow \R$ by
\begin{equation}\label{defTen}
 \mathbb{T}[f]_{\alpha \beta} \, := \, \int_{\mathcal{P}} f \, v_{\alpha} v_{\beta} \dr \m_{\mathcal{P}}, \qquad \text{where} \quad \int_{\mathcal{P}} f \, v_{\alpha} v_{\beta} \dr \m_{\mathcal{P}} := x \mapsto \int_{\mathcal{P}_x} f v_{\alpha} v_{\beta} \dr \m_{\mathcal{P}_x}.
\end{equation} If $x^0$ is a temporal function, then $(x^0,x^1,x^2 , x^3,v_1,v_2,v_3)$ are smooth coordinates on $\mathcal{P}$ since we have
$$v_0 =-\frac{1}{g^{00}} \left( g^{0j}v_j- \sqrt{ (g^{0j}v_j)^2-g^{00}g^{ij} v_iv_j }\right). \qquad \text{Moreover,} \quad \dr \m_{\mathcal{P}_x} = \frac{ |\det g^{-1}_x|^{\frac{1}{2}}}{ v_{\alpha} g^{\alpha 0}} \dr v_1 \wedge \dr v_2 \wedge \dr v_3.$$
\begin{Lem}\label{Lemannex}
Consider a smooth nonnegative function $f : \mathcal{P} \rightarrow \R_+$ and a compact subset $\mathcal{K} \subset \mathcal{M}$. Let $X, \,\widetilde{X}, \, T, \, \widetilde{T}$ be four smooth vector fields such that $T$ and $\widetilde{T}$ are strictly timelike and future oriented on $\mathcal{K}$. Then, there exists $D >0$ such that
$$ \forall \, x \in \mathcal{K}, \; \forall \, v \in \mathcal{P}_x, \qquad |v(X)| \, \leq \, - D \, v(T).$$
Moreover, there exists $C>0$ such that $|\mathbb{T}|f](X,\widetilde{X})| \leq C \cdot \mathbb{T}[f](T,\widetilde{T})$ holds uniformly on $\mathcal{K}$.
\end{Lem}
\begin{proof}
We start by the first estimate. Let $x \in \mathcal{K}$ and consider a coordinate system $(x^0,\dots x^3)$, defined on an open set $U_x$ containing $x$, such that $x^0$ is a smooth temporal function. Let further $O_x$ be an open set such that $x \in O_x$ and $\overline{O_x} \subset U_x$. If the result holds on $O_x \cap \mathcal{K}$ for any $x \in \mathcal{K}$, then it holds on the compact $\mathcal{K}$ since it can be covered by a finite number of such open sets $O_x$. Consequently, it suffices to prove the result on $O_x \cap \mathcal{K}$. Introduce now the compact set
$$ \mathcal{Q} \, := \, \{ (y,v) \, / \, y \in \overline{O_x} \cap \mathcal{K}, \quad v \in \mathcal{P}_y, \quad |v|^2 := |v_1|^2+|v_2|^2+|v_3|^2=1 \}.$$
Note that $v(T) <0$ on $\mathcal{K}$ for any $v \in \mathcal{P}$. Indeed, $v$ is causal, $T$ is strictly timelike on $\mathcal{K}$ and they are both future oriented. Consequently, we have $S:=\sup_{\mathcal{Q}} |v(X)| < +\infty$, $I:=\inf_{\mathcal{Q}} -v(T) >0$ and the first inequality of the lemma holds on $O_x \cap \mathcal{K}$, with $D = \frac{S}{I}$. Together with the definition \eqref{defTen} of $\mathbb{T}[f]$, this directly  implies the second estimate.
\end{proof}

\renewcommand{\refname}{References}
\bibliographystyle{abbrv}
\bibliography{biblio}

\begin{thebibliography}{10}

\bibitem{AB}
L.~Andersson and P.~Blue.
\newblock Hidden symmetries and decay for the wave equation on the {K}err
  spacetime.
\newblock {\em Ann. of Math. (2)}, 182(3):787--853, 2015.

\bibitem{ABJ}
L.~Andersson, P.~Blue, and J.~Joudioux.
\newblock Hidden symmetries and decay for the {V}lasov equation on the {K}err
  spacetime.
\newblock {\em Comm. Partial Differential Equations}, 43(1):47--65, 2018.

\bibitem{ABBM}
L.~Andersson, T.~Bäckdahl, P.~Blue, and S.~Ma.
\newblock {Stability for linearized gravity on the Kerr spacetime}.
\newblock arXiv:1903.03859, 2019.

\bibitem{Dejanadv}
Y.~Angelopoulos, S.~Aretakis, and D.~Gajic.
\newblock Late-time asymptotics for the wave equation on spherically symmetric,
  stationary spacetimes.
\newblock {\em Adv. Math.}, 323:529--621, 2018.

\bibitem{Dejanann}
Y.~Angelopoulos, S.~Aretakis, and D.~Gajic.
\newblock A vector field approach to almost-sharp decay for the wave equation
  on spherically symmetric, stationary spacetimes.
\newblock {\em Ann. PDE}, 4(2):Paper No. 15, 120, 2018.

\bibitem{aretakis1}
S.~Aretakis.
\newblock Stability and instability of extreme {R}eissner-{N}ordstr\"{o}m black
  hole spacetimes for linear scalar perturbations {I}.
\newblock {\em Comm. Math. Phys.}, 307(1):17--63, 2011.

\bibitem{aretakis2}
S.~Aretakis.
\newblock Stability and instability of extreme {R}eissner-{N}ordstr\"{o}m black
  hole spacetimes for linear scalar perturbations {II}.
\newblock {\em Ann. Henri Poincar\'{e}}, 12(8):1491--1538, 2011.

\bibitem{rVP}
L.~Bigorgne.
\newblock A vector field method for massless relativistic transport equations
  and applications.
\newblock {\em Journal of Functional Analysis}, 278(4):108365, 2020.

\bibitem{BFJS}
L.~Bigorgne, D.~Fajman, J.~Joudioux, J.~Smulevici, and M.~Thaller.
\newblock Asymptotic stability of {M}inkowski space-time with non-compactly
  supported massless {V}lasov matter.
\newblock {\em Arch. Ration. Mech. Anal.}, 242(1):1--147, 2021.

\bibitem{BlueSoffer}
P.~Blue and A.~Soffer.
\newblock Phase space analysis on some black hole manifolds.
\newblock {\em J. Funct. Anal.}, 256(1):1--90, 2009.

\bibitem{BlueSter}
P.~Blue and J.~Sterbenz.
\newblock Uniform decay of local energy and the semi-linear wave equation on
  {S}chwarzschild space.
\newblock {\em Comm. Math. Phys.}, 268(2):481--504, 2006.

\bibitem{CKM}
D.~Christodoulou and S.~Klainerman.
\newblock {\em The global nonlinear stability of the {M}inkowski space},
  volume~41 of {\em Princeton Mathematical Series}.
\newblock Princeton University Press, Princeton, NJ, 1993.

\bibitem{DHRTeuko}
M.~Dafermos, G.~Holzegel, and I.~Rodnianski.
\newblock Boundedness and decay for the {T}eukolsky equation on {K}err
  spacetimes {I}: {T}he case {$|a|\ll M$}.
\newblock {\em Ann. PDE}, 5(1):Paper No. 2, 118, 2019.

\bibitem{DHRlinstab}
M.~Dafermos, G.~Holzegel, and I.~Rodnianski.
\newblock The linear stability of the {S}chwarzschild solution to gravitational
  perturbations.
\newblock {\em Acta Math.}, 222(1):1--214, 2019.

\bibitem{redshift}
M.~Dafermos and I.~Rodnianski.
\newblock The red-shift effect and radiation decay on black hole spacetimes.
\newblock {\em Comm. Pure Appl. Math.}, 62(7):859--919, 2009.

\bibitem{rpDR}
M.~Dafermos and I.~Rodnianski.
\newblock A new physical-space approach to decay for the wave equation with
  applications to black hole spacetimes.
\newblock In {\em X{VI}th {I}nternational {C}ongress on {M}athematical
  {P}hysics}, pages 421--432. World Sci. Publ., Hackensack, NJ, 2010.

\bibitem{LecBH}
M.~Dafermos and I.~Rodnianski.
\newblock Lectures on black holes and linear waves.
\newblock In {\em Evolution equations}, volume~17 of {\em Clay Math. Proc.},
  pages 97--205. Amer. Math. Soc., Providence, RI, 2013.

\bibitem{Kerrdecay}
M.~Dafermos, I.~Rodnianski, and Y.~Shlapentokh-Rothman.
\newblock Decay for solutions of the wave equation on {K}err exterior
  spacetimes {III}: {T}he full subextremal case {$|a|<M$}.
\newblock {\em Ann. of Math. (2)}, 183(3):787--913, 2016.

\bibitem{FJS}
D.~Fajman, J.~Joudioux, and J.~Smulevici.
\newblock A vector field method for relativistic transport equations with
  applications.
\newblock {\em Anal. PDE}, 10(7):1539--1612, 2017.

\bibitem{FJS3}
D.~Fajman, J.~Joudioux, and J.~Smulevici.
\newblock The stability of the {M}inkowski space for the {E}instein-{V}lasov
  system.
\newblock {\em Anal. PDE}, 14(2):425--531, 2021.

\bibitem{GiorgiTeuko2}
E.~Giorgi.
\newblock Boundedness and decay for the {T}eukolsky equation of spin {$\pm1$}
  on {R}eissner-{N}ordstr\"{o}m spacetime: the {$\ell=1$} spherical mode.
\newblock {\em Classical Quantum Gravity}, 36(20):205001, 48, 2019.

\bibitem{ThomasJohnson}
T.~W. Johnson.
\newblock The linear stability of the {S}chwarzschild solution to gravitational
  perturbations in the generalised wave gauge.
\newblock {\em Ann. PDE}, 5(2):Paper No. 13, 92, 2019.

\bibitem{JosephKeir}
J.~Keir.
\newblock {The weak null condition and global existence using the p-weighted
  energy method}.
\newblock arXiv:1808.09982, 2018.

\bibitem{Kl85}
S.~Klainerman.
\newblock Uniform decay estimates and the {L}orentz invariance of the classical
  wave equation.
\newblock {\em Comm. Pure Appl. Math.}, 38(3):321--332, 1985.

\bibitem{KlSz}
S.~Klainerman and J.~Szeftel.
\newblock {\em Global nonlinear stability of {S}chwarzschild spacetime under
  polarized perturbations}, volume 210 of {\em Annals of Mathematics Studies}.
\newblock Princeton University Press, Princeton, NJ, 2020.

\bibitem{LR10}
H.~Lindblad and I.~Rodnianski.
\newblock The global stability of {M}inkowski space-time in harmonic gauge.
\newblock {\em Ann. of Math. (2)}, 171(3):1401--1477, 2010.

\bibitem{Lindblad}
H.~Lindblad and M.~Taylor.
\newblock Global stability of {M}inkowski space for the {E}instein-{V}lasov
  system in the harmonic gauge.
\newblock {\em Arch. Ration. Mech. Anal.}, 235(1):517--633, 2020.

\bibitem{Luk1}
J.~Luk.
\newblock Improved decay for solutions to the linear wave equation on a
  {S}chwarzschild black hole.
\newblock {\em Ann. Henri Poincar\'{e}}, 11(5):805--880, 2010.

\bibitem{Luk2}
J.~Luk.
\newblock A vector field method approach to improved decay for solutions to the
  wave equation on a slowly rotating {K}err black hole.
\newblock {\em Anal. PDE}, 5(3):553--625, 2012.

\bibitem{Moschidisrp}
G.~Moschidis.
\newblock The {$r^p$}-weighted energy method of {D}afermos and {R}odnianski in
  general asymptotically flat spacetimes and applications.
\newblock {\em Ann. PDE}, 2(1):Art. 6, 194, 2016.

\bibitem{Pasqualottononlin}
F.~Pasqualotto.
\newblock Nonlinear stability for the {M}axwell-{B}orn-{I}nfeld system on a
  {S}chwarzschild background.
\newblock {\em Ann. PDE}, 5(2):Paper No. 19, 172, 2019.

\bibitem{PasqualottoTeuko}
F.~Pasqualotto.
\newblock The spin {$\pm1$} {T}eukolsky equations and the {M}axwell system on
  {S}chwarzschild.
\newblock {\em Ann. Henri Poincar\'{e}}, 20(4):1263--1323, 2019.

\bibitem{SarbachSchw}
P.~Rioseco and O.~Sarbach.
\newblock Accretion of a relativistic, collisionless kinetic gas into a
  {S}chwarzschild black hole.
\newblock {\em Classical Quantum Gravity}, 34(9):095007, 48, 2017.

\bibitem{Sarbach}
O.~Sarbach and T.~Zannias.
\newblock The geometry of the tangent bundle and the relativistic kinetic
  theory of gases.
\newblock {\em Classical and Quantum Gravity}, 31(8):085013, apr 2014.

\bibitem{Sbierski}
J.~Sbierski.
\newblock Characterisation of the energy of {G}aussian beams on {L}orentzian
  manifolds: with applications to black hole spacetimes.
\newblock {\em Anal. PDE}, 8(6):1379--1420, 2015.

\bibitem{Volker}
V.~Schlue.
\newblock Decay of linear waves on higher-dimensional {S}chwarzschild black
  holes.
\newblock {\em Anal. PDE}, 6(3):515--600, 2013.

\bibitem{Taylor}
M.~Taylor.
\newblock The global nonlinear stability of {M}inkowski space for the massless
  {E}instein-{V}lasov system.
\newblock {\em Ann. PDE}, 3(1):Art. 9, 177, 2017.

\bibitem{ShiwuYangmedia0}
S.~Yang.
\newblock Global solutions of nonlinear wave equations in time dependent
  inhomogeneous media.
\newblock {\em Arch. Ration. Mech. Anal.}, 209(2):683--728, 2013.

\bibitem{ShiwuYangwaveLarge}
S.~Yang.
\newblock Global solutions of nonlinear wave equations with large data.
\newblock {\em Selecta Math. (N.S.)}, 21(4):1405--1427, 2015.

\bibitem{ShiwuYangMKG2}
S.~Yang.
\newblock Decay of solutions of {M}axwell-{K}lein-{G}ordon equations with
  arbitrary {M}axwell field.
\newblock {\em Anal. PDE}, 9(8):1829--1902, 2016.

\bibitem{ShiwuYangMKG1}
S.~Yang.
\newblock On the global behavior of solutions of the {M}axwell-{K}lein-{G}ordon
  equations.
\newblock {\em Adv. Math.}, 326:490--520, 2018.

\bibitem{ShiwuYangsemi1}
S.~Yang.
\newblock Pointwise decay for semilinear wave equations in {$\Bbb{R}^{1 + 3}$}.
\newblock {\em J. Funct. Anal.}, 283(2):Paper No. 109486, 59, 2022.

\bibitem{ShiwuYangYuMKG}
S.~Yang and P.~Yu.
\newblock On global dynamics of the {M}axwell-{K}lein-{G}ordon equations.
\newblock {\em Camb. J. Math.}, 7(4):365--467, 2019.

\end{thebibliography}

\end{document}